\documentclass[10pt,reqno]{amsart}
\usepackage{amssymb,mathrsfs,color}

\usepackage{enumerate}

\usepackage{graphicx}
\usepackage{xcolor} % A package to add color.
\usepackage{tensor}
\usepackage{slashed}
\definecolor{green}{rgb}{0,0.8,0} % Redefines the color green.
\newcommand{\nrm}[1]{\Vert#1\Vert}

\newcommand{\set}[1]{\{#1\}}

\renewcommand{\Im}{\mathrm{Im}}

\newcommand{\aeq}{\sim}
\newcommand{\aleq}{\lesssim}

\newcommand{\ud}{\mathrm{d}}
\newcommand{\rd}{\partial}

\newcommand{\bb}{\Big}

\newcommand{\bt}{\beta}

\newcommand{\eps}{\epsilon}

\newcommand{\lmb}{\lambda}
\newcommand{\Lmb}{\Lambda}
\newcommand{\sgm}{\sigma}

\newcommand{\zt}{\zeta}

\newcommand{\bbC}{\mathbb C}
\newcommand{\bbD}{\mathbb D}

\newcommand{\bbH}{\mathbb H}

\newcommand{\bbR}{\mathbb R}
\newcommand{\bbS}{\mathbb S}

\newcommand{\calD}{\mathcal D}

\newcommand{\calL}{\mathcal L}

\newcommand{\arctanh}{\mathrm{arctanh} \, }
\newcommand{\vartht}{\vartheta}

\newcommand{\ztinf}{\zeta_{\infty}}
\newcommand{\ztz}{\zeta_{0}}
\newcommand{\gapE}{\mu^{2}_{\lmb}}
\newcommand{\rnL}{\widetilde{\calL}}

\usepackage{wrapfig}
\usepackage{tikz}
\usetikzlibrary{arrows,calc,decorations.pathreplacing}
\definecolor{light-gray1}{gray}{0.90}
\definecolor{light-gray2}{gray}{0.80}

\definecolor{deepgreen}{cmyk}{1,0,1,0.5}

\newcommand{\E}{\mathcal{E}}
\newcommand{\F}{\mathcal{F}}
\newcommand{\G}{\mathcal{G}}
\newcommand{\HH}{\mathcal{H}}

\newcommand{\LL}{\mathcal{L}}
\newcommand{\NN}{\mathcal{N}}

\newcommand{\KK}{\mathcal{K}}

\newcommand{\EE}{\mathscr{E}}

\newcommand{\C}{\mathbb{C}}
\newcommand{\Hp}{\mathbb{H}}

\newcommand{\R}{\mathbb{R}}
\newcommand{\Sp}{\mathbb{S}}
\newcommand{\Z}{\mathbb{Z}}
\newcommand{\D}{\mathbb{D}}

\newcommand{\g}{\mathbf{g}}
\newcommand{\h}{\mathbf{h}}

\newcommand{\al}{\alpha}
\newcommand{\be}{\beta}
\newcommand{\ga}{\gamma}
\newcommand{\de}{\delta}
\newcommand{\e}{\varepsilon}
\newcommand{\fy}{\varphi}
\newcommand{\om}{\omega}
\newcommand{\la}{\lambda}
\newcommand{\te}{\theta}
\newcommand{\s}{\sigma}

\newcommand{\z}{\zeta}

\newcommand{\La}{\Lambda}

\newcommand{\p}{\partial}

\newcommand{\re}{\mathop{\mathrm{Re}}}

\makeatletter

\newcommand{\Rmnum}[1]{\expandafter\@slowromancap\romannumeral #1@}
\makeatother

\newcommand{\I}{\infty}

\newcommand{\ti}{\widetilde}
\newcommand{\ba}{\overline}

\newcommand{\ang}[1]{\left\langle{#1}\right\rangle}
\newcommand{\abs}[1]{\left\lvert{#1}\right\rvert}

\newcommand{\ant}[1]{\begin{align*}\begin{split} #1 \end{split}\end{align*}}
\newcommand{\EQ}[1]{\begin{equation}\begin{split} #1 \end{split}\end{equation}}
\newcommand{\pmat}[1]{\begin{pmatrix} #1 \end{pmatrix}}

\newcommand{\Del}[1]{}

\def\ti{\tilde}

\numberwithin{equation}{section}

\newtheorem{thm}{Theorem}[section]

\newtheorem{lem}[thm]{Lemma}
\newtheorem{prop}[thm]{Proposition}

\newtheorem{claim}[thm]{Claim}
\theoremstyle{remark}
\newtheorem{rem}{Remark}
\newtheorem{defn}{Definition}

\newcommand{\mand}{{\ \ \text{and} \ \  }}

\newcommand{\mif}{{\ \ \text{if} \ \ }}
\newcommand{\mfor}{{\ \ \text{for} \ \ }}
\newcommand{\mas}{{\ \ \text{as} \ \ }}

\newcommand{\euc}{\textrm{euc}}

\renewcommand\Im{\mathrm{Im}\,}

\begin{document}

\title[Wave maps from the hyperbolic plane]{Stability of stationary equivariant wave maps from the hyperbolic plane}

\author{Andrew Lawrie}
\author{Sung-Jin Oh}
\author{Sohrab Shahshahani}

\begin{abstract} In this paper we initiate the study of  equivariant wave maps from $2d$ hyperbolic space, $\Hp^2$, into rotationally symmetric surfaces. This problem exhibits markedly different phenomena than its Euclidean counterpart due to the exponential volume growth of concentric geodesic spheres on the domain.

In particular, when the target is $\Sp^2$, we find a family of equivariant harmonic maps $\Hp^2 \to \Sp^2$, indexed by a parameter that measures how far the image of each harmonic map wraps around the sphere. These maps have energies taking all values between zero and the energy of the unique co-rotational Euclidean harmonic map,  $Q_{\euc}$, from $\R^2$ to $\Sp^2$, given by stereographic projection. We prove that the harmonic maps are asymptotically stable for values of the parameter smaller than a threshold that is large enough to allow for maps that wrap more than halfway around the sphere. Indeed, we prove Strichartz estimates for the operator obtained by linearizing around such a harmonic map. However, for  harmonic maps with energies approaching the Euclidean energy of $Q_{\euc}$, asymptotic stability via a perturbative argument based on Strichartz estimates is precluded by the existence of gap eigenvalues in the spectrum of the linearized operator. 
%However,  stability breaks down at the level of the underlying linear wave for maps with energy approaching the Euclidean energy of $Q$, which is well known to be %unstable.

When the target is $\Hp^2$, we find a continuous family of asymptotically stable equivariant harmonic maps $ \Hp^2 \to \Hp^2$ with arbitrarily small and arbitrarily large energies. This stands in sharp contrast to the corresponding problem on Euclidean space, where all finite energy solutions scatter to zero as time tends to infinity.

\end{abstract}

\thanks{Support of the National Science Foundation from grants DMS-1302782 and NSF 1045119 for the first and third authors, respectively, is gratefully acknowledged. The second author is a Miller Research Fellow, and acknowledges support from the Miller Institute.
The authors thank Joachim Krieger, Wilhelm Schlag and Daniel Tataru for many helpful discussions, and also Ji Hyun Bak for help with numerical experiments at the initial stage of investigation.}
\maketitle

\section{Introduction}\label{intro}

In recent years there has been increased interest in the study of dispersive equations on curved spaces. Here we begin the investigation of a simple model problem,  namely, equivariant wave maps from $2d$ hyperbolic space, $\Hp^2$, into rotationally symmetric surfaces $M$. In local coordinates the equation is given by
\EQ{ \label{eq wm}
\psi_{tt} - \psi_{rr} - \coth r \, \psi_r + \frac{ g(\psi)g'(\psi)}{\sinh^2 r} = 0,
}
where  $(\psi, \theta)$ are geodesic polar coordinates on the target surface $M$, and $g$ determines the metric, $ds^2 = d \psi^2 + g^2(\psi) d \theta^2$.

An intriguing feature of this problem is that there is an abundance of finite energy stationary solutions.
This fact stems from two well-known geometric facts: conformal invariance of $2d$ harmonic maps (which are time independent solutions to the problem) and conformal equivalence between $\bbH^{2}$ and the unit disk $\bbD^{2}$.
Another notable feature of this problem is the lack of scaling symmetry. This feature, which is in stark contrast to the Euclidean case and akin to the exterior wave map problem considered by Kenig, Schlag and the first author \cite{KLS, LS}, rules out an {\it a priori} obstruction to asymptotic stability of stationary solutions. However, while the curved background eliminates any natural scaling invariance for the problem, the model still exhibits features of an energy critical equation. Indeed, solutions with highly localized initial data do not see the global geometry of the domain and thus can be well approximated by solutions to the corresponding scale invariant energy critical Euclidean equation $\R^{1+2} \to M$. %To be precise, the evolution of initial data with energy concentrated near the origin in $\Hp^2$ can be well approximated by the underlying Euclidean evolution as such a solution does not see the global geometry of hyperbolic space. When the target manifold is positively curved, we expect the finite energy harmonic maps for the underlying Euclidean evolution to lead to finite time blow-up as in the work of~\cite{KST}.
We believe that such properties make \eqref{eq wm} an interesting model for investigating stability of stationary solutions, and more ambitiously, asymptotic resolution for  large general finite energy solutions into solitons (\emph{soliton resolution}) in the case of negatively curved targets such as $\Hp^2$, and characterization of blow-up solutions in the case of positively curved targets such as $\Sp^2$.  

In this paper we establish asymptotic stability of various time independent solutions (i.e., harmonic maps) to the equivariant wave map equation \eqref{eq wm} on $\bbR \times \bbH^{2}$. More specifically, we consider two targets, namely the two sphere $\bbS^{2}$ and the hyperbolic plane $\bbH^{2}$.
In each case we classify all finite energy equivariant harmonic maps, which exist in abundance in contrast to the Euclidean case.
We then study the stability of each harmonic map by analyzing spectral properties of the linearized operator.

We begin with our results when the target is $\bbH^{2}$, as these are easier to describe. In this case, we show that the spectrum of the linearized operator about each harmonic map consists purely of the absolutely continuous part with no eigenvalue or resonance at the edge. This spectral information allows us to prove Strichartz estimates for the linearized operator, from which asymptotic stability of the harmonic map follows by a Picard iteration argument.

The picture changes drastically in the case of the $\bbS^{2}$ target.
Let $Q_{\lmb}$ be the family of finite energy equivariant harmonic maps from $\bbH^{2}$ to $\bbS^{2}$. These maps can be parametrized by $\lmb \in [0, \infty)$ in such a way that as $\lmb \to 0+$ the image of $Q_{\lmb}$ contracts to the north pole, and as $\lmb \to \infty$ the image of $Q_{\lmb}$ covers the whole of $\bbS^{2}$ except for the south pole. For $Q_{\lmb}$ with small $\lmb$, we prove that the spectrum of the linearized operator is absolutely continuous as in the case of $\bbH^{2}$. This allows us to prove Strichartz estimates in this case and therefore asymptotic stability holds by a perturbative argument. This scenario applies, in particular, to the harmonic map covering the northern hemisphere of $\bbS^{2}$. On the other hand, for $Q_{\lmb}$ with large $\lmb$, we show that there exists a unique simple \emph{gap eigenvalue} $\gapE$ in $(0, \frac{1}{4})$ ($\frac{1}{4}$ is the edge of the a.c. spectrum). Moreover, we demonstrate that $\gapE$ migrates toward zero as $\lmb \to \infty$.
While this phenomenon precludes the possibility of scattering to $Q_{\lmb}$ by a linear mechanism, it nevertheless suggests an interesting picture concerning nonlinear stability of $Q_{\lmb}$ and the rate of scattering; we refer the reader to Remark~\ref{rem:FGR}.

\subsection{Additional context for the problem} Although the authors are not aware of any previous investigations into the model at hand, there has been substantial activity of late regarding dispersive equations on $\R \times \Hp^d$, and it is partially in this context in which this problem can be viewed. Perhaps the most relevant recent works are the proofs of Strichartz estimates for the free wave equation on $\R \times \Hp^d$ together with global small data theory for semi-linear equations with power-type nonlinearities in~\cite{MT11, MTay12, AP}, see also the many references therein. There has also been substantial activity in this direction for the Schr\"odinger equation on $\R \times \Hp^d$ and we refer the reader to~\cite{Ban07, BCD, BCS, IS, IPS, AP09} as well as the references therein for more details. For treatments of the semi-linear elliptic problem, see~\cite{CM, ManSan}. As we are investigating the asymptotic stability of certain stationary solutions to~\eqref{eq wm}, we are forced to confront the linearized operator, which amounts to a radial free evolution operator on hyperbolic space plus a potential term. The dispersive estimates for the free evolution from \cite{AP} thus make up an essential ingredient in the proof.

 \subsection{Setup}
To explain the main results in more detail, we now give a more precise account of our setup.
%To explain the main results we begin by giving more precise details of the setup.
Consider polar coordinates on the hyperboloid model of $\Hp^2$:
\ant{
[0, \infty) \times S^1 \ni (r, \om) \mapsto (\sinh r \sin \om, \sinh r\cos \om, \cosh r) \in \R^{2+1}.
}
Denote this map by $\Psi: [0, \infty) \times S^1 \to (\R^{2+1}, \mathbf{m})$, where $\mathbf{m}$ is the Minkowski metric on $\R^{2+1}$. The hyperbolic metric $\h$ in these coordinates is given by the pull-back of the Minkowski metric on $\R^{2+1}$  by the map $\Psi$, i.e.,  $\h = \Psi^*\mathbf{m}$. We have
\ant{
(\h_{jk}) = \pmat{ 1  &0 \\ 0 & \sinh^2 r}.
}
The volume element is $\sqrt{\abs{\h(r, \om)}} = \sinh r$, and hence for $f : \Hp^2 \to \R$ we have
\ant{
\int_{\Hp^2} f(x) \, d \textrm{Vol}_{\h} = \int_0^{2\pi} \int_0^{\infty} f(\Psi(r ,\om)) \sinh r \, dr \, d \om.
}
For radial functions, $f: \Hp^2 \to \R$ we abuse notation and write $f(x) = f(r)$ and
\ant{
\int_{\Hp^2} f(x) \, d \textrm{Vol}_{\h} = 2\pi \int_0^{\infty} f(r) \sinh r \, dr .
}
We will focus attention on two rotationally symmetric target manifolds, namely, $M= \Sp^2$ and $M = \Hp^2$. We begin with the positively curved case, $\Sp^2$.

\subsection{Equivariant wave maps: $\R \times \Hp^2 \to \Sp^2$} In this section we consider wave maps $U: \R \times \Hp^2 \to \Sp^2$. As both the domain and the target are rotationally symmetric we can consider a restricted class of maps, $U,$ satisfying the equivariance $U \circ \rho = \rho \circ U$, for all rotations $\rho \in SO(2).$ In fact we consider the special subclass of such maps known as $1$-equivariant, or co-rotational, which corresponds to equivariant maps which in local coordinates take the form
\ant{
U(t, r, \om) = (\psi(t, r), \om) \hookrightarrow ( \sin \psi \sin \om, \sin \psi \cos  \om, \cos \psi),
}
where $\psi$ is the azimuth angle measured from the north pole of the sphere and the metric on $\Sp^2$ is given by $ds^2 = d \psi^2+ \sin^2 \psi\,  d \om^2$ (for a more general class of equivariant maps one can consider an ansatz of the form $U(t, r, \om) = (\psi(t, r), \om+\chi(t,r))$). In this formulation, $1$-equivariant wave maps are formal critical points of the Lagrangian
\ant{
\mathcal{L}(U)= \frac{1}{2} \int_{\R} \int_0^{\infty} \left( - \psi_t^2(t, r) + \psi_r^2(t, r) + \frac{ \sin^2 \psi(t, r)}{\sinh^2 r} \right) \, \sinh r \, dr \, dt.
}
The Euler-Lagrange equations reduce to an equation for the azimuth angle $\psi$ and we are led to the Cauchy problem:
\EQ{\label{wm}
&\psi_{tt} - \psi_{rr} - \coth r \, \psi_r + \frac{ \sin(2 \psi)}{2 \sinh^2 r} = 0,\\
& \vec \psi(0)= ( \psi_0, \psi_1).
}
We will often use the notation $\vec \psi(t)$ to denote the pair $\vec \psi(t, r):= ( \psi(t, r), \psi_t(t, r))$. The conserved energy is given by
\EQ{\label{con energy}
\E( \vec \psi)(t) = \frac{1}{2} \int_0^{\infty}  \left( \psi_t^2 + \psi^2_r + \frac{\sin^2 \psi}{\sinh^2 r} \right) \sinh r \, dr = \textrm{constant}.
}
Note that for initial data $\vec \psi(0) = (\psi_0, \psi_1)$ to have finite energy we need $\psi_0(0) = k \pi$ for some $k \in \Z$. For the solution to depend continuously on the initial data this integer $k$ must be preserved by the evolution. Here we consider the case $k =0,$ corresponding to maps that send $r=0$ (the vertex of the hyperboloid) to the north pole of $\Sp^2$, as the other cases are similar.

The behavior of finite energy data at $r=\infty$ is more flexible. One can check that $\psi_0(r)$ has a well defined limit as $r\rightarrow\infty,$ but that this limit can be any finite number, i.e., $\E(\psi_0,\psi_1)<\infty$ implies there exists $\alpha\in\R$ so that $\lim_{r\rightarrow\infty}\psi_0(r)=\alpha$.  This stands in sharp contrast to the corresponding problem for wave maps $\R^{1+2} \to \Sp^2$ where the endpoint can only be an integer multiple of $\pi$ giving such maps a fixed topological degree. Here, the fact that any finite endpoint is allowed can be attributed  to the rapid decay of $\sinh^{-1} r$ as $ r \to \infty$ in the last term in the integrand of~\eqref{con energy}, and is ultimately responsible for the existence of the family of harmonic maps mentioned in the abstract.

In this paper we will consider initial data with endpoints $\psi_0( \infty) = \al$ for $\al \in [0, \pi)$, which means that we will only consider those $\psi_0$ which do not reach the south pole. This leads us to define the energy classes
\EQ{
\E_{\la} := \{ ( \psi_0, \psi_1) \mid \E(\psi_0,\psi_1)< \infty, \, \,  \psi_0(0) =0, \, \, \psi_0( \infty) = 2 \arctan( \la)\}.
}
for $\la \in [0, \infty)$. The reason for this restriction to  $\al \in [0, \pi)$ is given by the presence of a $1-$parameter family, $Q_{\la}$, of finite energy harmonic maps with such endpoints, i.e., solutions to
\EQ{\label{hm}
&Q_{rr}+ \coth r\, Q_r = \frac{\sin 2Q}{2\sinh^2 r},\\
&Q(0) = 0, \quad \lim_{r \to \infty} Q(r) = 2\arctan(\la).
}
For every $\la \in [0, \infty)$ there is a unique finite energy solution $Q_{\la}$ to~\eqref{hm}, given by
\EQ{
Q_{\la}(r) = 2\arctan(\la \tanh(r/2)).
}
Moreover, $(Q_{\la}, 0)$ has energy
\ant{
\E( Q_{\la}, 0) = 1- \cos(Q_{\la}( \infty)) = 2 \frac{\la^2}{ \la^2 +1}
}
 which is minimal in $\E_{\la}$ -- in other words for each angle $\al \in [0, \pi)$, there exists a map connecting $0$ to $\al$ which uses the minimum possible amount of energy and this map is, in fact, the \emph{harmonic map} $Q_{\la}$ with $\la = \tan( \al/2)$.  For endpoints $\al \ge \pi$ there are \emph{no} finite energy harmonic maps. We provide a more detailed description of the $Q_{\la}$ with proofs of the preceding statements in Section~\ref{S2 hm}.

 The existence of the $Q_{\la}$ stands in stark contrast to the corresponding Euclidean problem,  equivariant wave maps $\R^{1+2} \to \Sp^2$, which reduce to the following equation for the azimuth angle $\psi$:
  \EQ{\label{euc wm}
  \psi_{tt} - \psi_{rr} - \frac{1}{r} \psi_r + \frac{ \sin 2 \psi}{2 r^2} = 0.
  }
  In fact, the unique (up to scaling) Euclidean equivariant harmonic map is given by $Q_{\euc}(r) = 2 \arctan(r)$ which connects the north pole to the south pole of the sphere. $Q_{\euc}$ is the unique, nontrivial, finite energy solution to
 \EQ{\label{euc hm}
 Q_{rr} + \frac{1}{r} Q_{r}  = \frac{\sin 2Q}{2 r^2}, \, \, \, Q(0) = 0.
 }
 We remark that $Q_{\euc}$ minimizes the Euclidean energy
 \EQ{
 \E_{\euc}(\psi_0, \psi_1) =  \frac{1}{2} \int_0^{\infty}\left[ (\p_r \psi_0)^2 + \psi_1^2 + \frac{\sin^2 \psi_0}{r^2} \right]\, r \, dr
 }
amongst all degree one maps, i.e., those which satisfy $\psi_0(0) = 0, \psi_0(\infty) = \pi$ and by direct computation one sees that $\E_{\euc}(Q_{\euc}, 0) = 2$. We note that for the hyperbolic harmonic maps $Q_{\la}$ we have \ant{
&\E(Q_{\la}, 0) \to \E_{\euc}(Q_{\euc}, 0) \mas \la \to \infty, \\
&\E(Q_{\la}, 0) \to 0 \mas \la \to 0.
}
\subsubsection{Asymptotic stability of $Q_{\la}$}
It is well known that  $Q_{\euc}$ is \emph{unstable} with respect to the Euclidean equivariant wave map flow and in fact, leads to finite time blow-up, see~\cite{Cote, KST}.  %In fact, the scaling symmetry, $\psi(t, r) \mapsto \psi_{\mu}(t, r) = \psi( \mu t, \mu r)$, leads to  a \emph{zero energy resonance} for the linearized operator, a fact that enters crucially in our analysis below.

A natural question to ask is whether $(Q_{\la}, 0)$ is asymptotically stable for fixed $\la \in [0, \infty)$ under the wave map evolution, \eqref{wm}, in $\E_{\la}$. It is this question that we address here. The natural space in which to consider solutions to~\eqref{wm} is the energy space
\EQ{\label{H0 def}
\| (\psi_0, \psi_1)\|_{\HH_0}^2:=  \int_0^{\infty}   \left[(\p_r\psi_0)^2(r) + \psi_1^2(r) + \frac{\psi_0^2(r)}{\sinh^2 r} \right] \, \sinh r \, dr.
}
Indeed, we endow $\E_{\la}$ with the  ``norm"
\EQ{
\| (\psi_0, \psi_1)\|_{\E_{\la}} := \| (\psi_0, \psi_1) -(Q_{\la}, 0)\|_{\HH_0}.
}
The first result is an affirmative answer to the above question for a range of $\la \in [0, \la_0)$ for some $\la_0 \ge\sqrt{15/8}$.
\begin{thm}\label{s2 stab} There exists $\la_0 \ge\sqrt{15/8}$ so that for every $0 \le \la <\la_0$, the harmonic map $Q_{\la}$ is asymptotically stable in the space $\E_{\la}$. In particular, there exists a $\de_0>0$ such that for every $(\psi_0, \psi_1) \in \E_{\la}$ with
\ant{
 \|(\psi_0, \psi_1)- (Q_{\la}, 0)\|_{\HH_0}< \de_0
}
there exists a unique global solution $\vec \psi(t) \in \E_{\la}$ to~\eqref{wm}. Moreover, $\vec \psi(t)$ scatters to $(Q_{\la}, 0)$ as $t \to \pm \infty$.
\end{thm}

\begin{rem} The phrase \emph{$ \vec \psi(t)$ scatters to $(Q_{\la}, 0)$  as $t \to \pm \infty$} means that there exist  solutions $\vec \fy_{L}^{\pm}(t)$ to the linearized equation
\EQ{\label{lin eq}
\fy_{tt} - \fy_{rr} - \coth r\,  \fy_r + \frac{1}{\sinh^2 r} \fy = 0,
}
so that
\EQ{
\| \vec \psi(t)- (Q_{\la}, 0)- \vec \fy_L^{\pm}(t) \|_{\HH_0} \to 0 \mas  t \to \pm \infty.
}
\end{rem}

\begin{rem}
We note that the number $\sqrt{15/8}$ appears for a technical reason that will be further explained in Section~\ref{spectra}. In short, it is the largest value for $\la$ for which we have a simple proof that the linearized operator about $Q_{\la}$ has no discrete spectrum. The number $2 \arctan{\sqrt{15/8}}$ is slightly less than $3\pi/5$ which means that our stability result holds for maps which wrap more that halfway around the sphere.
\end{rem}
%\Green{Should we include a comment like the following?
%\begin{rem}
%The maps $Q_\lambda,$ and hence $\psi_0$ in the Theorem, are not square integrable. This may sound surprising in view of the embedding $\dot{H}^1(\Hp^2)\hookrightarrow L^2(\Hp^2).$ The reason this is not a problem is that to bound $\|f\|_{L^2(\Hp^3)}$ by $\|\nabla f\|_{L^2(\Hp^3)}$ one has to assume that $\lim_{r\rightarrow\infty}f(r)=0.$ Indeed, the difference $Q_\lambda-\psi_0$ is square integrable.
%\end{rem}
%}
The proof of Theorem~\ref{s2 stab} reduces to Strichartz estimates for the linearized operator after first passing to a radial wave equation on $\R \times \Hp^4$. The reason we can pass to waves on $\Hp^4$ comes from the fact that the nonlinearity in~\eqref{wm} contains a repulsive potential term:
\ant{
\frac{\sin 2 \psi}{ 2\sinh^2 r} = \frac{1}{ \sinh^2 r} \psi + \frac{ \sin 2 \psi - 2 \psi}{2 \sinh^2 r}.
}
This indicates that the linear part of~\eqref{wm} has more dispersion than a free wave on $\R \times \Hp^2$. In fact, after linearizing about $(Q_{\la}, 0)$ we prove that for $\la$ as in Theorem~\ref{s2 stab}, the linear part has the same dispersion as a free wave on $\R \times \Hp^4$. This can be seen by making the following change of variables. For a solution $\vec \psi(t) \in \E_{\la}$ define $ u(t)$ by
\EQ{\label{u S2 def}
 \sinh r \,  u(t, r) :=  \psi(t, r) - Q_{\la}(r).
}
We obtain the following equation for $\vec u(t)$,
\EQ{\label{u eq}
&u_{tt}- u_{rr} - 3 \coth r \, u_r - 2 u + V_{\la}(r) u = \NN_{\Sp^2}(r, u)\\
&\vec u(0)= (u_0, u_1)
}
where the \emph{attractive} potential $V_{\la}$ and the nonlinearity $\NN_{\Sp^2}$ are given by
\begin{align}
& V_{\la}(r) := \frac{ \cos 2Q_{\la} -1}{ \sinh^2 r} \le 0 \label{Vla}\\
&\NN_{\Sp^2}(r, u):= \frac{\sin 2Q_{\la}}{ \sinh^3 r}\sin^2( 2 \sinh r\, u) + \cos2Q_\la \frac{2 \sinh r\, u -  \sin (2 \sinh r \, u)}{2 \sinh^3 r} \label{N S}
\end{align}
The underlying linear equation under consideration is then given by
\EQ{\label{v eq}
v_{tt} - \Delta_{\Hp^4} v -2 v+ V_{\la}v =0
}
for radially symmetric functions $v$. In Section~\ref{strich section} we prove Strichartz estimates in Proposition~\ref{strich} for~\eqref{v eq} with $\la \in [0, \la_0)$ using the spectral transformation, or the \emph{distorted Fourier transform}, for the self-adjoint Schr\"odinger operators
\EQ{\label{H Vla}
&H_0:=-\p_{rr}- 3 \coth r \, \p_r -2,\\
&H_{V_{\la}} := -\p_{rr}- 3 \coth r \, \p_r -2 +V_{\la},
}
following roughly the argument in \cite{LS}, which was based on techniques from~\cite{RodS}, see also \cite{SSS1, SSS2}.
The spectrum $\s(H_{V_{\la}})$ plays a central role in determining the dispersive properties of the wave equation~\eqref{v eq}. It is well known that the spectrum of the Laplacian on $\Hp^4$ is given by $\s( \Delta_{\Hp^4}) = \left[ 9/4, \infty\right)$ where here $9/4 = \left(\frac{d-1}{2}\right)^2$ for $d=4$, and thus we have $\s(H_0) = [1/4, \infty)$ for the shifted operator $H_0 = -\Delta_{\Hp^4}-2$.  The key to our analysis is the existence of $\la_0 \in (0, \infty)$ (in fact we can prove that $\la_0\ge\sqrt{15/8}$) so that for all $0\le \la< \la_0$, the perturbed operator $H_{V_{\la}}$ has purely absolutely continuous spectrum equal to $[1/4, \infty)$. In particular,  $H_{V_{\la}}$ has no negative spectrum, no eigenvalues in the gap $[0, 1/4)$, and the threshold $1/4$ is neither an eigenvalue nor a resonance.

However, as $\la$ becomes large, which means that the harmonic map $Q_{\la}$ wraps further around the sphere, we observe a change in the spectrum of $H_{V_{\la}}$ which results in  a breakdown in the dispersive behavior of solutions to the linear equation~\eqref{v eq}. In particular, as $\la \to \infty$ we establish the existence of a simple gap eigenvalue $\gapE \in (0, 1/4)$.  Moreover we show that as $\la \to \infty$ the eigenvalue $\gapE$ migrates to $0$.  In particular, we prove the following result.

\begin{thm}\label{e val}There exists $\Lambda_0>0$ so that for all $\la > \La_0$, the Schr\"odinger operator $H_{V_{\la}}$ has a unique, simple eigenvalue, $\gapE$, in the spectral gap $(0, 1/4)$. That is, there exists  a solution $\fy_\la \in L^2(\Hp^4)$ to
\EQ{\label{e vec}
H_{V_{\la}} \fy_{\la} =  \gapE \fy_{\la}.
}
where $\gapE \in (0, 1/4)$. Moreover, we have
\EQ{ \label{mu to 0}
\gapE \to 0 \mas \la \to \infty.
}
Finally, if we define
\EQ{
&\lambda_{\sup}:=\sup\{ \la \mid H_{V_{\ti \lambda}} \, \, \textrm{has no e-vals and no threshold resonance} \, \, \forall \, \ti \la < \la\}\\
&\Lambda_{\inf}:=\inf\{ \la \mid H_{V_{\ti \lambda}} \, \, \textrm{has a gap e-val}\, \,  \mu^2_{\ti{\la}} \in (0, 1/4) \, \, \forall \,  \ti \la > \la\}
}
Then both $H_{\la_{\sup}}$ and $H_{\Lambda_{\inf}}$ have threshold resonances.
\end{thm}

\begin{rem} \label{rem:FGR} One immediate consequence of the presence of the gap eigenvalues for large $\la$ is that we can no longer prove a stability  result as in  Theorem~\ref{s2 stab}, by  a perturbative argument based on the dispersive properties of the underlying linear equation, i.e., Strichartz estimates.  On the other hand, a Struwe-type bubbling argument \cite{Struwe} suggests that  any solution~$\vec \psi(t)$ to~\eqref{wm} that blows up in finite time must bubble off a Euclidean harmonic map $Q_{\euc}$ and thus must have enough energy to wrap completely around the sphere. This gives some evidence towards a conjecture that in fact every $Q_{\la}$ is  stable -- as small perturbations of $Q_{\la}$ will not have enough energy to bubble off a $Q_{\euc}$ --  but for large $\la$, the stability manifests via a completely nonlinear mechanism, possibly as in the work of Soffer, Weinstein~\cite{SW99}.
\end{rem}

\begin{rem} At this point we do not know the precise location  of $\la_0 = \lambda_{\sup}$ in Theorem~\ref{s2 stab}, or of $\Lambda_0= \Lambda_{\inf}$ in Theorem~\ref{e val} or whether these two values are equal. Indeed, the existence of gap eigenvalues for large $\la$ is demonstrated by a contradiction argument and thus does not reveal a precise geometric reason for the breakdown in \emph{linear stability} described in Remark~\ref{rem:FGR}. On the other hand, this asymptotic-in-$\la$ failure of linear stability is natural in view of the bubbling mentioned in Remark~\ref{rem:FGR} and the explicit blow-up constructions  for the corresponding Euclidean problem from~\cite{KST, RS, RR}. Indeed the Euclidean blow-up constructions rely on energy concentration schemes which see only the local geometry of space, which suggests similar behavior is possible for the hyperbolic problem at hand, as long as the solution has enough energy to bubble off a $Q_{\euc}$.
\end{rem}

\begin{rem}
The existence of gap eigenvalues is a rather surprising feature of this model, as this contrasts greatly with the corresponding Euclidean wave maps problem. Key to the proof of Theorem~\ref{e val} is the fact that after a renormalization, the Schr\"odinger operator $H_{V_\la}$ formally approaches (as $\la \to \infty$) the operator $H_{V_{\euc}}$ obtained by linearizing~\eqref{euc wm} about $Q_{\euc}$. Assuming, for contradiction, the nonexistence of a gap eigenvalue, this formal approximation can be made precise on a region that increases in size as $\la$ increases. This fact allows us to treat the hyperbolic spectral picture as a  perturbation of its Euclidean counterpart. We can then  pair the  existence of a threshold resonance for the Euclidean problem together with the existence of the spectral gap in the hyperbolic problem to force a contradiction. We refer the reader to Sections \ref{subsec:gapEoverview}--\ref{subsec:gapEmig} for details.
%A key observation in the proof of Theorem \ref{e val} is that the renormalized potential $\frac{1}{\lambda^2}V(r/\lambda)$ tends to $$V_{\mathrm{euc}}(r):=-\frac{2}{(1+(r/2)^2)^2}$$ as $\lambda\to\infty.$ The operator $\Delta_{\R^4}+V_{\mathrm{euc}}$ arises in linearizing the co-rotational wave maps equation $\R^{2+1}\to \Sp^2$ around the ground state harmonic map $Q_{\mathrm{euc}},$ much in the same way as $H_{V_\lambda}$ came about in our case. Now as this Euclidean operator has a resonance at the left end point of its spectrum $[0,\infty),$ and in view of the convergence $\frac{1}{\lambda^2}V(\cdot/\lambda)\to V_{\mathrm{euc}}(\cdot)$ it is not unreasonable to expect that for large $\lambda$ the operator $H_{V_{\lambda}}$ has an eigenvalue which tends to zero as $\lambda$ tends to infinity. As explained in the beginning of Subsection \ref{subsec:gapEoverview} we are able to make this heuristic argument precise through a renormalization argument and by comparison with the Euclidean resonance.
\end{rem}
\subsection{Equivariant wave maps: $\R \times \Hp^2 \to \Hp^2$} We next consider wave maps $U: \R \times \Hp^2 \to \Hp^2,$ again restricting attention to co-rotational maps, meaning maps $U$, which in coordinates take the form
\ant{
U(t, r, \om) = (\psi(t, r), \om) \hookrightarrow ( \sinh \psi \sin \om, \sinh \psi \cos  \om, \cosh \psi)\in \R^{2+1},
}
where the metric on the target $\Hp^2$ is given by $ds^2 = d \psi^2+ \sinh^2 \psi\,  d \om^2$. In this formulation, $1$-equivariant wave maps are formal critical points of the Largrangian
\ant{
\mathcal{L}(U)= \frac{1}{2} \int_{\R} \int_0^{\infty} \left( - \psi_t^2(t, r) + \psi_r^2(t, r) + \frac{ \sinh^2 \psi(t, r)}{\sinh^2 r} \right) \, \sinh r \, dr \, dt.
}
The Euler-Lagrange equations reduce to an equation for  $\psi$ and we are led to the Cauchy problem:
\EQ{\label{wm h}
&\psi_{tt} - \psi_{rr} - \coth r \, \psi_r + \frac{ \sinh(2 \psi)}{2 \sinh^2 r} = 0,\\
& \vec \psi(0)= ( \psi_0, \psi_1).
}
The conserved energy is given by
\EQ{
\EE( \vec \psi)(t) = \frac{1}{2} \int_0^{\infty}  \left( \psi_t^2 + \psi^2_r + \frac{\sinh^2 \psi}{\sinh^2 r} \right) \sinh r \, dr = \textrm{constant}.
}
Note that for initial data $\vec \psi(0) = (\psi_0, \psi_1)$ to have finite energy we need $\psi_0(0) = 0$, which  means that a finite energy map must  fix the vertex of the hyperboloid. The behavior of $\psi$ at $r = \infty$ is again more flexible than the corresponding Euclidean equation for wave maps $\R^{1+2} \to \Hp^2$ due to the rapid decay of $\sinh^{-1} r$ as $r \to \infty$. We note that for any finite energy data $( \psi_0, \psi_1)$ the limit $\lim_{r \to \infty} \psi_0(r) = \al$ exists but can take any value $\al \in [0, \infty)$. We thus again define the energy classes
\EQ{
\EE_{\la}:= \{ (\psi_0, \psi_1)  \mid \EE( \vec \psi) < \infty, \, \, \psi_0(0) = 0, \,  \, \psi_0( \infty) = 2 \arctanh( \la)\}
}
with $\lambda\in[0,1).$   %Once again in contrast to the case when the base manifold is $\R^2$,
 We will demonstrate the presence of a family of nontrivial harmonic maps taking all  energies ranging from $0$ to infinity. This is a surprising and distinctive feature of this model in light of the fact that no such maps exist in the corresponding Euclidean problem. In this context a harmonic map is a solution $P$ to the equation
 \EQ{\label{hm h}
&P_{rr}+ \coth r\, P_r = \frac{\sinh 2P}{2\sinh^2 r},\\
&P(0) = 0, \quad \lim_{r \to \infty} P(r) = 2\arctanh(\la).
}
For every $\la \in [0, 1)$ there is a unique finite energy solution $P_{\la}$ to~\eqref{hm h} given by
\EQ{\label{P la def}
P_{\la}(r):= 2 \arctanh(\la\tanh(r/2))
}
In Section~\ref{H2 harm} we show that $P_{\la}$ has energy
\ant{
\EE( P_{\la}, 0) = 2 \frac{\la^2}{1- \la^2}
}
which minimizes the energy in $\EE_\la$. Note that  $\EE( P_{\la},0) \to \infty$ as $\la \to 1^-$ and $\EE( P_{\la}, 0)  \to 0$ as $ \la \to 0$.

We recall the well known fact that for the Euclidean case of wave maps from $\R^{1+2} \to \Hp^2$ there are \emph{no finite energy nontrivial harmonic maps} due to the negative curvature of the target $\Hp^2$.

\subsubsection{Asymptotic stability of $P_{\la}$} We now turn to the question of the asymptotic stability of  $P_{\la}$ in $\EE_\la$ for fixed $\la \in [0, 1)$.  We prove the following result.
\begin{thm}\label{h2 stab} For every $\la \in [0, 1)$ the harmonic map $P_{\la}$ is asymptotically stable in $\EE_{\la}$. In particular, for each $\la \in [0, 1)$ there exists a $\de_0>0$ so that for every $(\psi_0, \psi_1) \in \EE_{\la}$ with
\ant{
 \| (\psi_0, \psi_1) - (P_{\la}, 0)\|_{\HH_0} < \de_0
}
there exists a unique, global solution $\vec \psi(t) \in \EE_{\la}$ to~\eqref{wm h}. Moreover, $\vec \psi(t)$  scatters to $(P_{\la}, 0)$  as $t \to \pm \infty$.
\end{thm}

The proof of Theorem~\ref{h2 stab}  follows the same outline as in the previous subsection. In particular we establish Strichartz estimates for the operator obtained by linearizing about $P_{\la}$ and then passing to an equation on $\R \times \Hp^4$. For a solution $\vec \psi(t) \in \EE_{\la}$ to~\eqref{wm h} we define $\vec u(t )$ by
\EQ{\label{u H2 def}
 \sinh r \, u(t, r):=  \psi(t, r) - P_{\la}(r).
}
Then $\vec u(t)$ solves
\EQ{\label{u eq h}
&u_{tt}- u_{rr} - 3 \coth r \, u_r - 2 u + U_{\la}(r) u = \NN_{\Hp^2}(r, u)\\
&\vec u(0)= (u_0, u_1)
}
where the \emph{repulsive} potential $U_{\la}$ and the nonlinearity $\NN_{\Hp^2}$ are given by
\begin{align}
& U_{\la}(r) := \frac{ \cosh 2P_{\la} -1}{ \sinh^2 r} \ge 0\label{Ula}\\
&\NN_{\Hp^2}(r, u):=- \frac{\sinh 2P_{\la}}{ \sinh^3 r}\sinh^2( 2 \sinh r\, u) + \cosh2P_{\la} \frac{2 \sinh r\, u -  \sinh (2 \sinh r \, u)}{2 \sinh^3 r} \label{N H}
\end{align}
The underlying linear equation under consideration is then given by
\EQ{\label{v eq h}
v_{tt} - \Delta_{\Hp^4} v -2 v+ U_{\la}v =0
}
for radially symmetric functions $v$. In Section~\ref{strich section} we prove Strichartz estimates in Proposition~\ref{strich}  for~\eqref{v eq h} using the spectral transformation, or the \emph{distorted Fourier transform}, for the self-adjoint Schr\"odinger operators
\EQ{
&H_0:=-\p_{rr}- 3 \coth r \, \p_r -2,\\
&H_{U_{\la}} := -\p_{rr}- 3 \coth r \, \p_r -2 +U_{\la},
}
again following roughly the argument in \cite{LS, RodS, SSS1, SSS2}. The key point here is that the repulsive potential $U_{\la}$ rules out the possibility of discrete spectrum for $\s( H_{U_{\la}})$, and thus $H_{U_\la}$ has essential spectrum  $[1/4, \infty)$, with no negative spectrum, no gap eigenvalues, and no eigenvalue or resonance at the threshold $1/4$. %From the Strichartz estimates we can prove the following result:

\subsection{Brief outline of the paper}
In Section~\ref{prelim} we establish the various facts about the harmonic maps $Q_{\la}$ and $P_{\la}$ defined above.  We also give more details concerning the passage to equations on $\R \times \Hp^4$ outlined above. In particular, we show that the small data Cauchy problems, the $2d$ linearized problem in $\HH_{0}$ and the $4d$ problem in $H^1 \times L^2( \Hp^4)$ are equivalent.

In Section~\ref{spectra} we study the spectrum of the linearized operator $H_{V_{\la}}$,  which corresponds to $\Sp^2$ valued maps. We begin by showing that $\s(H_{V_{\la}})$ has no discrete spectrum for $\la < \sqrt{15/8}$. %This is achieved by way of  a simple comparison argument with a solutions to an operator with a  deeper potential well.
Beginning from Section~\ref{subsec:gapEoverview}, we then present the proof of Theorem~\ref{e val}.
%We then present the proof of Theorem~\ref{e val} and we refer the reader to the beginning of Section~\ref{subsec:gapEoverview} for a brief overview of the argument.  
%The key insight which drives the proof is that by rescaling the operator  $H_{V_{\la}}$, we obtain a renormalized operator $\ti H_{V_{\la}}$ which formally approaches (as $\la \to \infty$) the operator $H_{\euc}$ obtained by linearizing the Euclidean equation~\eqref{euc wm} about $Q_{\euc} $. The explicit resonance for the $H_{\euc}$ is then used to

In Section~\ref{strich section} we prove Strichartz estimates -- Proposition~\ref{strich}-- for the linearized equations~\eqref{v eq} and~\eqref{v eq h}. In the former case, we need to restrict to values of $\la$ as in Theorem~\ref{s2 stab}. In the latter case, the Strichartz estimates hold for all $\la \in [0, 1)$ since the potential $U_{\la}$ is repulsive.

Finally, in Section~\ref{proofs} we prove Theorem~\ref{s2 stab} and Theorem~\ref{h2 stab} by the usual contraction mapping argument based on the Strichartz estimates proved in Proposition~\ref{strich}.

%%%%%%%%%%%%%%%%%%%%%%%%%%%%%%%%%%%%%%%%%%%%%%%%%%%%%%%%%%%%%%%%%%%%%%%%%%%%%%%

\section{Preliminaries}\label{prelim} In this section we establish the existence and uniqueness of the harmonic maps $Q_{\la}$ and $P_{\la}$ described in the introduction. We give simple geometric descriptions of these maps and prove several properties that we will need in the ensuing arguments. We also prove some additional preliminary facts including an equivalence between the $2d$ and $4d$ Cauchy problems described in the introduction.

We begin with the case of harmonic maps into $\Sp^2$.

\subsection{Harmonic maps into $\Sp^2$}\label{S2 hm} Here we prove various facts about the harmonic maps $Q_{\la}$. For convenience we collect these facts into a proposition.
\begin{prop}\label{hm prop}
For every $0 \le \al< \pi$ there exists a unique, finite energy stationary solution to~\eqref{wm}, i.e.,  a harmonic map, $(Q_{\la}, 0) \in \E_{\la}$ which solves~\eqref{hm}, where
\EQ{\label{Q def}
&Q_{\la}(r) = 2\arctan(\la \tanh(r/2))\\
&\la \in [0, \infty), \, \, \al = \al(\la) = 2 \arctan(\la) = \lim_{r \to \infty} Q_{\la}(r).
}
Moreover, $(Q_{\la}, 0) \in \E_{\la}$ has energy
\EQ{\label{hm en}
\E( Q_{\la}, 0) = 2 \frac{\la^2}{1+\la^2},
}
which is minimal in $\E_{\la}$. Finally, the $Q_{\la}$ with $\la\in[0, \infty)$ are the only finite energy stationary solutions to~\eqref{wm}.
\end{prop}

\begin{proof}
We are seeking to classify all stationary finite energy solutions to~\eqref{wm}. Recall from the introduction that  any finite energy harmonic map $Q$ must have $Q(0) = 0$ and $Q(\infty) = \al \in [0 , \infty)$. Thus we would like to find all solutions $Q$ to
\EQ{\label{hm bd}
&Q_{rr} + \coth r \, Q_r = \frac{\sin2 Q}{2\sinh^2 r},\\
&Q(0) = 0, \, \, \lim_{r \to \infty}Q(r) = \al \in [0, \infty).
}
One can check directly that $Q_{\la}$, as defined in~\eqref{Q def}, satisfies~\eqref{hm bd} with~$\al = 2 \arctan(\la)$. One can also directly compute the energy to verify~\eqref{hm en}.

To prove the remaining statements in Proposition~\ref{hm prop} we begin by giving a simple geometric interpretation of $Q_{\la}$. Recall that stereographic projection of $\Hp^2$ onto the Poincar\'e disc, $\D$, viewed as a subset of $\R^2$,  is given by the map
\ant{
( \sinh r \, \cos \om, \sinh r\, \sin \om, \cosh r ) \mapsto ( \tanh(r/2) \cos \om,  \tanh(r/2) \, \sin \om) .
}
Next, we rescale the disc by $\la \in [0, \infty)$ via
\ant{
( \tanh(r/2) \cos \om,  \tanh(r/2) \, \sin \om) \mapsto ( \la\tanh(r/2) \cos \om,  \la \tanh(r/2) \, \sin \om).
}
Finally, recall that the inverse of stereographic projection, $\R^2 \to \Sp^2-\{\textrm{south pole}\}$ is given by
\ant{
(\rho \cos \om, \rho \sin \om) \mapsto ( \sin(2\arctan \rho) \, \cos \om, \sin(2\arctan\rho) \,  \sin \om, \cos(2 \arctan\rho)).
}
Then as solutions to~\eqref{wm} or~\eqref{hm} are expressed in terms of the azimuth angle on $\Sp^2$,  we see that  $Q_{\la}$ is simply the composition of the above three maps. % hyperbolic stereographic projection $\Hp^2 \to B(0, \la)$, where $B(0, \la)$ is the ball of radius $\la$ centered at the origin, together with restriction of the inverse stereographic projection $\R^2 \to \Sp^2$, restricted to $B(0, \la)$.

This geometric interpretation motivates the following change of variables in~\eqref{hm bd}. Setting
\EQ{
s:= \log( \tanh(r/2)), \, \, \fy(s):= Q(r),
}
we see that~\eqref{hm bd} reduces to the following equation for $\fy$:
\EQ{\label{ode}
 &\fy'' =  \frac{1}{2} \sin 2 \fy,\\
 & \fy(- \infty) = 0, \, \, \fy(0) = \al.
 }
which is an autonomous ode and none other than the equation for the pendulum. Multiplying the first line in~\eqref{ode} by $\fy'$ and integrating from $s_1$ to $s_2$ yields the energy identity
\EQ{\label{ode en id}
\fy_s^2(s_2) - \fy_s^2(s_1)= \sin^2(\fy(s_2))- \sin^2(\fy(s_1))
}
A standard analysis of the phase portrait in $( \fy, \fy')$ coordinates together with~\eqref{ode en id} mandates the condition that any nontrivial solution satisfies $0 <\fy(0) < \pi$. In particular, we note that a solution with $\fy(- \infty) = 0$ corresponds to the unstable manifold at $(0, 0)$, (which connects to the stable manifold at $(\pi, 0)$ as $s \to +\infty$). Using~\eqref{ode en  id} one sees that if there existed a nontrivial trajectory emanating from $(0, 0)$ at $s = - \infty$ and such that $\fy(s_0) = \pi$ for some $s_0 \in \R$, then we would have $\fy_s(s_0) = 0$. But then $(\fy, \fy_s)(s_0) = (\pi, 0)$ and therefore $\fy$ must  be a trivial solution, which contradicts our assumption.

One can also see this by noting that the unique positive solution (up to translation in $s$) is given by $\fy(s) = 2 \arctan(e^s)$. In particular, we have shown that for each $\al \in [0, \pi)$ there is a unique solution to~\eqref{hm bd}. For $\al \ge \pi$ there are no solutions.

 It remains to show that $(Q_{\la}, 0)$ minimizes the energy in $\E_{\la}$. This follows as a direct consequence of the following ``Bogomol'nyi factorization": Let $\vec \psi(t) = ( \psi(t), \psi_t(t)) \in \E_{\la}$. Then we have
 \ant{
 &\E( \vec \psi) = \frac{1}{2} \int_0^{\infty} \psi_t^2 \, \sinh r \, dr + \frac{1}{2}\int_{0}^{\infty} \left( \psi_r - \frac{\sin \psi}{ \sinh r} \right)^2 \, \sinh r\, dr + \int_{0}^{\infty}  \sin \psi \psi_r \, dr\\
 & = \frac{1}{2} \int_0^{\infty} \psi_t^2 \, \sinh r \, dr + \frac{1}{2}\int_{0}^{\infty} \left( \psi_r - \frac{\sin \psi}{ \sinh r} \right)^2 \, \sinh r\, dr + \cos \psi(t, 0) - \cos \psi(t,  \infty)\\
 & = \frac{1}{2} \int_0^{\infty} \psi_t^2 \, \sinh r \, dr + \frac{1}{2}\int_{0}^{\infty} \left( \psi_r - \frac{\sin \psi}{ \sinh r} \right)^2 \, \sinh r\, dr + 1- \cos(2 \arctan( \la))
 }
 For the solution $\vec \psi(t) = (Q_{\la}, 0)$ the first two integrals--which we note are always  non-negative--vanish identically, which proves that $(Q_{\la}, 0)$ uniquely minimizes the energy in $\E_{\la}$. Finally, a simple  calculation yields
 \ant{
 \E(Q_{\la}, 0) = 1- \cos(2 \arctan( \la)) = 2 \frac{\la^2}{1+ \la^2}
 }
 and this completes the proof.
\end{proof}

\subsection{Harmonic maps into $\Hp^2$}\label{H2 harm}
Here we prove the analogous result for $P_{\la}$ while providing a simple geometric interpretation.
\begin{prop}\label{hm prop h}
For every $ \be \in[0, \infty)$ there exists a unique, finite energy stationary solution to~\eqref{wm h}, i.e.,  a harmonic map, $(P_{\la}, 0) \in \EE_{\la}$ which solves~\eqref{hm h}, where
\EQ{\label{P def}
&P_{\la}(r) = 2\, \arctanh(\la \tanh(r/2))\\
&\la \in [0, 1), \, \, \be= \be(\la) = 2 \arctanh(\la) = \lim_{r \to \infty} P_{\la}(r).
}
Moreover, $(P_{\la}, 0) \in \EE_{\la}$ has energy
\EQ{\label{hm en h}
\EE( P_{\la}, 0) = 2 \frac{\la^2}{1-\la^2},
}
which is minimal in $\EE_{\la}$. Finally, the $P_{\la}$ with $\la\in[0, 1)$ are the only finite energy stationary solutions to~\eqref{wm h}.
\end{prop}

\begin{proof} We would like to classify solutions to
\EQ{\label{hm h bd}
&P_{rr}+ \coth r\, P_r = \frac{\sinh 2P}{2\sinh^2 r},\\
&P(0) = 0, \quad \lim_{r \to \infty} P(r) =\be.
}
for $\be \in [0, \infty)$. One can check directly that $P_{\la}$ as defined in~\eqref{P def} solves~\eqref{hm h bd} with $\be = 2 \arctanh( \la)$, and the energy of $(P_{\la}, 0)$ satisfies~\eqref{hm en h}.

For a geometric interpretation of the $P_{\la}$ we recall again the stereographic of $\Hp^2$ onto the Poincar\'e disc, $\D$, which is a conformal isomorphism and is given, in coordinates by
 \ant{
( \sinh r \, \cos \om, \sinh r\, \sin \om, \cosh r ) \mapsto ( \tanh(r/2) \cos \om,  \tanh(r/2) \, \sin \om) .
}
Next, we perform the map $z \mapsto  \la z$, $z \in \D$, which is \emph{finite energy harmonic map} from $\D \to \D$  for $\la \in [0, 1)$. In coordinates this is given by
\ant{
( \tanh(r/2) \cos \om,  \tanh(r/2) \, \sin \om) \mapsto ( \la\tanh(r/2) \cos \om,  \la \tanh(r/2) \, \sin \om).
}
Finally, recall that the inverse of stereographic projection, $\D \to \Hp^2$  given by
\ant{
(\rho \cos \om, \rho \sin \om) \mapsto ( \sinh(2\arctanh \rho) \, \cos \om, \sinh(2\arctanh\rho) \,  \sin \om, \cosh(2 \arctanh\rho)).
}
It is clear that~$P_{\la}$ is a composition of these three maps.

As is the case for maps to the sphere, we can also view~\eqref{hm h bd} as an autonomous equation, with the change of variables
\ant{
s= \log( \tanh(r/2)) , \,  \, \phi(s) = P(r).
}
Then~\eqref{hm h bd} reduces to
\ant{
&\phi'' = \frac{1}{2} \sinh 2\phi\\
&\phi(-\infty) = 0, \, \, \phi(0) = \be
}
from which the existence and uniqueness of the $P_{\la}$ is also apparent.

Finally, to show that $P_{\la}$ minimizes the energy in $\EE_{\la}$ we again preform the ``Bogomol'nyi factorization": Let $\vec \psi(t) = ( \psi(t), \psi_t(t)) \in \EE_{\la}$. Then we have
 \ant{
 &\E( \vec \psi) = \frac{1}{2} \int_0^{\infty} \psi_t^2 \, \sinh r \, dr + \frac{1}{2}\int_{0}^{\infty} \left( \psi_r - \frac{\sinh \psi}{ \sinh r} \right)^2 \, \sinh r\, dr + \int_{0}^{\infty}  \sinh \psi \psi_r \, dr\\
 %& = \frac{1}{2} \int_0^{\infty} \psi_t^2 \, \sinh r \, dr + \frac{1}{2}\int_{0}^{\infty} \left( \psi_r - \frac{\sinh \psi}{ \sinh r} \right)^2 \, \sinh r\, dr + \cosh \psi(t, 0) - \cosh \psi(t,  \infty)\\
 & = \frac{1}{2} \int_0^{\infty} \psi_t^2 \, \sinh r \, dr + \frac{1}{2}\int_{0}^{\infty} \left( \psi_r - \frac{\sinh \psi}{ \sinh r} \right)^2 \, \sinh r\, dr + \cosh(2 \arctanh( \la))-1
 }
For the solution $\vec \psi(t) = (P_{\la}, 0)$ the first two integrals--which are always non-negative--vanish identically, which proves that $(P_{\la}, 0)$ uniquely minimizes the energy in $\EE_{\la}$. Finally, a simple  calculation yields
 \ant{
 \EE(P_{\la}, 0) =  \cosh(2 \arctan( \la))-1 = 2 \frac{\la^2}{1- \la^2}
 }
 and this completes the proof.
\end{proof}

\subsection{Reduction to equations on $\R \times \Hp^4$}\label{4d}
Next, we provide more details related to the $4d$ reductions to the Cauchy problems~\eqref{u eq} and~\eqref{u eq h}  outlined in the introduction.

First, we prove an estimate that gives an $L^{\infty}$ bound on solutions to~\eqref{wm} and to~\eqref{wm h} in terms of their energy. As the proof is the same in both cases we shorten the exposition by considering solutions to~\eqref{eq wm}, namely
\EQ{ \label{wm g}
&\psi_{tt} - \psi_{rr} - \coth r \, \psi_r + \frac{ g(\psi)g'(\psi)}{\sinh^2 r} = 0,\\
&E( \vec \psi) :=  \frac{1}{2}\int_0^\infty  \left( \psi_t^2 + \psi^2_r + \frac{g^2( \psi)}{\sinh^2 r} \right) \sinh r \, dr.
}
where in the cases under consideration we have $g(\psi) = \sin \psi$, $E= \E$ for maps into $\Sp^2$, and $g( \psi) = \sinh \psi$, $E= \EE$ for maps into $\Hp^2$.
\begin{lem} Let $\vec \psi(t)$ be a finite energy solution to~\eqref{wm g} defined on the interval $t \in I$ with $\psi(t, 0) = 0$ for every $t \in I$. Then there exists a function $C$ with $C( \rho) \to 0$ as $\rho \to 0$ so that
\EQ{\label{l inf bound}
\sup_{t \in I} \| \psi(t) \|_{L^{\infty}} \le C( E( \vec \psi)).
}
\end{lem}
%We introduce the notation
%\ant{
%E_a^b( \psi, \phi) := \frac{1}{2}\int_a^b  \left( \phi^2 + \psi^2_r + \frac{\sin^2 \psi}{\sinh^2 r} \right) \sinh r \, dr.
%}
\begin{proof} Following, e.g.,~\cite[Chapter $8$]{SSbook} we define the function
\ant{
G(\psi)  = \int_{0}^{\psi} \abs{g(\rho)} \, d \rho,
}
and we note that $G(0) = 0$, $G$ is increasing, and $G(\psi) \to \infty$ as $\psi \to \infty$. %Using the function $G$ we can immediately deduce that any finite energy solution $ \psi(t)$ to \eqref{wm} must be bounded on its maximal interval of existence $I$. Indeed,
For any fixed $t \in I$ we have
\EQ{\label{G bound}
\abs{G( \psi(t, r)) }&= \abs{G( \psi(t,r)) - G( \psi(t, 0))} = \abs{\int_{\psi(t, 0)}^{\psi(t, r)} \abs{g (\rho)} \, d \rho}\\
& = \int_0^{r} \abs{ g( \psi(t, r))} \abs{ \psi_r(t, r)} \, dr  \le E( \vec \psi)\\
%& \le \E_0^{r}( \psi(t, \cdot), 0) \le \E( \vec \psi)
}
Then~\eqref{l inf bound} follows from \eqref{G bound} and the fact that $G$ is increasing. %we can deduce that
%\ant{
%\sup_{t \in I}\|\psi(t)\|_{L^{\infty}} \le C(\E(\vec \psi))
%}
\end{proof}
%In these notes we will restrict our attention to data $(\psi_0, \psi_1) \in \E_{\al}$ with $0 \le \al< \pi$. As will be apparent in the next section, it will be convenient to write such $\al$ as $\al= \al(\la) = 2\arctan(\la)$ with $\la \in [0, \infty)$. Note that

Next, we establish  an equivalence of the Cauchy problems~\eqref{wm} with~\eqref{u eq} as well as~\eqref{wm h} with~\eqref{u eq h}  by proving an isomorphism between the spaces $\HH_0$ and $H^1 \times L^2( \Hp^4)$, where $\HH_0$ is defined as in~\eqref{H0 def} and where for radially symmetric $u, v : \Hp^4 \to \R$ we set
\ant{
\|(u, v)\|_{ H^1 \times L^2( \Hp^4)}^2:= \int_0^{\infty} \left(u_r^2(r)+ v^2(r) \right) \sinh^3 r \, dr.
}
We use the notation $H^1$ for the above as opposed to $\dot{H}^1$ due to the embedding~$\dot{H}^1( \Hp^d) \hookrightarrow L^2( \Hp^d)$ for $d \ge 2$. %For radial functions $u(r)$ which vanish at the origin, this embedding is follows from the identity
We prove the following simple lemma.
\begin{lem}\label{2d to 4d}Let $(\psi, \phi) \in \HH_0(\Hp^2)$ with $\psi(0)=0$, $\psi( \infty) = 0$. Then if we define $(u, v)$ by
\ant{
 (\psi(r), \phi(r)) = ( \sinh r u(r), \sinh r v(r))
 }
  we have
  \EQ{\label{H=H}
  \| (\psi, \phi) \|_{\HH_0}^2 \le  \|(u, v) \|_{H^1 \times L^2(\Hp^4)}^2 \le 9 \|(\psi, \phi) \|_{\HH_0}^2.
}
\end{lem}

\begin{rem} We note that Lemma~\ref{2d to 4d} implies that in order to prove Theorem~\ref{s2 stab} and Theorem~\ref{h2 stab} it suffices to consider the corresponding results for the Cauchy problems~\eqref{u eq}, respectively~\eqref{u eq h}, with initial data $\vec u(0) = (u_0, u_1) \in H^1 \times L^2 ( \Hp^4)$.
\end{rem}

\begin{proof} Since $\|\phi\|_{L^2(\Hp^2)}^2= \|v \|_{L^2( \Hp^4)}^2$ it suffices to just consider $u$ and $\psi = \sinh r u$. Integration by parts yields the following identity:
\EQ{\label{2d 4d}
\int_0^{\infty} \left(\psi_r^2 + \frac{\psi^2}{\sinh^2 r} \right) \, \sinh r \, dr =  \int_0^{\infty} u_r^2 \sinh^3 r \, dr - 2 \int_0^{\infty} u^2 \, \sinh^3 r \, dr,
}
which implies that
\ant{
\int_0^{\infty} \left(\psi_r^2 + \frac{\psi^2}{\sinh^2 r} \right) \, \sinh r \, dr  \le  \int_0^{\infty} u_r^2 \sinh^3 r \, dr,
}
giving the left-hand inequality in~\eqref{H=H}. On the other hand,
\ant{
%\int_0^{\infty} u^2 \, \sinh^3 r \, dr &=
\int_0^{\infty} \psi^2 \, \sinh r \, dr &\le \int_0^{\infty} \psi^2 \cosh r \, dr  = -2 \int_0^{\infty} \psi \psi_r \, \sinh r \, dr \\&\le2 \left( \int_0^{\infty} \psi_r^2 \sinh r \, dr\right)^{\frac{1}{2}}\left( \int_0^{\infty} \psi^2 \sinh r \, dr\right)^{\frac{1}{2}},
 }
 which means that
 \ant{
 \int_0^{\infty} u^2 \, \sinh^3 r \, dr &=\int_0^{\infty} \psi^2 \, \sinh r \, dr \le 4 \int_0^{\infty} \psi_r^2 \sinh r \, dr.
 }
 Combining the above with~\eqref{2d 4d} yields the right-hand-side of~\eqref{H=H}.
\end{proof}

%%%%%%%%%%%%%%%%%%%%%%%%%%%%%%%%%%%%%%%%%%%%%%%%%%%%%%%%%%%%%%%%%%%%%%%%%%%%%%%%%

\section{The linearized operator $H_{V_{\la}}$: Analysis of the Spectrum}\label{spectra}
This section gives a detailed analysis of the spectrum of the~Schr\"odinger operator $H_{V_{\la}}$ defined in~\eqref{H Vla}, which is  self-adjoint  %$H_0:=-\Delta_\g -2$ and $H_V=-\Delta_\g - 2+ V$
on the domain
$\calD: = H^2( \Hp^4)$, restricted to radial functions. In Section \ref{subsec:smallLmb}, we establish a positive result, Proposition~\ref{good spec}, for a range of $\la$, namely $0 \le \la< \sqrt{15/8}$. We prove that for $\la$ in this range the spectrum of $H_{V_{\la}}$ coincides with that of the unperturbed operator $H_{0} := - \Delta_{\Hp^4} - 2$. Next, we show that this breaks down for large $\la$. In particular, in the rest of this section, we prove that for $\la$ large there is a unique simple eigenvalue $\gapE$ in the spectral gap $(0, \frac{1}{4})$ and $\gapE \to 0$ as $\la \to \infty$. This is the content of Theorem~\ref{e val}.

First we pass to the half-line by conjugating  by $\sinh^{\frac{3}{2}} r$.  Indeed,  the map
\EQ{ \label{4 to 1}
L^2(\Hp^4) \ni \fy\mapsto \sinh^{\frac{3}{2}} r\, \fy =: \phi \in L^2(0, \infty)
}
is an isomorphism of $L^2(\Hp^4)$, restricted to radial functions, onto $L^2([0, \infty))$. If we define $\LL_0, \LL_{V_{\la}}$ by
\EQ{
&\LL_0 := - \p_{rr} + \frac{1}{4} + \frac{3}{4 \sinh^2 r},\\
&\LL_{V_{\la}} := - \p_{rr} + \frac{1}{4} + \frac{3}{4 \sinh^2 r} + V_{\la}(r),
}
we have
\EQ{ \label{conj}
&(H_0 \fy)(r) = \sinh^{-\frac{3}{2}}r (\LL_0  \phi)(r),\\
&(H_{V_{\la}} \fy)(r) = \sinh^{-\frac{3}{2}} r(\LL_{V_{\la}} \phi)(r).
}
Hence it suffices to work with $\LL_0$ and with $\LL_{V_{\la}}$ on the half-line. We recall that $V_{\la}$ is an \emph{attractive} potential and is given by
\EQ{
V_{\la}(r) = \frac{\cos 2Q_{\la} - 1}{\sinh^2 r} \le 0.
}
Some elementary computations using the definition of $Q_{\la}$ give us the explicit representation
\EQ{\label{Vla def}
V_{\la}(r) = \frac{-8 \la^2}{[(1+\lmb^{2}) \cosh r + (1-\lmb^{2})]^{2} }.
}
Below, we collect a few useful facts about $V_{\lmb}$.
\begin{lem} \label{Vla lem}
The following statements hold for $V_{\lmb}$.
\begin{enumerate}
\item We have
\begin{equation*}
V'_{\lmb} = \frac{16 \lmb^{2} (1+\lmb^{2}) \sinh r}{[(1+\lmb)^{2} \cosh r + (1-\lmb^{2})]^{3}}.
\end{equation*}

\item The potential $V_{\lmb}$ is \emph{attractive}. More precisely, $V_{\lmb}$ is always non-decreasing on $[0, \infty)$ and
\begin{equation*}
	V_{\lmb}(0) = - 2 \lmb^{2}, \quad
	\lim_{r \to \infty} V_{\lmb}(r) = 0.
\end{equation*}

\item For $\lmb = 0$, $V_{0} = 0$. For $\lmb = 1$,
\begin{equation*}
	V_{1} = - \frac{2}{\cosh^{2} r}.
\end{equation*}

\item For $0 \leq \lmb \leq 1$, we have
\begin{equation*}
	V_{\lmb} \geq V_{1}.
\end{equation*}
\end{enumerate}
\end{lem}
\begin{proof}
Statements (1)--(3) are trivial. To see why (4) holds, note that for $0 \leq \lmb \leq 1$,
\begin{align*}
	- V_{\lmb}
	\leq \frac{8\lmb^{2}}{(1+\lmb^{2})^{2} \cosh^{2} r}
%	\leq \frac{2(2\lmb)^{2}}{(1+\lmb^{2})^{2} \cosh^{2} r} 	
%	\leq \frac{2(1+\lmb^{2})^{2}}{(1+\lmb^{2})^{2} \cosh^{2} r} 	
	\leq - V_{1}.
\end{align*}
where we used $4 \lmb^{2} \leq (1+\lmb^{2})^{2}$.
\end{proof}

\subsection{Spectrum of $H_{V_{\lmb}}$ for small $\lmb$} \label{subsec:smallLmb}

We note that the spectrum for the self-adjoint operator $\LL_0$ is purely absolutely continuous and is given by $\s(\LL_0) = [1/4, \infty)$, and in particular there is no negative spectrum, no eigenvalue in the gap $[0, 1/4)$, and the threshold $1/4$ is neither an eigenvalue nor a resonance. The following result shows that in the case $0 < \la < \sqrt{15/8}$, the same can be said of the spectrum $ \s(\LL_{V_{\la}})$.

\begin{prop}\label{good spec} Let $0 \le \la< \sqrt{15/8}$. Then the spectrum for the self-adjoint operator $\LL_{V_{\la}}$  is purely absolutely continuous and given by
\EQ{
\s( \LL_{V_{\la}}) = [ 1/4, \infty).
}
In particular, there is no negative spectrum, there are no eigenvalues in the gap $[0, 1/4)$, and the threshold $\frac{1}{4}$ is neither an eigenvalue nor a resonance.
\end{prop}

Before we prove Proposition~\ref{good spec} we observe a few preliminary facts concerning solutions to
\EQ{\label{LL ev}
\LL_{V_{\la}} \phi =  \mu^2 \phi, \, \, \mfor \mu^2 \in \R, \, \,  \mu \in \C.
}
\begin{lem} \label{LL ev 0} Let $\mu \in \C, \, \, \mu^2 \in \R$ and suppose that $\phi_{\mu}$ is a solution to~\eqref{LL ev} such that $\phi_{\mu} \in L^2([0, c))$ for some $c>0$. Then, there exists $a \in \R$ so that
\EQ{
\phi_{\mu}(r) = a\,  r^{\frac{3}{2}} + o(r^{\frac{3}{2}})  \mas r \to 0.
}
\end{lem}
\begin{proof} This follows from the fact that the operator $\LL_0- 1/4$ is well approximated near $r=0$ by the singular operator
\ant{
L_0 := - \p_{rr} + \frac{3}{4 r^2}.
}
$L_0$ is in the limit point case at $r=0$ and a fundamental system for $L_0f = 0$ is given by $\{r^{\frac{3}{2}}, r^{-\frac{1}{2}}\}$. It follows that a solution $\phi_{\mu}$ as in Lemma~\ref{LL ev} can be written in terms of these  two solutions via the variation of parameters formula which converges for small $r$. The $L^2([0, c))$ requirement then guarantees that the coefficient in front of $r^{-\frac{1}{2}}$ must be $0$ and the leading order behavior is given by $r^{\frac{3}{2}}$.
\end{proof}
\begin{lem} \label{LL ev 1}Suppose $\phi_{0}$ is a solution to~\eqref{LL ev} with $\mu^2 = \frac{1}{4}$. Then there exist constants $a, b \in \R$ so that
\EQ{
\phi_0(r) = a + b \, r + O(re^{-2r}) \mas r \to \infty.
}
\end{lem}
\begin{proof} This follows from the fact that we can find constants $C_{\la}, C>0$ so that for $r$ large we have $V_{\la}(r) \le C_{\la} e^{-2r}$ and $\frac{3}{4 \sinh^2 r} \le C e^{-2r}$. Thus the operator
\ant{
L_{\infty} := -\p_{rr}
}
is a good approximation of  $\LL_{V_{\la}} - 1/4$ near $r= \infty$. A fundamental system for $L_{\infty} f = 0$ is given by $\{1, r\}$. The variation of parameters formula then yields the conclusions of Lemma~\ref{LL ev 1}.
\end{proof}

\begin{defn}\label{res def} Given the conclusions of Lemma~\ref{LL ev} and Lemma~\ref{LL ev 1} we can give a precise definition of  what we mean by threshold resonance. We say that that $\phi_0$ is a \emph{threshold resonance} for $\LL_{V_{\la}}$ if $\phi_0$ is not in $L^2(0,\infty)$ and it is a bounded solution to
\ant{
\LL_{V_{\la}} \phi_0 = \frac{1}{4} \phi_0.
}
In particular we can find non-zero $a, b \in \R$ so that
\EQ{
&\phi_0(r) = a r^{\frac{3}{2}} + o(r^{\frac{3}{2}}) \mas r \to 0,\\
&\phi_0(r) = b + O(r e^{-2r}) \mas r \to \infty.
}
\end{defn}

We can now prove Proposition~\ref{good spec}.

\begin{proof}[Proof of Proposition~\ref{good spec}]
Let $\mu \in \C$ with $\mu^2 \le \frac{1}{4}$. Suppose that $\phi_\mu$ is a solution to
\EQ{\label{phi m}
\LL_{V_{\la}} \phi_\mu = \mu^2 \phi_{\mu}
}
If  $\mu^2 \le 1/4$ is an eigenvalue, we can assume that it is the smallest eigenvalue, and by a variational principle, we can further assume that corresponding eigenfunction $\phi_{\mu} \in L^2$ is unique, (i.e., $\mu^2$ is simple) and strictly positive. If $\mu^2 = \frac{1}{4}$ and is not an eigenvalue, we  assume that $\phi_{\mu}$ is a
threshold resonance. In either case, we know by Lemma~\ref{e val} that $\phi_{\mu}(r) = O(r^{\frac{3}{2}})$ as $r \to 0$. If $\phi_{\mu}$ is an eigenvalue, then $\phi_{\mu}(r) \to 0$ as $r \to \infty$. If $\phi_{\mu}(r)$ is a threshold resonance, we know by Definition~\ref{res def} that $\phi_\mu(r) \to b >0$ as $r \to \infty$.

Now, define the operator
\EQ{
\KK := -\p_{rr} + \frac{3}{4 \sinh^2 r}. %- \frac{15}{4 \cosh^2 r}
}
Observe that the function
\EQ{
f(r)= \tanh^{\frac{3}{2}}r,
}
solves
\EQ{\label{Kf=}
\KK f = \frac{15}{4 \cosh^2 r} f.
}
Then, for any $R>0$ we can integrate by parts, using~\eqref{phi m} and~\eqref{Kf=}   to obtain
\begin{multline}\label{R equal}
 \left( \mu^2- \frac{1}{4}\right) \int_0^{R} \tanh^{\frac{3}{2}} r \, \phi_\mu(r) \, dr = - \phi'_{\mu}(R) \tanh^{\frac{3}{2}} R + \frac{3}{2}\phi_{\mu}(R)\frac{ \tanh^{\frac{1}{2}} R}{ \cosh^2 R}\\
  \quad  +\int_0^{R} \left( \frac{15}{4 \cosh^2 r} + V_{\la}(r) \right) \tanh^{\frac{3}{2}} r \, \phi_{\mu}(r) \, dr
\end{multline}
%previous version.
%For $0< \la< \sqrt{15/8}$, we can plug in the definition \eqref{Vla def} to see that for such a $\la$ fixed,  we have% there exist $ \de >0$ so that
%
%% Previous version
%
%Since we are assuming $0< \la< \sqrt{15/8}$, we can use part (2) of Lemma~\ref{Vla lem} to conclude that for such a $\la$ fixed, there exist $ \de >0$ so that
%\EQ{
%\frac{15}{4 \cosh^2 r} + V_{\la}(r) \ge \de>0,\, \quad \forall \, r \ge 0
%}
%Plugging this estimate into~\eqref{R equal}, and noting that $\frac{3}{2}\phi_{\mu}(R)\frac{ \tanh^{\frac{1}{2}} R}{ \cosh^2 R} \ge0$  for all $R>0$ we obtain
%\EQ{\label{R inequal}
%0>\left( \mu^2- \frac{1}{4}- \de\right) \int_0^{R} \tanh^{\frac{3}{2}} r \, \phi_\mu(r) \, dr \ge - \phi'_{\mu}(R) \tanh^{\frac{3}{2}} R
%}
%The left-hand side above is strictly negative and decreasing in $R$ and hence can be bounded away from $0$ from above by a fixed negative constant. On the other hand we know that $ - \phi'_{\mu}(R) \tanh^{\frac{3}{2}} R \to 0$ as $R \to \infty$, which means that~\eqref{R inequal} gives a contradiction for $R$ large enough. This completes the proof of Proposition~\ref{good spec}.
%For $0< \la< \sqrt{15/8}$, we can plug in the definition \eqref{Vla def} to see that for such a $\la$ fixed,  we have% there exist $ \de >0$ so that
%
%% Previous version
%
%Since we are assuming $0< \la< \sqrt{15/8}$, we can use part (2) of Lemma~\ref{Vla lem} to conclude that for such a $\la$ fixed, there exist $ \de >0$ so that
Since $\mu^2 \le \frac{1}{4}$ and  $\frac{3}{2}\phi_{\mu}(R)\frac{ \tanh^{\frac{1}{2}} R}{ \cosh^2 R} \ge0$  for all $R>0$, we can deduce that
\EQ{\label{pos<0}
\int_0^{R} \left( \frac{15}{4 \cosh^2 r} + V_{\la}(r) \right) \tanh^{\frac{3}{2}} r \, \phi_{\mu}(r) \, dr  \le \phi'_{\mu}(R) \tanh^{\frac{3}{2}} R
}
for all $R>0$. For $0< \la< \sqrt{15/8}$, we can plug in the definition \eqref{Vla def} to see that for such a $\la$ fixed,  we have
\EQ{
\frac{15}{4 \cosh^2 r} + V_{\la}(r) > 0, \quad \forall  r \in [0, \infty).%\de>0,\, \quad \forall \, r \ge 0
}
This means that the left-hand-side of~\eqref{pos<0} is strictly positive and increasing in $R$ and hence we can find $\de >0$ so that
\EQ{\label{contradiction}
0< \de \le \phi'_{\mu}(R) \tanh^{\frac{3}{2}} R
}
for all $R >0$.
%Plugging this estimate into~\eqref{R equal}, and noting that $\frac{3}{2}\phi_{\mu}(R)\frac{ \tanh^{\frac{1}{2}} R}{ \cosh^2 R} \ge0$  for all $R>0$ we obtain
%\EQ{\label{R inequal}
%0>\left( \mu^2- \frac{1}{4}- \de\right) \int_0^{R} \tanh^{\frac{3}{2}} r \, \phi_\mu(r) \, dr \ge - \phi'_{\mu}(R) \tanh^{\frac{3}{2}} R
%}
%The left-hand side above is strictly negative and decreasing in $R$ and hence can be bounded away from $0$ from above by a fixed negative constant.
However, we know that $ \phi'_{\mu}(R) \tanh^{\frac{3}{2}} R \to 0$ as $R \to \infty$, which means that~\eqref{contradiction} gives a contradiction for $R$ large enough. This completes the proof of Proposition~\ref{good spec}.
\end{proof}

\subsection{Spectrum of $H_{V_{\lmb}}$ for large $\lmb$: Beginning of the proof of Theorem \ref{e val}} \label{subsec:gapEoverview}
The rest of this section is devoted to the proof of Theorem \ref{e val} which asserts, in particular, existence of a unique simple gap eigenvalue $\gapE \in (0, 1/4)$ of $\calL_{V_{\lmb}}$ for large $\lmb$ and migration of $\gapE$ to $0$ as $\lmb$ tends to $\infty$. We remind the reader that $\calL_{V_{\lmb}}$ is $L^{2}$-equivalent to $H_{V_{\lmb}}$.

In this subsection, we begin our proof of Theorem \ref{e val} by establishing some elementary facts concerning the spectrum of $\calL_{V_{\lmb}}$ for all $\lmb \in [0, \infty)$. In particular, we prove that if an eigenvalue exists, then it must occur in the spectral gap $(0, 1/4)$. We also show that $\calL_{V_{\lmb}}$ has a threshold resonance when $\lmb = \lmb_{\sup}$ or $\lmb = \Lmb_{\inf}$, where $\lmb_{\sup}$ and $\Lmb_{\inf}$ were defined in Theorem \ref{e val}. Next, we briefly explain the idea of renormalization, which is key to the rest of our proof of Theorem \ref{e val}.  At the end of this subsection, we give an outline of the structure of the proof of Theorem \ref{e val}.

To begin, we state and prove some general facts about the spectrum of $\calL_{V_{\lmb}}$.
\begin{prop} \label{no neg spec}
The following statements concerning $\calL_{V_{\lmb}}$ hold.
\begin{enumerate}[(i)]
\item For every $\lmb \geq 0$, the spectrum of $\calL_{V_{\lmb}}$ does not contain any non-positive reals, i.e.,
\begin{equation*}
	\sgm(\calL_{V_{\lmb}}) \cap (-\infty, 0] = \emptyset.
\end{equation*}
\item There does not exist any eigenvalue in $[\frac{1}{4}, \infty)$.
\end{enumerate}
\end{prop}

As a consequence of this proposition, any eigenvalue of the operator $\calL_{V_{\lmb}}$ must occur in the \emph{spectral gap} $(0, 1/4)$.

In our proof of the first statement of Proposition \ref{no neg spec}, we make use of the positive solution $\ztz^{(\lmb)}$ to the equation
\begin{equation} \label{eq:eq4ztz}
	\calL_{V_{\lmb}} \ztz^{(\lmb)} = 0,
\end{equation}
which is obtained by differentiating $Q_{\lmb}$ with respect to $\lambda$ and conjugating by $\sinh^{1/2} r$. It can be computed explicitly to be
\begin{equation} \label{eq:ztz}
	\ztz^{(\lmb)} (r) = \frac{\tanh(r/2)}{1 + \lmb^{2} \tanh^{2}(r/2)}  \sinh^{1/2} r.
\end{equation}

The explicit solution $\ztz^{(\lmb)}$ (more precisely, its conjugate $\ztinf^{(\lmb)}$) will make another entrance in our proof of migration of the gap eigenvalue in Section \ref{subsec:gapEmig}.

\begin{proof} [Proof of Proposition \ref{no neg spec}] The existence of the solution $ \z_0^\la$ rules out the possibility of  an eigenvlaue at $\mu =0$. 
Therefore, to prove the first statement, it suffices to rule out eigenvalues in $(-\infty, 0)$.
Suppose that such an eigenvalue exists. Then, as in the proof of Proposition \ref{good spec}, there exists $\mu \in \bbC$ with $\mu^{2} \leq 0$ and an $L^{2}$ solution $\phi_{\mu}$ to \eqref{phi m} which is strictly positive. Proceeding as in \eqref{R equal} with $\ztz^{(\lmb)}$ in place of $\tanh^{3/2} r$, for any $R > 0$ we obtain
\begin{equation*}
	\mu^{2} \int_{0}^{R} \ztz^{(\lmb)}(r) \phi_{\mu}(r) \, \ud r = - \phi'_{\mu}(R) \ztz^{(\lmb)}(R) + \phi_{\mu}(R) (\ztz^{(\lmb)})'(R).
\end{equation*}

Arguing as in Proposition \ref{good spec}, we see that the left-hand side is strictly negative and decreasing in $R$. On the other hand, the right-hand side is non-negative for sufficiently large $R$, which is a contradiction.

The second statement follows from the fact that if $\mu^{2} \geq 1/4$ then there does not exist any non-zero solution to $\calL_{V_{\lmb}} \phi = \mu^{2} \phi$ in $L^{2}([1, \infty))$. To prove this fact, observe that $\calL_{V_{\lmb}} - \mu^{2}$ is well-approximated by $-\rd_{rr} - (\mu^{2} - 1/4)$, near $r = \infty$. Moreover, note that a fundamental system for $-\rd_{rr} f - (\mu^{2} - 1/4) f = 0$ is $\set{e^{\pm i \sqrt{\mu^{2} - 1/4} \, r}}$ when $\mu^{2} > 1/4$ and $\set{1, r}$ when $\mu^{2} = 1/4$, all of which do not decay as $r \to \infty$. \qedhere
\end{proof}

Next we show, roughly speaking, that the transition from a $\lmb$-regime with no eigenvalue and threshold resonance to a $\lmb$-regime with a gap eigenvalue must be accompanied by a threshold resonance.
\begin{prop} \label{res exists}
As in Theorem \ref{e val}, define
\begin{align*}
&\lambda_{\sup}=\sup\{ \la \mid \calL_{V_{\ti \lambda}} \, \, \textrm{has no e-vals and no threshold resonance} \, \, \forall \, \ti \la < \la\}\\
&\Lambda_{\inf}=\inf\{ \la \mid \calL_{V_{\ti \lambda}} \, \, \textrm{has a gap e-val}\, \,  \mu^2_{\ti{\la}} \in (0, 1/4) \, \, \forall \,  \ti \la > \la\}
\end{align*}
Then both $\calL_{V_{\lmb_{\sup}}}$ and $\calL_{V_{\Lmb_{\inf}}}$ have threshold resonances.
\end{prop}

In view of the result that we will prove Section \ref{subsec:gapEexists} (existence of a gap eigenvalue), we have $\lmb_{\sup} \leq \Lmb_{\inf} < \infty$.

\begin{proof}
To prove this proposition, we study the solution $\phi_{0}^{\lmb}$ to $\calL_{V_{\lmb}} \phi_{0}^{\lmb} = (1/4) \phi_{0}^{\lmb}$ such that $\phi_{0}^{\lmb}(r) = r^{3/2} + o(r^{3/2})$. By Sturm's oscillation theory and Proposition \ref{no neg spec} (which rules out negative spectrum), existence of an eigenvalue in $(0, 1/4)$ is equivalent to existence of a zero (i.e., a sign change) of $\phi_{0}^{\lmb}$. As a consequence, we have the following alternative characterization of $\lmb_{\sup}$ and $\Lmb_{\inf}$ in terms of $\phi_{0}^{\lmb}$:
\begin{align*}
&\lambda_{\sup}=\sup\{ \la \mid \phi_{0}^{\ti{\lmb}} \textrm{ does not change sign and is not a threshold resonance} \, \forall \, \ti \la < \la\}\\
&\Lambda_{\inf}=\inf\{ \la \mid \phi_{0}^{\ti{\lmb}} \textrm{ changes sign} \, \, \forall \,  \ti \la > \la\}
\end{align*}

Observe that $A)$ ``changing sign" and $B)$ ``not changing sign and not being a threshold resonance" are open conditions in $\lambda$ for $\phi_{0}^{\lmb}$. Indeed, that $A)$ is an open condition is an easy consequence of pointwise continuity of $\phi_{0}^{\lmb}(r)$ in $\lmb, r$. That $B)$ is an open condition follows from the fact that the coefficient $b = b(\lmb)$ from Lemma \ref{LL ev 1} is continuous in $\lmb$.  By the above characterization of $\lmb_{\sup}$ and $\Lmb_{\inf}$, the only remaining possibility is that $\phi_{0}^{\lmb}$ is a threshold resonance for $\lmb = \lmb_{\sup}$ or $\lmb = \Lmb_{\inf}$, which proves the proposition. \qedhere
\end{proof}

We now explain the idea of renormalization, which will play an important role in our arguments in the rest of this section. For each $\la > 0$, we define the rescaled Schr\"odinger operator $\rnL_{\lmb}$ by
\EQ{ \label{L def}
\rnL_{\lmb}  := -\rd_{\rho \rho} + \frac{3}{4} \frac{1}{ \la^2 \sinh^2 (\rho/\la)} + \frac{1}{4 \lmb^{2}} + \frac{1}{ \la^2}V_{\la}(\rho/ \la).
}
We will refer to $\rho = \lmb r$ as the \emph{renormalized coordinate}.
The operator $\rnL_{\lmb}$ is related to $\calL_{V_{{\lmb}}}$ as follows: Given a function $\phi(r)$ on $(0, \infty)$, define $\widetilde{\phi}(\rho) := \phi(\rho / \lmb)$. Then
\begin{equation*}
	\rnL_{\lmb} \widetilde{\phi}
		= \frac{1}{\lmb^{2}} (\calL_{V_{\lmb}} \phi)(\cdot / \lmb)
		= \frac{1}{\lmb^{2}} \widetilde{(\calL_{V_{\lmb}} \phi)}.
\end{equation*}

%\Green{$\slashed{This}~\rnL_{\lmb}$} is related to $\LL_{V_\lambda}$ by $$\frac{1}{\lmb^{2}} \LL_{V_{\la}} (\psi(\cdot/\lambda))=\rnL_\lambda(\psi(\cdot)).$$
%\Green{(We will refer to $\rho = \lmb r$ as the \emph{renormalized coordinate}). Remove. Also on the LHS above I think $\psi(\cdot/\lambda)$ should be $\psi(\lambda\cdot).$ }
A simple but important observation is that in the limit $\lambda\rightarrow\infty,~\rnL_\lambda$ formally tends to the operator
%%%%%%%%%%%%%%
\EQ{\label{re euc}
\LL_{\euc} \fy  :=& - \fy_{\rho \rho} + \frac{3}{4} \frac{1}{\rho^2}  \fy + V_{\euc}(\rho) \fy, \\
V_{\euc}(\rho):=& - \frac{2}{(1+ (\rho/2)^2)^2}.
}
%%%%%%%%%%%%%%%

The equation $\LL_{\euc} \varphi = 0$ possesses an explicit solution
\begin{equation} \label{eq:eucRes}
	\varphi_{0}(\rho) := \frac{\rho^{\frac{3}{2}}}{1 + (\rho/2)^{2}}.
\end{equation}

The Schr\"odinger operator $\LL_{\euc}$ arises in linearizing the co-rotational wave maps equation $\R^{2+1} \to \mathbb{S}^2$ around the ground state harmonic map $Q_{\euc}$. The explicit solution $\varphi_{0}$ is obtained from the scaling invariance of the problem, and is a resonance at zero of $\LL_{\euc}$. See \cite{KST} for more details.

The idea of renormalization is to exploit the formal resemblance of $\rnL_{\lmb}$ and $\calL_{\euc}$ by (essentially) working with a fundamental system for $\calL_{\euc} \varphi = 0$ consisting of the explicit solution $\varphi_{0}$ and its conjugate.
More precisely, given a solution $\psi$ to $\calL_{V_{\lmb}} \psi = \overline{\mu}^{2} \psi$, we define its \emph{renormalization} $g(\rho)$ by making a change of variable $\rho = \lmb r$ and dividing by $\varphi_{0}(\rho)$, i.e.,
\begin{equation}
	g(\rho) := \frac{\psi(\rho / \lmb)}{\varphi_{0}(\rho)}.
\end{equation}

Then $g(\rho)$ obeys the equation
\begin{equation*}
	(g' \varphi^{2}_{0})' = \varphi_{0}^{2} W_{\lmb, \overline{\mu}} g
\end{equation*}
where
\begin{equation*}
	W_{\lmb, \overline{\mu}}(\rho) := \frac{3}{4} \frac{1}{\lmb^{2} \sinh^{2}(\rho/\lmb)} - \frac{3}{4} \frac{1}{\rho^{2}} + \frac{1}{4 \lmb^{2}} - \frac{\overline{\mu}^{2}}{\lmb^{2}} + \frac{1}{\lmb^{2}} V_{\lmb}(\rho/\lmb) - V_{\euc}(\rho).
\end{equation*}

We call $W_{\lmb, \overline{\mu}}$ the \emph{renormalized potential}. By the same computation that shows $\rnL_{\lmb} \to \calL_{\euc}$, it follows that $W_{\lmb, \overline{\mu}}(\rho) \to 0$ for each $\rho>0$ as $\lmb \to \infty$. This simple fact already suggests that we have a good control on $g(\rho)$ for $0 \leq \rho \aleq 1$. This will be one of the main ideas of our proof of uniqueness of the gap eigenvalue.

Remarkably, the renormalization technique can be also made effective in a $\rho$-interval of the form $0 \leq \rho \aleq \lmb$, which can be made arbitrarily long, provided that an appropriate {\it a priori} estimate for $g$ holds. This observation is crucial in our proofs of existence and migration of the gap eigenvalue below, where we obtain an appropriate {\it a priori} estimate from a contradiction hypothesis.

We conclude this subsection with an outline of the structure of the proof of Theorem \ref{e val}. In Section \ref{subsec:gapEexists}, we establish existence of an eigenvalue of $\calL_{V_{\lmb}}$ in $(0, 1/4)$ for sufficiently large $\lmb$. Then in Section \ref{subsec:gapEunique}, we show that if an eigenvalue exists in $(0, 1/4)$, then it must be simple and unique if $\lmb$ is sufficiently large. We also rule out threshold resonance for large $\lmb$. Finally, in Section \ref{subsec:gapEmig}, we show that the gap eigenvalue $\gapE$ tends to $0$ as $\lmb \to \infty$. Combined with Propositions \ref{no neg spec} and \ref{res exists}, Theorem \ref{e val} then follows.

\subsection{Existence of gap eigenvalues for large $\la$} \label{subsec:gapEexists}
The goal of this subsection is to prove existence of an eigenvalue of $\calL_{V_{\lmb}}$ in the spectral gap $(0, 1/4)$ for sufficiently large $\lmb$. The main result of this subsection is the following proposition:
\begin{prop}\label{sign change} There exists $\Lambda_0>0$ with the following property: Let $\la \ge \Lambda_0$ and let $\phi_0$ be the solution to
\begin{equation}\label{orig L eq}
\calL_{V_{\lambda}} \phi_{0}= \frac{1}{4} \phi_{0}
\end{equation}
that satisfies $\phi_0(r) = r^{3/2} + o(r^{3/2}) \mas r \to 0.$ Then $\phi_0(r)$ changes sign on $[0, \infty)$ at least once.
\end{prop}

Existence of an eigenvalue $\mu^{2} < \frac{1}{4}$ is then a direct consequence of Sturm's oscillation theory. By Proposition \ref{no neg spec}, it follows furthermore that $\mu^{2} \in (0, 1/4)$.

Henceforth, we will work with the rescaled solution $\psi^{\lmb}_{0}(\rho) := \lmb^{3/2} \phi_{0}(\rho / \lmb)$ with $\lmb > 1$, in anticipation of the application of the renormalization technique. Then $\psi^{\lmb}_{0}$ solves the equation
\begin{equation} \label{L eq}
	\rnL_{\lmb} \psi^{\lmb}_{0} = \frac{1}{4 \lmb^{2}} \psi^{\lmb}_{0}.
\end{equation}
Moreover, $\psi^{\lmb}_{0}(\rho) = \rho^{3/2} + o(\rho^{3/2})$ as $\rho \to 0$.

Proposition~\ref{sign change} will be a consequence of the following two lemmas.
%First we establish existence of a particular solution $\psi_{\infty}^{\la}$ to~\eqref{L eq} with prescribed behavior at $r = \infty$.
\begin{lem}\label{psi I}
For every $ \la>1$, there exists a unique solution $\psi_{\infty}^{\la}$ to~\eqref{L eq} so that
\EQ{\label{psi I decay}
&\psi_{\infty}^{\la}(\rho) = 1 + O(e^{-2\rho/\la}) \mas \rho  \to \infty\\
&(\psi_{\infty}^{\la})'(\rho) = O(e^{-2\rho /\la}) \mas \rho \to \infty
}
Moreover, for sufficiently large $\lmb$, $\psi_{\infty}^{\lmb}(\rho)$ satisfies
\EQ{
\psi_{\infty}^{\la}(\rho)>0 \quad \forall \rho \in [\la, \infty)
}
\end{lem}

\begin{lem}\label{neg wr}
For sufficiently large $\lmb$, the following statement holds: Let $\psi_0^{\la}$ be the solution to~\eqref{L eq} that satisfies
\EQ{
\psi_0^{\la}(\rho) = \rho^{3/2} + o(\rho^{3/2}) \mas \rho \to 0.
}
Then, either $\psi_0^{\la}$ changes sign on $[0, \la]$, or
\EQ{\label{wronsk1}
\frac{(\psi_0^{\lmb})'(\la)}{ \psi_0^{\lmb}(\la)} < \frac{(\psi_{\infty}^{\lmb})'(\la)}{ \psi_{\infty}^{\lmb}(\la)}
}
where $\psi_{\infty}^{\la}$ is the function given by Lemma~\ref{psi I}.
\end{lem}
Before proving Lemma~\ref{psi I} and Lemma~\ref{neg wr}, we assume their conclusions and establish Proposition~\ref{sign change}.
\begin{proof}[Proof of Propositon~\ref{sign change}] Let $\psi_0:= \psi_0^{\la}$ be as in Lemma~\ref{neg wr}, and let $\psi_{\infty}:= \psi_{\infty}^{\la}$ be as in Lemma~\ref{psi I}.

Assume, for the sake of contradiction that $\psi_0$ does not change signs on $[0, \infty)$. Without loss of generality, we assume that $\psi_0(\rho) \ge 0$ for all $r \in [0, \infty)$.

Since there is no first order term in $\rnL_{\lmb}$, the Wronskian of $\psi_{0}$ and $\psi_{\infty}$ is constant.
%\ant{
%W[\psi_0, \psi_{\infty}](\rho) = \textrm{constant}.
%}
Hence we are free to evaluate it at any point $\rho \in[0, \infty)$. Since we are assuming that $\psi_0$ does not change its sign,~\eqref{wronsk1} gives
\EQ{\label{neg wronk}
W[\psi_0, \psi_{\infty}](\la) = \psi_0( \la) \psi_{\infty}'( \la) - \psi_0'(\la) \psi_{\infty}(\la) >0.
}
Due to the rapid decay as $\rho \to \infty$ of
\ant{
\frac{3}{4} \frac{1}{ \la^2 \sinh^2 (\rho/\la)} + \frac{1}{ \la^2}V_{\la}(\rho/ \la)
}
we can view any solution of~\eqref{L eq} as a perturbation of a solution to the free equation
\EQ{ \label{free}
-\rd_{\rho \rho} \psi = 0
}
at $\rho = \infty$. Since a fundamental system for~\eqref{free} is given by $\{1, \rho \}$, we can use a variation of parameters argument to find $a, b \in \R$ such that
\EQ{
&\psi_0(\rho) = a + b \rho + o(1) \mas \rho \to \infty,\\
&\psi_0'(\rho) \to b \mas \rho \to \infty.
}
Combining the above with~\eqref{psi I decay} we can deduce that
\EQ{
\lim_{\rho \to \infty}W[\psi_0,\psi_{\infty}](\rho) = -b.
}
By~\eqref{neg wronk} we can conclude that $b<0$. But this means that $\psi_{0}(\rho)<0$ for large $\rho$, which contradicts our assumption that $\psi_0(\rho)>0$ for all $\rho \in [0,\infty)$.
\end{proof}
%%%%%%%%%%%%%%

Lemma \ref{psi I} follows from a standard Volterra iteration argument combined with an easy observation concerning the sign of $(3/4 \sinh^{2} r) + V_{\lmb}(r)$.
\begin{proof} [Proof of Lemma \ref{psi I}]
Note that if $\phi_\infty$ is a solution of $\LL_{V_\la}\phi_\infty=(1/4) \phi_{\infty},$ then $\psi^\lambda_\infty(\cdot):=\phi_\infty(\cdot/\lambda)$  is a solution of $\rnL_\lambda\psi_\infty^\lambda=(1/4\lmb^{2}) \psi_{\infty}^{\lmb}.$ Hence it suffices to prove the lemma for $\phi_{\infty}$, where the positivity statement is now $\phi_\infty(r)>0$ for $r\geq1.$

Existence and uniqueness of $\phi_{\infty}$ can be proved by a standard iteration argument applied to the Volterra equation
%%%%%%%%%%
\begin{equation}\label{psi Volterra}
\phi_\infty(r)=1+\int_r^\infty(s-r)N(s)\phi_\infty(s)ds,
\end{equation}
%%%%%%%%%%%%
where
%%%%%%%%%%%%%
\[N(r):=\frac{3}{4\sinh^2r}+V_\lambda(r).\]
%%%%%%%%%%%%%

From the definition \eqref{Vla def} of $V_{\lmb}(r)$, it is not difficult to see that if $\lmb$ is sufficiently large, then $N(r) \geq 0$ for $r \geq 1$ (for a proof of a stronger statement, see Lemma \ref{positive potential}). Now it follows from a simple continuity argument that $\phi_{\infty}(r) \geq 1$ for $r \geq 1$, which proves the desired positivity statement. \qedhere
\end{proof}
We have thus reduced matters  to proving~Lemma~\ref{neg wr}. In the proof we will make use of the following estimates for the renormalized potential $W_{\lmb, 1/2}$. 
%%%%%%%%%%%%%%%%%%%%%%%%%%%%%
\begin{lem}\label{W lemma}
Let $W(\rho,\lambda) := W_{\lmb, \frac{1}{2}}(\rho)$, i.e.,
%%%%%%%%%%%%%%%%%%%%%%%%%%%%%
\begin{align*}
W(\rho,\lambda) = \frac{3}{4} \frac{1}{\lambda^2\sinh^2(\rho/\lambda)} + \frac{1}{\lambda^2} V_\lambda(\rho/\lambda)
				-\frac{3}{4\rho^2} - V_{\euc}(\rho).
\end{align*}
%%%%%%%%%%%%%%%%%%%%%%%%%%%%%
Then there exists $\Lambda_1$ such that the following hold.
\begin{enumerate}[(i)]
\item $\abs{\lambda^2W(\rho,\lambda)}$ is uniformly bounded for $\lambda>\Lambda_1$ and $\rho\leq\lambda.$
Moreover we can find $\rho_1<\Lambda_1$ such that if $\lambda>\Lambda_1$ then $\lmb^{2} W(\rho,\lambda)<-b$ in the region $\rho>\rho_1$, where $b>0$ is a constant independent of $\rho$ and $\lambda.$
%Moreover we can find $\rho_0<\rho_1<\Lambda_1$ such that if $\lambda>\Lambda_1$ then $W(\rho,\lambda)<-b$ in the regions $\rho\leq \rho_0$ and $\rho>\rho_1,$ where $b>0$ is a constant independent of $\rho$ and $\lambda.$
\item $W(\rho,\lambda)<0$ for $\rho\geq\lambda\geq\Lambda_1.$
\end{enumerate}
\end{lem}
%%%%%%%%%%%%%%%%%%%%%%%%%%%%%
\begin{proof}[Proof of Lemma \ref{W lemma}]
It is convenient to introduce the notation $\beta=\frac{1}{\lambda}.$ Then the region $\rho \leq\lambda$ corresponds to $\rho\beta\leq1.$ With this in mind we first note that
%%%%%%%%%%%%%%%%%%%%%%%%%%%%%%%%%%%%%%%%%%%%%%%%
\begin{equation}\label{first potential comparison}
\frac{3}{4\lambda^2\sinh^2(\rho/\lambda)}-\frac{3}{4\rho^2}=\frac{3}{4}\left(\frac{(\rho\beta)^2-\sinh^2(\rho\beta)}{(\rho\beta)^2 \sinh^2(\rho \be)}\right)\beta^2\leq-\frac{\beta^2}{10}
\end{equation}
%%%%%%%%%%%%%%%%%%%%%%%%%%%%%%%%%%%%%%%%%%%%%%%%%%%%%%
in the region $\rho\beta\leq1$. Next we write $\lambda^{-2}V_\lambda(\rho/\lambda)=\tilde{V}(\rho,\beta)$ where
%%%%%%%%%%%%%%%%%%%%%%%%%
\[\tilde{V}(\rho,\beta):=\frac{-8}{\left(\frac{\cosh(\rho\beta)-1}{\beta^2}+\cosh(\rho\beta)+1\right)^2}\]
%%%%%%%%%%%%%%%%%%%%%%%%
and note that $\lim_{\beta\rightarrow0}\tilde{V}(\rho,\beta)=V_{\euc}(\rho)$, where $V_{\euc}$ was defined in \eqref{re euc}. This, of course, is just a restatement of the fact that in the limit $\lambda\rightarrow\infty$ the hyperbolic potential approaches its Euclidean counterpart. It follows that
%%%%%%%%%%%%%%%%%%%%%%%%%%%%%%%%%%%%%%%%%
\begin{equation}\label{V difference}
V_{\euc}(\rho)-\frac{V_\lambda(\rho/\lambda)}{\lambda^2}=V_{\euc}(\rho)-\tilde{V}(\rho,\beta)=-\int_0^\beta\partial_\beta\tilde{V}(\rho,\tau)d\tau.
\end{equation}
%%%%%%%%%%%%%%%%%%%%%%%%%%%%%%%%%%%%%%%%%%%%
Now
%%%%%%%%%%%%%%%%%%%%%%%%%%%%%%%%%%%%%%%
\begin{equation}\label{V tilde derivative}
\partial_\beta\tilde{V}(\rho,\beta)=\frac{16\left(\left(\frac{(\rho\beta)\sinh(\rho\beta)-2(\cosh(\rho\beta)-1)}{(\rho\beta)^3}\right)\rho^3+\rho\sinh(\rho\beta)\right)}{\left(\left(\frac{\cosh(\rho\beta)-1}{(\rho\beta)^2}\right)\rho^2+\cosh(\rho\beta)+1\right)^3}.
\end{equation}
%%%%%%%%%%%%%%%%%%%%%%%%%%%%%%%%%%%%%%%%
But in the region $\rho\beta\leq1$
%%%%%%%%%%%%%%%%%%%%%%%%%%%%%%%%%%%%%
\begin{align*}
&\left|\frac{(\rho\beta)\sinh(\rho\beta)-2(\cosh(\rho\beta)-1)}{(\rho\beta)^3}\right|+|\sinh(\rho\beta)|\lesssim \rho\beta,\\
&\left(\frac{\cosh(\rho\beta)-1}{(\rho\beta)^2}\right)\gtrsim 1,
\end{align*}
%%%%%%%%%%%%%%%%%%%%%%%%%%%%%%%%%%%%%%%%
and therefore
%%%%%%%%%%%%%%%%%%%%%%%%%%%%%%%%%%%%%%%
\begin{align*}
|\partial_\beta\tilde{V}(\rho,\beta)|\lesssim(1+\rho^2)^{-3}(1+\rho^4)\beta\lesssim\frac{\beta}{(1+\rho^2)}.
\end{align*}
%%%%%%%%%%%%%%%%%%%%%%%%%%%%%%%%%%%%%%%
Inserting this into (\ref{V difference}) we get
%%%%%%%%%%%%%%%%%%%%%%%%%%%%%%%%%%%%%%%%
\begin{equation}\label{second potential comparison}
\left|V_{\euc}(\rho)-\frac{V_\lambda(\rho/\lambda)}{\lambda^2}\right|\lesssim\frac{\bt^{2}}{(1+\rho^2)}.
\end{equation}
%%%%%%%%%%%%%%%%%%%%%%%%%%%%%%%%%%%%%%
The conclusions of the first part of the lemma now follow from an inspection of (\ref{first potential comparison}) and (\ref{second potential comparison}).

The second part of the lemma concerns the region $\rho\beta \geq 1$. We first write
%%%%%%%%%%%%%%%%
\begin{equation}\label{third potential comparison}
\frac{3}{4\lambda^2\sinh^2(\rho/\lambda)}-\frac{3}{4\rho^2}=\frac{3}{4\rho^2}\left(\frac{(\rho\beta)^2-\sinh^2(\rho\beta)}{\sinh^2(\rho\beta)}\right)\leq-\frac{c}{\rho^2},
\end{equation}
%%%%%%%%%%%%%%%%
for some positive constant $c$ and for $\rho\beta\geq1.$ Next note that $x\sinh(x)>2(\cosh(x)-1)$ for all real $x$, and therefore by (\ref{V tilde derivative}) if $\rho\beta\geq1$ and $\lmb$ is sufficiently large,
%%%%%%%%%%%%%%%%
\begin{align*}
|\partial_\beta\tilde{V}(\rho,\beta)|&\lesssim\frac{\left(\frac{\sinh(\rho\beta)}{(\rho\beta)^2}\right)\rho^3+\rho\sinh(\rho\beta)}{\left(\frac{\cosh(\rho\beta)-1}{(\rho\beta)^2}\right)^3 \rho^6 +(\cosh(\rho\beta)+1)^3}
%\\	&\lesssim\frac{e^{\rho\beta}(\rho\beta)^{-2}\rho^3+\rho^2+e^{2\rho\beta}}{e^{3\rho\beta}(\rho\beta)^{-6}\rho^6 +e^{3\rho\beta}}
\lesssim\frac{1}{\rho^{3}}.
\end{align*}
%%%%%%%%%%%%%%%%%%
The second part of the lemma now follows from combining this estimate with (\ref{third potential comparison}) and (\ref{V difference}).
\end{proof}
%%%%%%%%%%%%%%%%%%%%%%%%%%%%%%%%%%%%%%%
%%%%%%%%%%%%%%%%%%%%%%%%%%%%%%
\begin{proof}[Proof of Lemma \ref{neg wr}]
For simplicity, we will write $\psi_{0} := \psi_{0}^{\lmb}$, $\psi_{\infty} := \psi_{\infty}^{\lmb}$ and $W(\rho, \lmb) := W_{\lmb, 1/2}(\rho)$. We divide the proof into two steps.
\begin{flushleft}\textbf{Step $1$:} The first step in the proof consists of establishing the following claim, which compares $\psi_{\infty}$ with the renormalized Euclidean resonance, $\fy_0$ at $\rho= \la$.\end{flushleft}
\begin{claim} For sufficiently large $\lmb$, we have
\EQ{\label{fy psi I}
\frac{ \fy_0'(\la)}{ \fy_0( \la)} < \frac{ \psi_{\I}'(\la)}{ \psi_{\I}(\la)}
}
\end{claim}
\begin{proof}
The proof follows from another comparison argument. %that we have been using several times in this paper.
Using that $\psi_{\infty}$ solves~\eqref{L eq} and $\fy_0$ solves~\eqref{re euc} we have
\ant{
\psi_{\infty}'(\la)& \fy_0(\la) - \psi_{\infty}(\la) \fy_0'(\la)
	= \int_{\la}^{\infty} \frac{d}{d\rho} \left( \psi_{\infty}(\rho) \fy_0'(\rho) - \psi_{\infty}'(\rho) \fy_0(\rho)  \right) \, d\rho\\
	&= \int_{\la}^{\infty} \psi_{\infty}(\rho) \fy_0''(\rho) - \psi_{\infty}''(\rho) \fy_0(\rho)   \, d\rho\\
	&= \int_{\la}^{\infty}\left[\frac{3}{4}\left( \frac{1}{\rho^2}  -\frac{1}{ \la^2 \sinh^2 (\rho/\la)}\right)  - V_{\euc}(\rho) +\frac{1}{ \la^2}V_{\la}(\rho/ \la)		\right]  \psi_{\infty}(\rho) \fy_0(\rho)   \, d\rho\\
	&= -\int_{\la}^{\infty}W(\rho, \la)  \psi_{\infty}(\rho) \fy_0(\rho)   \, d\rho.
}
Therefore,
\ant{
 \frac{ \psi_{\I}'(\la)}{ \psi_{\I}(\la)}-\frac{ \fy_0'(\la)}{ \fy_0( \la)}
 	= \frac{-1}{  \psi_{\infty}(\la) \fy_0(\la)}\int_{\la}^{\infty} W(\rho, \la)  \psi_{\infty}(\rho) \fy_0(\rho)   \, d\rho
}
Note that by Lemma~\ref{psi I} we have that $\psi_{\I}(\rho)>0$ for all $\rho \in [\la, \infty)$. Also, $\fy_0(\rho) \ge 0$ for all $\rho \ge 0$. Therefore, the claim follows from the second part of Lemma \ref{W lemma}
\end{proof}
%%%%%%%%%%%%%%%%%%%%
\begin{flushleft}\textbf{Step $2$:} In this second step we prove the following claim.
\end{flushleft}
\begin{claim}\label{euc inf comp} For sufficiently large $\lmb$, we have
\EQ{\label{fy psi 0}
\frac{ \psi_0'(\la)}{ \psi_0( \la)} < \frac{ \fy_{0}'(\la)}{ \fy_{0}(\la)}.
}
\end{claim}

\begin{proof}
We apply the renormalization technique introduced in Section \ref{subsec:gapEoverview}. For simplicity of notation we write $W(\rho) = W_{\lmb, 1/2}(\rho)$. We define $g$ by the relation $\psi_0(\rho)=g(\rho)\fy_0(\rho)$. Since by assumption $\psi_0$ does not change sign in $[0,\lambda]$ and since $\fy_0$ is positive there, we must have $g(\rho)>0$ for $\rho \in[0,\lambda].$ It follows that (\ref{fy psi 0}) is equivalent to
%%%%%%%%%%%%%%%%%%%%%%%%%%%%%%%%%%%%%%%
\begin{equation}{\label{g condition}}
g^\prime(\lambda)<0.
\end{equation}
%%%%%%%%%%%%%%%%%%%%%%%%%%%%%%%%%%%%%%%
Notice that $g$ satisfies the equation
%%%%%%%%%%%%%%%%%%%%%%%%%%%%%%%%%%%%%
\begin{align*}
\left(g^\prime\fy_0^2\right)^\prime=\fy_0^2Wg.
\end{align*}
%%%%%%%%%%%%%%%%%%%%%%%%%%%%%%%%%%%%%
Moreover, by our normalization $\psi_{0}(\rho) = \rho^{3/2} + o(\rho^{3/2})$, it follows that $(g, g')(0) = (1, 0)$. Therefore
%%%%%%%%%%%%%%%%%%%%%%%%%%%%%%%%%%%%%%%%%%%%
\begin{align}
&g^\prime(\rho)=\frac{1}{\fy^2_0(\rho)}\int_0^\rho\fy_0^2(\sgm)W(\sgm)g(\sgm) \,d\sgm, \label{g prime rep}\\
&g(\rho)=1+\int_0^\rho \int_0^\tau \frac{\fy_0^2(\sgm)}{\fy_0^2(\tau)}W(\sgm)g(\sgm) \, d\sgm d\tau. \label{g rep}
\end{align}
%%%%%%%%%%%%%%%%%%%%%%%%%%%%%%%%%%%%%%%%%%%
Assume by contradiction that (\ref{g condition}) does not hold, or in other words
%%%%%%%%%%%%%%%%%%%%%%%%%%%
\begin{equation}\label{g contradiction}
g^\prime(\lambda)\geq0.
\end{equation}
%%%%%%%%%%%%%%%%%%%%%%%%%%%%%%%%%%%%%%%%%
Take $\lmb$ large enough so that Lemma \ref{W lemma} applies. Then in view of the representation \eqref{g prime rep} and the fact that $W(\rho)$ is negative for $\rho_1\leq \rho \leq\lambda$ we must have $g^\prime(\rho)\geq0$ for $\rho \in [\rho_1,\lambda]$.

To derive the desired contradiction we begin by showing that $g$ is bounded away from zero. According to the observation above, $g$ can decrease only on the interval $[0, \rho_1].$
From Lemma \ref{W lemma},  we see that $\sup_{\sgm \in [0, \rho_{1}]} \abs{W(\sgm)} \to 0$ as $\lmb \to \infty$. Recalling the definition of $\varphi_{0}$ in \eqref{eq:eucRes}, it follows (by carrying out an explicit integration) that
\begin{equation*}
	\int_{0}^{\rho_{1}} \bb(\int_{\sgm}^{\rho_{1}} \frac{1}{\varphi_{0}^{2}(\tau)} \, \ud \tau \bb) \varphi_{0}^{2}(\sgm) \abs{W(\sgm)} \, \ud \sgm
	\aleq_{\rho_{1}} \sup_{\sgm \in [0, \rho_{1}]} \abs{W(\sgm)} \to 0
\end{equation*}
as $\lmb \to \infty$. By a Volterra-type iteration argument, we conclude that
\begin{equation*}
	\sup_{\rho \in [0, \rho_{1}]} \abs{g(\rho) - 1} = o(1) \mas \lmb \to \infty.
\end{equation*}
%
%% Previous version:
%
%Moreover since $\rho_1$ is independent of $\lambda$ the representations \eqref{g rep}) and \eqref{g prime rep} imply that $|g(\rho)|\leq c_{\rho_1}$ for $\rho \in[0, \rho_1],$ where $c_{\rho_1}$ does not depend on $\lambda.$ Using (\ref{g rep}) and Lemma (\ref{W lemma}), for any $r\leq \rho_1$ we have
%%%%%%%%%%%%%%%%%%%%%%%%%%%%%%%%%%
%\begin{align*}
%|g(\rho)-1|\leq c_{\rho_1}\int_0^{\rho_1}\int_0^\tau \frac{\fy_0^2(\sgm)}{\fy_0^2(\tau)}|W(\sgm)|d\sgm d\tau\lesssim\frac{1}{\lambda^2}.
%\end{align*}
%%%%%%%%%%%%%%%%%%%%%%%%%%%%%%%%%%%%
Taking $\lambda$ larger if necessary we can guarantee that $g(\rho) \geq 1/2$ for $\rho \leq \rho_1.$ Since $g^\prime(\rho)\geq0$ for $\rho\in [\rho_{1},\lambda]$ we have a global bound $g(\rho)\geq1/2$ on $\rho\leq\lambda.$ It follows that
%%%%%%%%%%%%%%%%%%%%%%%%%%%%%%%
\begin{align*}
\int_{\rho_1}^\lambda\fy_0^2(\sgm)W(\sgm)g(\sgm) \, d\sgm
&\leq- \frac{b}{2 \lambda^2}\int_{\rho_1}^\lambda\fy_0^2(\sgm)d\sgm\\
&\leq-\frac{C_1 b \log\lambda}{\lambda^2}
\end{align*}
%%%%%%%%%%%%%%%%%%%%%%%%%%%%%%%%%%%
for some universal constant $C_1$ independent of $\lmb$. On the other hand
%%%%%%%%%%%%%%%%%%%%%%%%%%%%%%%%%%%%
\begin{align*}
\int_0^{\rho_1}\fy_0^2(\sgm)|W(\sgm)|g(\sgm) \, d\sgm \leq \frac{C_2}{\lambda^2}
\end{align*}
%%%%%%%%%%%%%%%%%%%%%%%%%%%%%%%%%%%%%%%
for another universal constant $C_2$ also independent of $\lmb$. Inserting the last two estimates into the representation (\ref{g prime rep}) we conclude that if $\lambda$ is sufficiently large, $g^\prime(\lambda)<0$ contradicting (\ref{g contradiction}).
%%%%%%%%%%%%%%%%%%%%%%%%%%%%%%%%%%%%%%%
\end{proof}
Lemma~\ref{neg wr} now follows from combining the conclusions of Steps 1 and 2.
\end{proof}
%%%%%%%%%%%%%%%%%%%%%%%%%%%%%%%%%%%%%%%%%
%%%%%%%%%%%%%%%%%%%%%%%%%%%%%%%%%%%%%%%%%
\subsection{Uniqueness of gap eigenvalues for $H_{V_{\la}}$ for large $\la$} \label{subsec:gapEunique}
Our next goal is to prove that for large $\lambda$, the eigenvalue found in the previous section is simple and unique, and moreover that $\calL_{V_{\lmb}}$ does not have a threshold resonance at $1/4$. This will be accomplished by showing that eigenfunctions in the spectral gap and threshold resonances cannot change sign.

As before we need to treat the case of large and small $r$ separately. We begin with the following technical lemma.
%%%%%%%%%%%%%%%%%%%%%%
\begin{lem}\label{positive potential}
For $\lambda$ sufficiently large, there is a constant $C$ independent of $\lambda$ such that for $r\geq \frac{C}{\lambda}$
\[\frac{3}{4\sinh^2 r}+V_\lambda(r)\geq0.\]
\end{lem}
%%%%%%%%%%%%%%%%%%%%%%
\begin{proof}
Note that since
\[|V_\lambda(r)|\leq \frac{8}{\lambda^2(\cosh r-1)^2}\]
it suffices to show that for $r\geq C\lambda$
\[\frac{3}{4\sinh^2r}-\frac{8}{\lambda^2(\cosh r-1)^2}\geq0.\]
Writing $\sinh^{2} r = \cosh^{2} r - 1$, we see that this is equivalent to
\begin{equation*}
	\bb( \frac{3}{4} - \frac{8}{\lmb^{2}} \bb) \cosh r \geq \frac{3}{4} + \frac{8}{\lambda^2}.
\end{equation*}
%
%% Note: Jusitifcation
%
%\begin{align*}
%&\frac{3}{4\sinh^2r}-\frac{8}{\lambda^2(\cosh r-1)^2} \geq 0 \\
%&\impmi \frac{3}{4(\cosh r -1)(\cosh r + 1)} \geq \frac{8}{\lambda^2(\cosh r-1)^2} \\
%&\impmi \frac{3}{4} (\cosh r - 1) \geq \frac{8}{\lambda^2} (\cosh r + 1) \\
%&\impmi \bb( \frac{3}{4} - \frac{8}{\lmb^{2}} \bb) \cosh r \geq \frac{3}{4} + \frac{8}{\lambda^2} \\
%&\impmi \cosh r \geq \frac{3/4 + 8/\lmb^{2}}{3/4 - 8/\lmb^{2}}.
%\end{align*}

Assume that $\lmb$ is large enough so that $3/4 - 8/\lmb^{2} > 1/2$. Then from the elementary fact that $\cosh r \geq 1 + (1/2)r^{2}$, the preceding inequality holds for $r \geq C / \lambda$ with an absolute constant $C > 0$. \qedhere
\end{proof}
%%%%%%%%%%%%%%%%%%%%%
We can now carry out the analysis for large $r$.
%Assume that $\lambda$ is so large that the results of the previous section guarantee the existence of an eigenvalue.
%%%%%%%%%%%%%%%%%%
\begin{lem}\label{eval uniqueness r-infty}
Let $\phi$ be either an eigenfunction (i.e., a nonzero $L^{2}$ solution) of $\LL_{V_\lambda} \phi =\mu^2 \phi$ with $\mu^{2} \in(0,\frac{1}{4})$ or a threshold resonance at $\mu^{2} = 1/4$. Let $C$ be as in Lemma \ref{positive potential}. Then for sufficiently large $\lmb$, $\phi$ cannot change sign in the region $r\in (\frac{C}{\lambda},\infty).$
\end{lem}
%%%%%%%%%%%%%%%%%%%
\begin{proof}
We begin with the case of an eigenvalue $\mu^{2} \in (0, \frac{1}{4})$.
Define $m > 0$ by $m^{2} = \frac{1}{4} - \mu^{2}$. The idea is to compare $\phi$ with $f=e^{-m r}$, which up to scaling is the unique nonzero $L^2$ solution of $\partial_{rr} f=m^2 f$. As usual, after suitable renormalization we may assume that $\phi (r)=  e^{-m r} + o(e^{-m r})$ as $r\rightarrow\infty.$ Defining
\[W(r) : = W[\phi,f](r)=\phi(r) f^\prime(r)-\phi^\prime(r) f(r),\]
we have
\begin{equation*}
	W'(r)
%	= \phi(r) f''(r) - \phi''(r) f(r)
%	= (\frac{1}{4} - \mu^{2} - (\frac{3}{4 \sinh^{2} r} + V_{\lmb} r + \frac{1}{4} -\mu^{2} )) \phi(r) f(r)
	= - \bb(\frac{3}{4 \sinh^{2} r} + V_{\lmb}(r) \bb) \phi(r) f(r).
\end{equation*}
Therefore, in view of Lemma \ref{positive potential}, we see that $W^\prime(r) \leq 0$ for $r \geq C/\lambda$ and so long as $\phi$ is positive (note that $f > 0$ everywhere). Now let $R$ denote the largest zero of $\phi$ and for contradiction assume $R\geq C/\lambda$. Then $W^\prime(r)<0$ and $\phi \sim e^{-m r}$ as $r\rightarrow\infty$ imply that $W(R) \geq 0$. This means that
\[\lim_{r\rightarrow R^{+}}\frac{f^\prime(r)}{f(r)}\geq\lim_{r\rightarrow R^{+}}\frac{\phi^\prime(r)}{\phi(r)}=\infty,\]
and therefore we must have $f(R)=0$ which is impossible.

In the case of a threshold resonance, we compare $\phi$ with $g = 1$, which up to scaling is the unique nonzero bounded solution of $\rd_{rr} g = 0$. We omit the details, which are very similar to the previous case. \qedhere
\end{proof}
%%%%%%%%%%%%%%%%%%
Our task now is to show that $\phi$ as in Lemma \ref{eval uniqueness r-infty} does not change sign in the interval $r\leq C / \lambda$.
For this purpose we use the technique of renormalization.
%For this we will use the rescaled operator $\LL_\lambda$ from the previous section.
%%%%%%%%%%%%%%%%%%%
\begin{lem}\label{eval uniqueness r-zero}
Let $\phi$ be either an eigenfunction (i.e., a nonzero $L^{2}$ solution) of $\LL_{V_\lambda} \phi =\mu^2 \phi$ with $\mu^{2} \in(0,\frac{1}{4})$ or a threshold resonance at $\mu^{2} = 1/4$. Let $C$ be as in Lemma \ref{positive potential}. Then for sufficiently large $\lmb$, $\phi$ cannot change sign in the region $r\in(0, \frac{C}{\lambda}]$.
\end{lem}
%%%%%%%%%%%%%%%%%%%%%
\begin{proof}
We work with the rescaled operator $\rnL_\lambda$. It suffices to show that if $\psi(\rho)$ is a solution of $\rnL_\lambda \psi=(\mu^{2}/ \lmb^{2}) \psi$ in $L^2((0, C])$, then $\psi(\rho)$ does not change sign in the region $0 \leq \rho \leq C$. Arguing as in Lemma \ref{LL ev 0}, we see that any $L^{2}((0, C])$ solution of $\rnL_{\lmb} \psi = (\mu^{2}/\lmb^{2}) \psi$, after suitable normalization, has the behavior $\psi(\rho) = \rho^{3/2} + o(\rho^{3/2})$ and $\psi'(\rho) = (3/2) \rho^{1/2} + o(\rho^{1/2})$ as $\rho \to 0$.

Define $g(\rho) := \psi(\rho) / \varphi_{0}(\rho)$, where $\varphi_{0}$ is the Euclidean resonance defined in \eqref{eq:eucRes}. Since $\fy_0$ is always positive, we need to show that $g$ is bounded away from zero in the region $0 \leq r \leq C$. Recall from Section \ref{subsec:gapEoverview} that $g$ satisfies the equation
\begin{equation*}
	(g' \varphi_{0}^{2})' = \varphi_{0}^{2} W_{\lmb, \mu} g.
\end{equation*}
Note furthermore that we have $(g, g')(0) = (1,0)$, thanks to our normalization of $\psi$. Therefore, we have the integral formula
\begin{align}
%&g^\prime(\rho)=\frac{1}{\fy_0^2(\rho)}\int_0^\rho \fy_0^2(\sgm) W_{\lmb, \mu} (\sgm) g(\sgm) \, d\sgm,\label{new g prime eq}\\
&g(\rho)=1+\int_0^\rho \int_0^\tau \frac{\fy_0^2(\sgm)}{\fy_0^2 (\tau)} W_{\lmb, \mu}(\sgm) g(\sgm) \, d\sgm d\tau.\label{new g eq}
\end{align}

By a Volterra-type iteration argument as in the proof of Lemma \ref{neg wr} (see the proof of Claim \ref{euc inf comp}), we see that
\begin{equation*}
	\sup_{\rho \in [0, C]} \abs{g(\rho) - 1} = o(1) \mas \lmb \to \infty,
\end{equation*}
from which the lemma follows. \qedhere
%Observe that $W_{\lmb, \mu}(\sgm) \to 0$ for each $\sgm$ as $\lmb \to \infty$. Recalling the definition of $\varphi_{0}$ in \eqref{eq:eucRes}, it follows (by carrying out an explicit integration) that
%\begin{equation*}
%	\int_{0}^{C} \bb(\int_{\sgm}^{C} \frac{1}{\varphi_{0}^{2}(\tau)} \, \ud \tau \bb) \varphi_{0}^{2}(\sgm) \abs{W_{\lmb, \mu}(\sgm)} \, \ud \sgm
%	\aleq_{C} \sup_{\sgm \in [0, C]} \abs{W_{\lmb, \mu}(\sgm)} \to 0
%\end{equation*}
%as $\lmb \to \infty$. By a Volterra-type iteration argument, we conclude that
%\begin{equation*}
%	\sup_{\rho \in [0, C]} \abs{g(\rho) - 1} = o(1) \mas \lmb \to \infty,
%\end{equation*}
%from which the lemma follows. \qedhere
\end{proof}
%%%%%%%%%%%%%%%%%%%%%%%%%
\begin{prop}\label{eval uniqueness}
If $\lambda$ is sufficiently large, then $\calL_{V_{\lambda}}$ has a unique simple eigenvalue in $(0,\frac{1}{4})$, with no threshold resonance at $1/4$.
\end{prop}
%%%%%%%%%%%%%%%%%%
\begin{proof}
Existence was seen in the previous subsection, so it suffices to establish uniqueness and simpleness.
Let $\phi$ be an eigenfunction corresponding to any eigenvalue $\mu^{2} \in (0, \frac{1}{4})$, or a threshold resonance at $\mu^{2} = \frac{1}{4}$. Combining Lemmas \ref{eval uniqueness r-infty} and \ref{eval uniqueness r-zero}, it follows that $\phi$ cannot change sign on $(0, \infty)$. It follows by the variational principle that there is no eigenvalue below $\mu^{2}$, and when $\mu^{2} \in (0,\frac{1}{4})$ is an eigenvalue it is simple. The proposition follows. \qedhere
\end{proof}
\subsection{Migration of the gap eigenvalue} \label{subsec:gapEmig}
In this subsection, we conclude the proof of Theorem \ref{e val} by demonstrating that the gap eigenvalue $\gapE$ approaches $0$ as $\lmb \to \infty$.
By Sturm's oscillation theory, Proposition \ref{no neg spec} and the uniqueness of the gap eigenvalue, it suffices to establish the following proposition:
\begin{prop} \label{prop:gapEmig:main}
Let $\overline{\mu}^{2} \in (0, 1/4]$. Then for $\lmb$ sufficiently large (depending on $\overline{\mu}^{2}$), the solution $\phi_{0}$ to the ODE
\begin{equation*}
\left\{
\begin{aligned}
	\calL_{V_{\lmb}} \phi_{0} =& \overline{\mu}^{2} \phi_{0} \\
	\phi_{0} =& r^{3/2} + o(r^{3/2}) \hbox{ as } r \to 0
\end{aligned}
\right.
\end{equation*}
must change sign.
\end{prop}

We will prove Proposition \ref{prop:gapEmig:main} by a contradiction argument, which is similar in spirit to the proof of Proposition \ref{sign change}.
The key additional idea is to use $\ztinf^{(\lmb)}$, which is the solution to the problem
\begin{equation} \label{eq:gapEmig:eq4ztinf}
\left\{
\begin{aligned}
	\calL_{V_{\lmb}} \zt_{\infty}^{(\lmb)} = &0,\\
	\zt_{\infty}^{(\lmb)} \aeq& c_{\lmb} e^{-r/2} \hbox{ as } r \to \infty,
\end{aligned}
\right.
\end{equation}
for an appropriate $c_{\lmb} > 0$. An interesting feature of $\ztinf^{(\lmb)}$ is that it is a conjugate solution to the explicit solution $\ztz^{(\lmb)}$ used in the proof of Proposition \ref{no neg spec}. By standard ODE theory, it follows that $\ztinf^{(\lmb)}$ can also be explicitly determined.

We now briefly explain why $\ztinf^{(\lmb)}$ is useful for proving Proposition \ref{prop:gapEmig:main}. By a comparison argument, the contradiction hypothesis $\phi_{0} > 0$ leads to a lower bound for $\phi_{0}$ in terms of $\ztinf^{(\lmb)}$, i.e.,
\begin{equation*}
	\phi_{0}(r) \geq \frac{\phi_{0}(r_{0})}{\zt_{\infty}(r_{0})} \zt_{\infty}(r) \qquad \hbox{ for } 0 < r_{0} \leq r.
\end{equation*}
(For details, see the proof of \eqref{eq:gapEmig:compareZeta} below.) Thanks to the explicit expression for $\ztinf^{(\lmb)}$, we are able to derive from this inequality a uniform lower bound for $\phi_{0}$ in an $r$-interval of length $\simeq 1$. In the renormalized coordinate, this lower bound holds on a $\rho$-interval of length $\simeq \lmb$, which can be made arbitrarily large. This gives enough `time' for the renormalized potential (which is \emph{negative} since $\overline{\mu}^{2} > 0$) to force a sign change of $\phi_{0}$, which is a contradiction.

\begin{rem}
In fact, our proof of Proposition \ref{prop:gapEmig:main} does not depend on Proposition \ref{sign change}, and therefore furnishes an alternative proof of existence of a gap eigenvalue. We have nevertheless elected to include both proofs in this paper, since the proof of Proposition \ref{sign change} presented in Section \ref{subsec:gapEexists} requires a weaker hypothesis (in particular, there is no need for the knowledge of the explicit solution $\ztinf^{(\lmb)}$) and therefore might be of independent interest.
\end{rem}

Some lemmas needed for proving Proposition \ref{prop:gapEmig:main} are in order. The first lemma consists of an upper and lower bound on the explicit solution $\ztinf^{(\lmb)}$.
\begin{lem} \label{lem:gapEmig:est4ztinf}
There exist $\eps_{1} > 0$ such that
\begin{equation} \label{eq:gapEmig:est4ztinf}
	\ztinf^{(\lmb)}(r) \simeq \lmb^{1/2} \bb( \lmb^{2} + \frac{1}{\lmb^{2} r^{2}} \bb) \frac{(\lmb r)^{3/2}}{1+\lmb^{2} r^{2}} \qquad
	\hbox{ for } 0 \leq r \leq \eps_{1}, \, \lmb > 0,
\end{equation}
%\begin{equation}
%	c_{1} \bb( \lmb^{2} + \frac{1}{\lmb^{2} r^{2}} \bb) \frac{(\lmb r)^{3/2}}{1+\lmb^{2} r^{2}}
%	\leq \ztinf^{(\lmb)}(r)
%	\leq c_{2} \bb( \lmb^{2} + \frac{1}{\lmb^{2} r^{2}} \bb) \frac{(\lmb r)^{3/2}}{1+\lmb^{2} r^{2}},
%\end{equation}
where the implicit constants are independent of $r$ and $\lmb$.
\end{lem}
\begin{proof}
We begin by computing $\ztinf^{(\lmb)}$ explicitly. For simplicity, we will omit writing the superscript $(\lmb)$.

Since $\ztz$ and $\ztinf$ solve the same equation (with no first order term), their Wronskian is constant. We choose $c_{\lmb}$ in the definition of $\ztinf$ so that
\begin{equation*}
W[\ztz, \ztinf]
	= \ztz \ztinf' - \ztz' \ztinf
	= - 1
\end{equation*}

Dividing by $\ztz^{2}$, we have
\begin{equation*}
	\bb( \frac{\ztinf}{\ztz} \bb)' = -\frac{1}{\ztz^{2}}.
\end{equation*}
Because of the vanishing condition as $r \to \infty$, it follows that
\begin{align*}
	\ztinf (r)
	=& \bb( \int_{r}^{\infty} \frac{1}{\ztz^{2}(s)} \, \ud s \bb) \ztz (r) \\
	=& \bb( \int_{r}^{\infty} \frac{1}{\ztz^{2}(s)} \, \ud s \bb) \frac{\tanh(r/2)}{1+\lmb^{2} \tanh^{2}(r/2)} \sinh^{1/2} r.
\end{align*}

We now compute the $s$-integral. Using the identity
\begin{equation*}
	\sinh s = 2 \cosh (s/2) \sinh (s/2) = \frac{ 2 \tanh(s/2)}{1- \tanh^{2}(s/2)},
\end{equation*}
we see that
\begin{equation*}
	\frac{1}{\ztz^{2}(s)} = \frac{(1 + \lmb^{2} \tanh^{2}(s/2))^{2} (1 - \tanh^{2}(s/2))}{2 \tanh^{3}(r/2)}
\end{equation*}

Therefore, making a change of variables $u = \tanh(s/2)$,
%for which
%\begin{equation*}
%\ud u = \frac{1}{2} ( 1- \tanh^{2}(s/2) ) \, \ud s,
%\end{equation*}
we have
\begin{align}
	& \int_{r}^{\infty} \frac{1}{\ztz^{2}(s)} \, \ud s \notag \\
	& \quad = \int_{\tanh (r/2)}^{1} \frac{(1 + \lmb^{2} u^{2})^{2}}{u^{3}} \, \ud u \label{eq:gapEmig:ztinf:sIntegral}\\
%	& \quad = \int_{\tanh (r/2)}^{1} u^{-3} + 2 \lmb^{2} u^{-1} + \lmb^{4} u \, \ud u \notag \\
	& \quad = - \frac{1}{2} (1 - \tanh^{-2} (r/2) ) - 2 \lmb^{2} \log \tanh (r/2) + \frac{\lmb^{4}}{2} (1 - \tanh^{2}(r/2)). \notag
\end{align}

%Therefore, we obtain
%\begin{equation} \label{eq:gapEmig:ztinf}
%\begin{aligned}
%	\ztinf
%	= & c \lmb^{2} \bb( - \frac{1}{2 \lmb^{2}} (1 - \tanh^{-2}(r/2)) + 2 (\log \lmb - \log (\lmb \tanh) (r/2)) + \frac{\lmb^{2}}{2} (1-\tanh^{2}(r/2)) \bb) \\
%	& \quad \times \frac{\tanh(r/2)}{1+\lmb^{2} \tanh^{2}(r/2)} \sinh^{1/2} r.
%\end{aligned}
%\end{equation}

To prove \eqref{eq:gapEmig:est4ztinf}, it suffices to establish
\begin{align}
	\frac{\tanh(r/2)}{1+\lmb^{2} \tanh^{2}(r/2)} \sinh^{1/2} r
	\simeq & \frac{r^{3/2}}{1 + \lmb^{2} r^{2}}, 	\label{eq:gapEmig:est4ztinf:easy1}\\
	\int_{r}^{\infty} \frac{1}{\ztz^{2}(s)} \ud s
	\simeq & \lmb^{2} \bb( \lmb^{2} + \frac{1}{\lmb^{2}r^{2}} \bb),  \label{eq:gapEmig:est4ztinf:easy2}
\end{align}
for $0 \leq r \leq \eps_{1}$, where $\eps_{1} > 0$ is to be chosen below.

Estimate \eqref{eq:gapEmig:est4ztinf:easy1} is an immediate consequence of the easier estimates
\begin{equation*}
	\tanh (r/2) \simeq r, \quad
	\sinh r \simeq r, \quad
	1 + \lmb \tanh (r/2) \simeq  1+ \lmb r.
\end{equation*}
which holds when $\eps_{1} > 0$ is sufficiently small. On the other hand, \eqref{eq:gapEmig:est4ztinf:easy2} is easy to prove (taking $\eps_{1} > 0$ smaller if necessary) by directly estimating the integral \eqref{eq:gapEmig:ztinf:sIntegral}, whose integrand is positive. We leave the details to the reader. \qedhere
\end{proof}

In the following lemma, we collect useful facts for estimating the renormalized potential $W_{\lmb, \overline{\mu}}$.
\begin{lem} \label{lem:gapEmig:est4ptnl}
For $0 \leq \rho \leq  \lmb$ we have
\begin{gather}
	\bb\vert \frac{3}{4} \frac{1}{\lmb^{2} \sinh^{2} (\rho/\lmb)} - \frac{3}{4} \frac{1}{\rho^{2}}  + \frac{1}{4 \lmb^{2}} - \frac{\overline{\mu}^{2}}{\lmb^{2}} \bb\vert
	\aleq \frac{1}{\lmb^{2}}, \label{eq:gapEmig:est4ptnl1} \\
	\bb\vert \frac{1}{\lmb^{2}} V_{\lmb}(\rho / \lmb) - V_{\euc}(\rho) \bb\vert
	\aleq \frac{1}{\lmb^{2}} \frac{1}{(1+\rho^{2})}. \label{eq:gapEmig:est4ptnl2}
\end{gather}

Moreover, there exists $\eps_{2} = \eps_{2}(\overline{\mu}^{2}) > 0$, which are independent of $\rho$ and $\lmb$, such that for $0 \leq \rho \leq \eps_{2} \lmb$ we have
\begin{equation} \label{eq:gapEmig:est4ptnl3}
	\frac{3}{4} \frac{1}{\lmb^{2} \sinh^{2} (\rho/\lmb)} - \frac{3}{4} \frac{1}{\rho^{2}}  + \frac{1}{4 \lmb^{2}} - \frac{\overline{\mu}^{2}}{\lmb^{2}}
	\leq - \frac{\overline{\mu}^{2}}{2 \lmb^{2}}.
\end{equation}
\end{lem}
\begin{proof}
Estimate \eqref{eq:gapEmig:est4ptnl1} follows from \eqref{first potential comparison} and estimate \eqref{eq:gapEmig:est4ptnl2} is exactly \eqref{second potential comparison} in the proof of Lemma \ref{W lemma}.
To prove \eqref{eq:gapEmig:est4ptnl3}, we begin by observing that the Taylor expansion of $r^{2} / \sinh^{2} r$ at $r = 0$ is given by
\begin{align*}
	\frac{r^{2}}{\sinh^{2} r} = 1 - \frac{1}{3} r^{2} + \frac{1}{2} \int_{0}^{r} (r - r')^{2} E(r') \, \ud r'
\end{align*}
where
\begin{equation*}
	E(r) = \frac{\ud^{3}}{\ud r^{3}} \bb( \frac{r^{2}}{\sinh^{2} r} \bb).
\end{equation*}

%
%% NOTE: Taylor expansion formula
%
%\begin{align*}
%	f(x) 	=& f(0) + \int_{0}^{x} f'(x') \, \ud x' \\
%		=& f(0) + x f'(0) + \int_{0}^{x} \int_{0}^{x'} f''(x'') \, \ud x'' \ud x' \\
%		=& f(0) + x f'(0) + \frac{1}{2} f''(0) + \int_{0}^{x} \int_{0}^{x'} \int_{0}^{x''} f'''(x''') \, \ud x''' \ud x'' \ud x' \\
%		=& f(0) + x f'(0) + \frac{1}{2} f''(0) + \int_{0}^{x} \bb(\int_{x'''}^{x} \int_{x''}^{x} \, \ud x' \ud x'' \bb) f'''(x''') \, \ud x''' \\
%		=& f(0) + x f'(0) + \frac{1}{2} x^{2} f''(0) + \frac{1}{2} \int_{0}^{y} (x-y)^{2} f'''(y) \, \ud y.
%\end{align*}

Note that $E(r)$ obviously enjoys the bound $\sup_{r \in [0,1]} \abs{E(r)} \leq C$ for some absolute constant $C > 0$. Therefore,
\begin{equation*}
	\bb\vert
	\frac{r^{2}}{\sinh^{2} r} - 1 + \frac{1}{3} r^{2} \bb\vert
	\leq C r^{3}
\end{equation*}
for $0 \leq r \leq 1$. Making a change of variable $r = \rho / \lmb$ and restricting to $0 \leq \rho \leq \eps_{2} \lmb$, it follows that
\begin{align*}
		\bb\vert
		 \frac{3}{4} \frac{1}{\lmb^{2} \sinh^{2}(\rho/\lmb)}
		 - \frac{3}{4} \frac{1}{\rho^{2}} + \frac{1}{4 \lmb^{2}} \bb\vert
		 \leq \eps_{2} \frac{C}{\lmb^{2}} .
\end{align*}

Choosing $\eps_{2} > 0$ sufficiently small compared to $\overline{\mu}^{2} > 0$, \eqref{eq:gapEmig:est4ptnl3} follows. \qedhere

%
% NOTE: Exact computation for Taylor expansion of $r / \sinh r$.
%
%\begin{align*}
%	\frac{\ud^{3}}{\ud r^{3}} \bb( \frac{r}{\sinh r} \bb)
%	\\	=& \frac{\ud^{2}}{\ud r^{2}} \bb( \frac{\sinh r - r \cosh r}{\sinh^{2} r} \bb)
%	\\	=& \frac{\ud}{\ud r} \bb( \frac{- r \sinh^{2} r - 2 \cosh r (\sinh r - r \cosh r)}{\sinh^{3} r} \bb)
%	\\ 	=& \frac{\ud}{\ud r} \bb( \frac{r (1 + \cosh^{2} r) - 2 \sinh r \cosh r}{\sinh^{3} r} \bb)
%	\\ 	=& \frac{[(1+\cosh^{2} r) + 2 r \sinh r \cosh r - 2 \cosh^{2} r - 2 \sinh^{2} r] \sinh r - 3 (r (1 + \cosh^{2} r) - 2 \sinh r \cosh r) \cosh r}{\sinh^{4} r}
%\end{align*}

\end{proof}

We are now ready to prove Proposition \ref{prop:gapEmig:main}.
\begin{proof} [Proof of Proposition \ref{prop:gapEmig:main}]
In what follows, we will omit the superscript $(\lmb)$ in $\ztinf^{(\lmb)}$ for simplicity.
For the sake of contradiction, suppose that $\phi_{0}$ does not change sign, i.e., $\phi_{0}$ is positive.

\vskip1em
\noindent\textbf{Step 1:}
We claim that
\begin{equation} \label{eq:gapEmig:compareZeta}
	W[\phi_{0}, \ztinf] = \phi_{0} (\ztinf)' - \phi'_{0} \ztinf \leq 0.
\end{equation}

Indeed, suppose the contrary. Then at some $R > 0$, we must have
\begin{equation*}
	W[\phi_{0}, \zt_{\infty}] (R) > 0 .
\end{equation*}

We introduce an auxiliary function $\zt$ which solves the equation
\begin{equation*}
\left\{
\begin{aligned}
	\calL_{V_{\lmb}} \zt =& 0, \\
	(\zt, \zt')(R) =& (\phi_{0}, \phi'_{0})(R).
\end{aligned}
\right.
\end{equation*}
Since $\zt$ and $\ztinf$ solve the same equation, their Wronskian is constant. Therefore,
\begin{equation*}
W[\zt, \ztinf] (r)
= W[\zt, \ztinf] (R)
= W[\phi_{0}, \ztinf] (R) > 0
\end{equation*}
for all $r \geq 0$. It follows that $\zt$ must tend to $-\infty$ as $r \to \infty$, thus changing sign. Then a comparison argument between $\phi_{0}$ and $\zt$ as in Propositions \ref{good spec} and \ref{no neg spec}, using crucially the fact that $\overline{\mu}^{2} > 0$, shows that $\phi_{0}$ must also change sign, which is a contradiction.

\vskip1em
\noindent\textbf{Step 2:}
As discussed earlier, the benefit of \eqref{eq:gapEmig:compareZeta} is that it gives a lower bound on $\phi_{0}$ in terms of $\ztinf$ on an arbitrarily long interval in the renormalized coordinate. Indeed, by \eqref{eq:gapEmig:compareZeta}, we have
\begin{equation} \label{eq:gapEmig:lb4dpsi0}
	\frac{\ud}{\ud r} \log \phi_{0}(r) \geq \frac{\ud}{\ud r} \log \ztinf(r).
\end{equation}
Thus, for any $r \geq r_{0} > 0$,
\begin{equation} \label{eq:gapEmig:lb4psi0}
	\phi_{0}(r) \geq \frac{\phi_{0}(r_{0})}{\zt_{\infty}(r_{0})} \zt_{\infty}(r).
\end{equation}

To translate this lower bound to the renormalized picture, we make the change of variable $\rho = \lmb r$ (thus $\rho_{0} = \lmb r_{0}$) and define
\begin{align*}
	\widetilde{\ztinf}(\rho) := & \lmb^{-1/2} \ztinf(\rho / \lmb), \\
	g(\rho) :=& \frac{\lmb^{3/2} \phi_{0}(\rho/ \lmb)}{\varphi_{0}(\rho)},
\end{align*}
where we remind the reader that
\begin{equation*}
	\varphi_{0}(\rho) = \frac{\rho^{3/2}}{1+(\rho/2)^{2}}.
\end{equation*}
Since $\varphi_{0} > 0$, our contradiction hypothesis $\phi_{0} > 0$ is equivalent to $g > 0$.
Moreover, \eqref{eq:gapEmig:lb4psi0} translates to
\begin{align}
%	\frac{g'(\rho)}{g(\rho)}
%	\geq & \, \frac{\widetilde{\ztinf}'(\rho)}{\widetilde{\ztinf}(\rho)} - \frac{\varphi_{0}'(\rho)}{\varphi_{0}(\rho)} \\
	g(\rho)
	\geq & \, g(\rho_{0}) \bb( \frac{\widetilde{\ztinf}(\rho_{0})}{\varphi_{0}(\rho_{0})} \bb)^{-1}  \bb( \frac{\widetilde{\ztinf}(\rho)}{\varphi_{0}(\rho)} \bb).\label{eq:gapEmig:lb4g}
\end{align}

Applying Lemma \ref{lem:gapEmig:est4ztinf} and plugging in the definition of $\varphi_{0}$, we see that
\begin{equation}
	g(\rho) \geq C_{\rho_{0}} g(\rho_{0}), \qquad \hbox{ for } \rho_{0} \leq \rho \leq \eps_{1} \lmb, \quad \lmb \geq 1,
\end{equation}
where $C > 0$ is independent of $\rho$ and $\lmb$.

\vskip1em
\noindent\textbf{Step 3:}
Note that $g$ satisfies $(g, g')(0) = (1,0)$ and
\begin{equation} \label{eq:gapEmig:eq4g}
\left(g^\prime\fy_0^2\right)^\prime=\fy_0^2 W_{\lmb, \overline{\mu}} g
\end{equation}
where
\begin{equation*}
	W_{\lmb, \overline{\mu}}(\rho) = \frac{3}{4} \frac{1}{\lmb^{2} \sinh^{2}(\rho/\lmb)} - \frac{3}{4} \frac{1}{\rho^{2}} + \frac{1}{4 \lmb^{2}} - \frac{\overline{\mu}^{2}}{\lmb^{2}} + \frac{1}{\lmb^{2}} V_{\lmb}(\rho/\lmb) - V_{\euc}(\rho).
\end{equation*}

At this point, we fix a large enough $\rho_{0} > 0$ so that we have
\begin{equation} \label{eq:gapEmig:ub4W}
	W_{\lmb, \overline{\mu}} (\rho) \leq - \frac{\overline{\mu}^{2}}{4 \lmb^{2}} \qquad \hbox{ for } \rho_{0} \leq \rho \leq \eps_{2} \lmb.
\end{equation}
Indeed, \eqref{eq:gapEmig:ub4W} follows by combining \eqref{eq:gapEmig:est4ptnl2} and \eqref{eq:gapEmig:est4ptnl3}.
We now claim that the following bounds hold for $g(\rho)$ and $g'(\rho)$: For $0 \leq \rho \leq \rho_{0}$ and $\lmb \geq \rho_{0}$, we have
\begin{align}
	\abs{g(\rho) - 1} \leq & \frac{C_{\rho_{0}}}{\lmb^{2}}, 	\label{eq:gapEmig:bnd4grho0}\\
	\abs{g'(\rho)} \leq & \frac{C_{\rho_{0}}}{\lmb^{2}}.	\label{eq:gapEmig:bnd4dgrho0}
\end{align}

These bounds are proved in a similar fashion to the proof of uniqueness of $\gapE$, cf. Lemma \ref{eval uniqueness r-zero}.
From \eqref{eq:gapEmig:eq4g} it follows that
\begin{align*}
	g'(\rho)
		=& \frac{1}{\varphi_{0}^{2}(\rho)} \int_{0}^{\rho} \varphi_{0}^{2}(\sgm) W_{\lmb, \overline{\mu}}(\sgm) g(\sgm) \, \ud \sgm, \\
	g(\rho)
		=& 1 + \int_{0}^{\rho} \int_{0}^{\tau} \frac{\varphi_{0}^{2}(\sgm)}{\varphi_{0}^{2}(\tau)} W_{\lmb, \overline{\mu}}(\sgm) g(\sgm) \, \ud \sgm \ud \tau.
\end{align*}

Using \eqref{eq:gapEmig:est4ptnl1} and \eqref{eq:gapEmig:est4ptnl2} to estimate $W_{\lmb, \overline{\mu}}$, substituting $\varphi_{0}$ by its explicit definition and estimating the resulting integral, it follows that
\begin{gather*}
	\int_{0}^{\rho_{0}} \varphi_{0}^{2}(\sgm) \abs{W_{\lmb, \overline{\mu}}(\sgm)} \, \ud \sgm
		 \aleq_{\rho_{0}} \frac{1}{\lmb^{2}}, \\
	\int_{0}^{\rho_{0}} \bb( \int_{\sgm}^{\rho_{0}} \frac{1}{\varphi_{0}^{2}(\tau)} \, \ud \tau \bb) \varphi_{0}^{2}(\sgm) \abs{W_{\lmb, \overline{\mu}}(\sgm)} \, \ud \sgm
		 \aleq_{\rho_{0}} \frac{1}{\lmb^{2}}.
\end{gather*}
Then by a Volterra-type iteration, \eqref{eq:gapEmig:bnd4grho0} follows. Moreover, \eqref{eq:gapEmig:bnd4dgrho0} is an immediate  consequence of plugging in \eqref{eq:gapEmig:bnd4grho0} to the formula for $g'(\rho)$.

\vskip1em
\noindent\textbf{Step 4:}
We now derive a contradiction. Our starting point is the identity
\begin{equation*}
	g(\rho) = g(\rho_{0}) + g'(\rho_{0}) \int_{\rho_{0}}^{\rho} \frac{\varphi_{0}^{2}(\rho_{0})}{\varphi_{0}^{2}(\tau)} \, \ud \tau +  \int_{\rho_{0}}^{\rho} \int_{\rho_{0}}^{\tau} \frac{\varphi_{0}^{2}(\sgm)}{\varphi_{0}^{2}(\tau)} W_{\lmb, \overline{\mu}}(\sgm) g(\sgm) \, \ud \sgm \ud \tau,
\end{equation*}
for $\rho \geq \rho_{0}$, which is obtained by integrating \eqref{eq:gapEmig:eq4g} twice.
%
%% NOTE: Alternatively,
%
%\begin{align*}
%	g(\rho) = g(\rho_{0})
%			&+ \int_{\rho_{0}}^{\rho} \int_{0}^{\rho_{0}} \frac{\varphi_{0}^{2}(\sgm)}{\varphi_{0}^{2}(\tau)} W_{\lmb, \overline{\mu}} (\sgm) g(\sgm) \, \ud \sgm \ud \tau \\
%			&+  \int_{\rho_{0}}^{\rho} \int_{\rho_{0}}^{\tau} \frac{\varphi_{0}^{2}(\sgm)}{\varphi_{0}^{2}(\tau)} W_{\lmb, \overline{\mu}}(\sgm) g(\sgm) \, \ud \sgm \ud \tau.
%\end{align*}

Fix $\eps > 0$ so that $\eps = \min \set{\eps_{1}, \eps_{2}}$. Taking $\lmb \geq \rho_{0} / \eps$, let $\rho = \eps \lmb$ in the preceding identity. Then since $\eps \leq \eps_{2}$, we can apply \eqref{eq:gapEmig:est4ptnl3} and conclude the following one-sided inequality:
\begin{equation} \label{eq:gapEmig:key4g}
	g(\eps \lmb) \leq g(\rho_{0}) \bb(1 - \frac{C \overline{\mu}^{2}}{\lmb^{2}} \int_{\rho_{0}}^{\eps \lmb} \int_{\rho_{0}}^{\tau} \frac{\varphi_{0}^{2}(\sgm)}{\varphi_{0}^{2}(\tau)} \, \ud \sgm \ud \tau \bb) + \abs{g'(\rho_{0})} \int_{\rho_{0}}^{\eps \lmb} \frac{\varphi_{0}^{2}(\rho_{0})}{\varphi_{0}^{2}(\tau)} \, \ud \tau.
\end{equation}

Recalling the definition of $\varphi_{0}$, we easily compute
\begin{align*}
	\int_{\rho_{0}}^{\eps \lmb} \frac{1}{\varphi_{0}^{2}(\tau)} \, \ud \tau \leq & C_{\rho_{0}} \eps^{2} \lmb^{2}, \\
	- \frac{1}{\lmb^{2}} \int_{\rho_{0}}^{\eps \lmb} \int_{\rho_{0}}^{\tau} \frac{\varphi_{0}^{2}(\sgm)}{\varphi_{0}^{2}(\tau)} \, \ud \sgm \ud \tau
	 \leq & - C_{\rho_{0}} \eps^{2} \log (2+\eps \lmb).
\end{align*}
Therefore, we obtain
\begin{equation*}
	g(\eps \lmb) \leq g(\rho_{0}) \big( 1 - C_{\rho_{0}, \eps} \, \overline{\mu}^{2} \log (2+\eps \lmb) \big) + C_{\rho_{0}} \eps^{2} (\lmb^{2} \abs{g'(\rho_{0})}).
\end{equation*}
We now recall the bounds \eqref{eq:gapEmig:bnd4grho0} and \eqref{eq:gapEmig:bnd4dgrho0} for $g(\rho_{0})$ and $\lmb^{2} \abs{g'(\rho_{0})}$.
Thanks to the term $\log (2+\eps \lmb)$,  we then see that the right-hand side is negative when $\lmb$ is sufficiently large. It follows that $\phi_{0}(\eps) = g(\eps \lmb) \varphi(\eps \lmb) < 0$, which contradicts our hypothesis that $\phi_{0} > 0$. \qedhere
\end{proof}

%%%%%%%%%%%%%%%%%%%%%%%%%%%%%%%%%%%%%%%%%%%%%%%%%%%%%%%%%%%%%%%%%%%%%%%%%%%%%%%
%%%%%%%%%%%%%%%%%%%%%%%%%%%%%%%%%%%%%%%%%%%%%%%%%%%%%%%%%%%%%%%%%%%%%%%%%%%%%%%

\section{Strichartz estimates for the linearized operators} \label{strich section}

The goal of this section is to prove Strichartz  estimates for a radial shifted~\footnote{In \cite{AP} and most of the related literature ``the shifted wave equation" refers to the equation $(\Box_\mathbf{g}-\frac{9}{4})u=F$ because the spectrum of the Laplacian $\Delta_\g$ on $\Hp^4$ is $[9/4,\infty)$.  Nevertheless, we have preferred to use the term ``a shifted wave equation"  here as well since there is little risk for confusion.} linear wave equation in $\R \times \Hp^4$, perturbed by a radial potential $V$, i.e.,
\EQ{ \label{lin wave V}
&u_{tt} - u_{rr} - 3\coth r \, u_r -2 u + V u  = F,\\
&\vec u(0) = (u_0, u_1).
}
We will make several assumptions about $V$, which are consistent with the potentials $V_{\la}, U_{\la}$, where $V_{\la}$ is as in Proposition~\ref{good spec}, and $U_{\la}$ is as in~\eqref{Ula}. First define $$H_V:= - \p_{rr}- 3 \coth r\, \p_r - 2 +V(r).$$  We will work under the assumptions that
\begin{itemize}
\item[($A$)] $V$ is real-valued, smooth, radial, and bounded on $\Hp^4$, and $H_V$ is self-adjoint on the domain $\calD = H^2(\Hp^4)$. Moreover $V(r)  \le C e^{-2r}$ as $r \to \infty$.
\item[($B$)] The operator $H_V$ defined above has purely absolutely continuous spectrum $$\sigma(H_V) = \left[1/4, \infty\right).$$
In particular,  $H_V$ has no negative spectrum and no eigenvalues in the gap $[0, 1/4)$. Moreover,  the threshold energy $\frac{1}{4}$ is neither an eigenvalue nor a resonance.
\end{itemize}

\begin{rem} We note that by Proposition~\ref{good spec}, $H_{V_{\la}}$ satisfies $(A)$ and $(B)$ above for $0 \le \la < \sqrt{15/8}$. On the other hand, $H_{U_\la}$ satisfies $(A)$ and $(B)$ for all $\la \in [0, 1),$ because $U_{\la} \ge 0$ is a repulsive potential. In fact, we have
\EQ{
U_{\la}(r) := \frac{ \cosh 2P_{\la} - 1}{ \sinh^2 r} = \frac{8 \lmb^{2} }{[\cosh r + 1 - \lmb^{2} (\cosh r - 1) ]^{2}} \ge 0.
}
\end{rem}

Strichartz estimates for the free equation, that is with $V \equiv 0$, were  proved by Anker-Pierfelice,~\cite{AP}, and we briefly recall their set-up and main result.  The corresponding free shifted linear wave equation on $\R \times \Hp^4$ is given by
\EQ{ \label{shift wave}
&(\Box_\g -2) v := v_{tt} - \Delta_\g v - 2 v = F,\\
&\vec v(0) = (v_0, v_1).
}
 A triple $(p, q, \sigma)$ is called hyperbolic-admissible if
\EQ{
p, q>2, \quad \frac{1}{p} + \frac{3}{2q}  \le \frac{3}{4}, \quad \frac{1}{p} + \frac{4}{q} = 2- \s.
}
  %from~\cite{AP}.
\begin{prop}\cite[Corollary~$5.3$]{AP}\label{free strich}  Suppose $\vec v(t)$ is a solution to~\eqref{shift wave} with initial data $\vec v(0) = (v_0, v_1)$ and let $0 \in I \subset \R$ be any time interval. Let $(p, q, \s)$ and $(a, b, \ga)$ be any two hyperbolic-admissible triples. Then we have the estimates
\EQ{
\|  v\|_{L^p(I; W^{ 1-\s, q}(\Hp^4))}+   \|  \p_t v\|_{L^p(I; W^{ -\s, q}(\Hp^4))} \lesssim  \|\vec v(0) \|_{H^1 \times L^2(\Hp^4)} + \|F\|_{L^{a'}(I; W^{\ga, b'}(\Hp^4))}.
%\|  (v, v_t)\|_{(L^p W^{ 1-\s, q} \times L^p W^{ \s, q})(I \times \Hp^4) } \lesssim  \|\vec v(0) \|_{H^1 \times L^2(\Hp^4)} + \|F\|_{L^{a'}(I; W^{\ga, b'}(\Hp^4))}.
}
\end{prop}

We will use Proposition~\ref{free strich} together with a perturbative argument to establish the corresponding estimates for~\eqref{lin wave V}.  In particular we prove the following result.
\begin{prop}\label{strich}
Suppose $\vec u(t)$ is a solution to~\eqref{lin wave V} with initial data $\vec u(0) = (u_0, u_1)$ and with $V$ satisfying assumptions (A) and (B)  above. Let $0 \in I \subset \R$ be any time interval. Let $(p, q, \s)$ and $(a, b, \ga)$ be any two hyperbolic-admissible triples. Then we have the estimates
\EQ{\label{strich est}
\|  u\|_{L^p(I; W^{ 1-\s, q}(\Hp^4))}+    \|  \p_t u\|_{L^p(I; W^{ - \s, q}(\Hp^4))}  \lesssim  \|\vec u(0) \|_{H^1 \times L^2(\Hp^4)} + \|F\|_{L^{a'}(I; W^{\ga, b'}(\Hp^4))}.
}
\end{prop}

 %The estimates~\eqref{strich est} will be established via a reduction to certain localized energy estimates.

\begin{proof}[Proof of Proposition~\ref{strich}]

The proof  roughly follows the approach in~\cite[Section~$5$]{LS}, which in turn is a variant of an argument in~\cite{RodS}. Note that by the standard $TT^*$ argument and Minkowski's inequality it suffices to consider the case $F=0$. %Also, we let  $X := L^{p}_tW_x^{1- \s, q}$ be any Strichartz norm, i.e., we assume $(p, q, \s)$ is hyperbolic admissible.

The argument hinges on certain estimates related to the distorted Fourier transforms relative to the self-adjoint operators $H_0:=-\Delta_\g -2$ and $H_V=-\Delta_\g - 2+ V$ on the domain
$\calD: = H^2( \Hp^4)$, restricted to radial functions. First though, we can reduce the proof of Proposition~\eqref{strich} to a pair of local energy estimates, in particular~\eqref{free loc en} and~\eqref{pert loc en} below.

Indeed,   %Let $\Delta  = \Delta_\g$ be the Laplacian on $\Hp^4$.
define the operator
 \ant{
 A:= \sqrt{-\Delta_\g - 2},
 }
and note that
\EQ{ \label{A norm}
\| A f\|_{L^2(\Hp^4)}  \simeq \| f\|_{H^1(\Hp^4)}.
}
For any real valued $ \vec u = (u_0, u_1) \in H^1 \times L^2(\Hp^4)$ we set
\ant{
w:= A u_0 + i u_1
}
Then~\eqref{A norm} implies that $\| w\|_{L^2} \simeq \|\vec u\|_{H^1 \times L^2}$. Moreover, $\vec u(t)$ solves~\eqref{lin wave V} if and only if
\EQ{
i \p_t w = A w + V u,\\
w(0) = Au_0 + iu_1.
}
The Duhamel formula then gives us
\ant{
w(t) = e^{-itA} w(0) - i \int_0^t e^{-i(t-s)A} V u(s) \, ds
}
and by~\eqref{A norm} we need to show that
\ant{
\|P w\|_X \le C\|w(0)\|_{L^2}
}
where $P:= A^{-1} \re $, and  $X := L^{p}_tW_x^{1- \s, q}$ is any Strichartz norm, i.e., $(p, q, \s)$ is hyperbolic admissible. By Proposition~\ref{free strich}, we have
\ant{
\|Pe^{-itA} w(0)\|_X \le C\|w(0)\|_{L^2}.
}
By the Christ-Kislelev lemma, see~\cite{CK01, Sogge}, it then suffices to show
\EQ{\label{CK}
\left\| P \int_{- \infty}^{\infty} e^{-i(t-s)A} V u(s) \, ds\right\|_X \lesssim \| w(0)\|_{L^2} \simeq \| \vec u(0)\|_{H^1 \times L^2}.
}
To prove~\eqref{CK} we factor the potential $V(r) = V_1(r) V_2(r)$ so that each factor decays like $e^{-r}$ as $r \to \infty$. Then,
 \EQ{ \label{reduce}
 &\left\| P \int_{- \infty}^{\infty} e^{-i(t-s)A} V u(s) \, ds\right\|_X \le \|K\|_{L^{2}_{t, x} \to X}    \|V_2 u\|_{L^2_{t, x}} \\
 & (Kf)(t):= P\int_{- \infty}^{\infty} e^{-i(t-s)A} V_1 f(s) \, ds
}
Next, note that for each $f \in L^2_{t, x}(\R \times \Hp^4)$ we have
\ant{
\|Kf\|_{X}
\le    \|P e^{-i tA}\|_{L^{2}_{x}  \to X} 
		\left\| \int_{- \infty}^{\infty} e^{isA} V_1 f(s) \, ds\right\|_{L^2_x}
}
The first factor on the right-hand-side above is bounded by a fixed constant by Proposition~\ref{free strich}. We claim that the second factor is bounded by $C \|f\|_{L^2_{t, x}}$. By duality, this is equivalent to the localized energy bound:
\EQ{\label{free loc en}
\|V_1 e^{-itA} \fy\|_{L^2_{t, x}} \le C \| \fy\|_{L^2_x}, \quad \forall \fy\in L^2(\Hp^4)
}
Therefore,  by~\eqref{reduce} we have reduced the proof of Proposition~\ref{strich} to proving~\eqref{free loc en} in addition to a local energy estimate for the perturbed evolution, namely
\EQ{\label{pert loc en}
   \|V_2 u\|_{L^2_{t, x}}  \le C \| \vec u(0) \|_{ H^1 \times L^2} \simeq C \| w(0)\|_{L^2}
}
To prove~\eqref{free loc en} and ~\eqref{pert loc en} we will need to develop some machinery. First we pass to an equation on the half-line by conjugating  by $\sinh^{\frac{3}{2}} r$.  Indeed,  the map
\EQ{ \label{4d to 1d}
L^2(\Hp^4) \ni \fy\mapsto \sinh^{\frac{3}{2}} r\, \fy =: \phi \in L^2(0, \infty)
}
is an isomorphism of $L^2(\Hp^4)$, restricted to radial functions, onto $L^2([0, \infty))$. If we define $\LL_0, \LL_V$ by
\EQ{
&\LL_0 := - \p_{rr} + \frac{1}{4} + \frac{3}{4 \sinh^2 r},\\
&\LL_V := - \p_{rr} + \frac{1}{4} + \frac{3}{4 \sinh^2 r} + V(r),
}
we have
\EQ{ \label{conj eq}
&(H_0 \fy)(r) = \sinh^{-\frac{3}{2}}r (\LL_0  \phi)(r),\\
&(H_V \fy)(r) = \sinh^{-\frac{3}{2}} r(\LL_V \phi)(r).
}
We claim that we can associate with $\LL_0$ and $\LL_V$ distorted Fourier bases, $\phi_0(r; \xi)$, respectively $\phi(r; \xi)$, that satisfy
\begin{align}
&\LL_0 \phi_0(r; \xi) = (\frac{1}{4} + \xi^2) \phi_0(r; \xi), \quad \phi_0(r; \xi)  \in L^2([0,  b]), \, \, \forall\, \, 0<b< \infty \label{fb 0}\\
&\LL_V \phi(r; \xi) = (\frac{1}{4} + \xi^2) \phi(r; \xi), \quad \phi(r; \xi)  \in L^2([0,  b]), \, \, \forall \, \, 0<b< \infty \label{fb V}.
\end{align}

To prove~\eqref{free loc en} we will need the following additional information about $\phi_0(r; \xi)$.  For all $g \in L^2([0, \infty))$,  set
\EQ{\label{free ft}
\ti g(\xi) := \int_0^{\infty} \phi_0(r; \xi) g(r) \, dr.
}
We can find a positive measure $\rho_0(d \xi) = \om_0( \xi) \, d\xi$ so that
\begin{align}\label{free f inv}
&g(r) = \int_0^{\infty} \phi_0(r; \xi) \, \ti g(\xi) \, \rho_0(d \xi),\\
 & \|g\|_{L^2(0, \infty)}^2 = \int_0^{\infty} \abs{ \ti g( \xi)}^2 \,  \rho_0( d \xi) \label{free planch},\\
 & \sup_{r>0, \xi>0}   \frac{\abs{ \phi_0(r; \xi)}^2}{(r+1)^2}\frac{\ang{\xi}}{\xi} \om_0(\xi) \le C < \infty  \label{free spec},
 \end{align}
 where we are using  the notation $\ang{\xi} := \sqrt{\xi^2 + \frac{1}{4}}$.
We remark here that the distorted Fourier basis~\eqref{fb 0} is explicit and is obtained by  setting  $\phi_0(r, \xi):= \sinh^{\frac{3}{2}} r \Phi_{\xi}(r)$ where the $\Phi_{\xi}$ are the hyperbolic spherical functions in $\Hp^4$. The spectral measure $\rho_0(d \xi):=\abs{c(\xi)}^{-2} d\xi$  where $c(\xi)$ is the Harish-Chandra $c$-function. Then~\eqref{free f inv} and   ~\eqref{free planch} follow from the corresponding facts about the Helgason-Fourier transform on $L^2(\Hp^4)$. The estimates~\eqref{free spec} will follow  from well known estimates for $\Phi_{\xi}$ and $c(\xi)$ and we will sketch the proof in Lemma~\ref{dis ft L0} below.
 %For $\vec g  = (g_0, g_1) \in H^1 \times L^2(0, \infty)$ we define $h$ by
%\ant{
%h:=  \sqrt{\LL_0} g_0 + i g_1
%}
%Then $g(t)$ solves $\LL_0 g= 0$ if and only if $h$ solves

With~\eqref{free f inv}--\eqref{free spec} we can easily prove~\eqref{free loc en}. First, note that after passing to the half-line formulation via conjugation by $\sinh^{\frac{3}{2}}r$,~\eqref{free loc en} is equivalent to proving
\EQ{
\|V_1 e^{-it \sqrt{\LL_0}} g\|_{L^2_{t, x}(\R \times (0, \infty))} \lesssim \| g\|_{L^2(0, \infty)},
}
for $g \in L^2(0, \infty)$. As $\LL_0$ is given by multiplication by $\ang{\xi}$ on the Fourier side, the above, using~\eqref{free f inv}, reduces to showing that
\EQ{\label{free loc en g}
\int_{-\infty}^{\infty} \left\| V_1 \int_0^{\infty} e^{-it \ang{\xi}} \phi_0(r; \xi) \ti g(\xi) \rho_0(d \xi) \right\|^2_{L^2(0, \infty)} \, dt \lesssim \| g\|_{L^2(0, \infty)}^2.
}
Expanding and carrying out the $t$-integration, the left-hand-side becomes
\EQ{
&\int_0^{\infty} V_1^2(r) \int_0^{\infty} \int_0^{\infty} \phi_0(r; \xi) \phi_0(r; \mu)  \ti g( \xi) \ba{\ti g(\mu)}   \de( \ang{\xi}- \ang{ \mu}) \om_0( \xi) d \xi\,  \om_0( \mu) d \mu \, dr\\
&=\int_0^{\infty} V_1^2(r) \int_0^{\infty}   \phi_0^2(r ;\xi)  \abs{ \ti g ( \xi)}^2 \, \om_0^2( \xi)  \frac{ \ang{\xi}}{ \xi} \, d \xi \, dr
}
Using the estimates~\eqref{free spec}, the exponential decay of $V_1(r)$, and~\eqref{free planch}, we can bound the above as follows:
\begin{multline*}
\int_0^{\infty} V_1^2(r) \int_0^{\infty}   \phi_0^2(r ;\xi)  \abs{ \ti g ( \xi)}^2 \, \om_0^2( \xi)  \frac{ \ang{\xi}}{ \xi} \, d \xi \, dr \lesssim\\ \int_0^\infty V_1(r) (r+1)^2 \, dr \int_0^{\infty} \abs{ \ti g ( \xi)}^2 \, \om_0( \xi)  \, d \xi
 \lesssim \int_0^{\infty} \abs{ \ti g ( \xi)}^2 \, \om_0( \xi)  \, d \xi  = \|g\|_{L^2(0, \infty)},
\end{multline*}
which proves~\eqref{free loc en g} and hence~\eqref{free loc en}.

The key point here  is that we can establish the analogs of~\eqref{free ft}--\eqref{free spec}  for the perturbed operator $\LL_V$ and $\phi( r; \xi)$ as in~\eqref{fb V}.  In particular, for $f \in L^2(0, \infty)$ we set
\EQ{\label{dist ft}
\hat f(\xi) := \int_0^{\infty} \phi(r; \xi) f(r) \, dr.
}
We can find a positive measure $\rho(d \xi) = \om( \xi) \, d\xi$ so that
\begin{align}\label{f inv}
&f(r) = \int_0^{\infty} \phi(r; \xi) \, \hat f(\xi) \, \rho(d \xi),\\
 & \|f\|_{L^2(0, \infty)}^2 = \int_0^{\infty} \abs{ \hat f( \xi)}^2 \,  \rho( d \xi) \label{planch},\\
 &%\| f\|_{H^1(0, \infty)}^2 \simeq
 \| \sqrt{\LL_V} f\|_{L^2( 0, \infty)}= \int_0^{\infty} \ang{\xi}^2 \abs{ \hat f( \xi)}^2 \,  \rho( d \xi)  \label{H1 dis},\\
 & \sup_{r>0, \xi>0}   \frac{\abs{ \phi(r; \xi)}^2}{(r+1)^3}\frac{\ang{\xi}}{\xi} \om(\xi) \le C < \infty  \label{spec}.
 \end{align}
The proof of~\eqref{spec} is where  the spectral information for $H_V$ plays a crucial role. In particular, the fact that the spectrum of $H_V$ is supported on $(1/4, \infty)$ and that the threshold energy $1/4$ is not a resonance allows us to establish the same rate of decay for $\om(\xi)$ as for the free spectral measure $\om_0(\xi)$ as $\xi \to 0+$.  We postpone the proof of~\eqref{fb V}, \eqref{dist ft}--\eqref{spec} to Lemma~\ref{dis ft} below, and first use these estimates to prove~\eqref{pert loc en}.

In light of~\eqref{fb V},  \eqref{f inv}, and~\eqref{A norm}, conjugating by $\sinh^{\frac{3}{2}} r$ reduces the local energy estimate~\eqref{pert loc en}  to showing that
\begin{multline}\label{pert reduce}
\int_{-\infty}^{\infty} \int_0^{\infty}  \abs{ V_2(r) \int_0^{\infty}  \phi(r; \xi)  \left( \cos(t \ang{\xi}) \hat f(\xi) + \frac{\sin(t \ang{\xi})}{ \ang{\xi}} \hat g( \xi) \right) \rho_0(d \xi)  }^2\, dt \\\lesssim \| (\sqrt{\LL_V}f, g)\|_{L^2 \times L^2(0, \infty)}^2.
\end{multline}
%\Green{Comment by Sohrab: I think for the first part of equation \ref{H1 dis} we need to assume $\sinh^{-\frac{3}{2}}rf\in H^1(\Hp^4)$ as well. In that case, writing $f=\sinh^{\frac{3}{2}}rh$, the statement is correct because
%\begin{align*}
%\|h\|_{H^1(\Hp^4)}^2&\simeq\|\sqrt{H_0}h\|_{L^2(\Hp^4)}^2=\int \sqrt{ H_0}h\cdot \sqrt {H_0} h\sinh^3rdr\\
%&=\int \sqrt {\LL_0}(\sinh^{\frac{3}{2}}r h)\cdot\sqrt{\LL_0}(\sinh^{\frac{3}{2}}r)hdr=\|\sqrt{\LL_0}(\sinh^{\frac{3}{2}}rh)\|_{L^2(0,\infty)}^2.
%\end{align*}
%On the other hand the left side of the above equation can be written as
%\begin{align*}
%\|h^\prime\|_{L^2(\Hp^4)}^2\simeq\|\sinh^{\frac{3}{2}}rh^\prime\|_{L^2(0,\infty)}^2\simeq\|(\sinh^{\frac{3}{2}}rh)^\prime\|_{L^2(0,\infty)}^2,
%\end{align*}
%where the last equivalence holds because
%\begin{align*}
%\|(\sinh^{\frac{3}{2}}rh)^\prime\|_{L^2(0,\infty)}^2&\simeq\int\frac{h^2(r)}{\sinh^2 r}\cosh^2r\sinh^3rdr+\int (h^\prime (r))^2\sinh^3rdr\\
%&\lesssim\int\frac{h^2(r)}{\sinh^2r}\sinh^3rdr+\int h^2(r)\sinh^3rdr+\int (h^\prime (r))^2\sinh^3rdr\\
%&\simeq\int (h^\prime (r))^2\sinh^3rdr=\|\sinh^{\frac{3}{2}}rh^\prime\|_{L^2(0,\infty)}.
%\end{align*}
%In fact it seems to me that this last computations is also needed to pass to the right hand side of (\ref{pert reduce}) above.\\
%}
For simplicity, we first consider the case $g=0$ above. After expanding and carrying out the $t$-integration on the left-hand-side as before, one obtains
 \begin{multline*}
 \int_0^{\infty} V_2^2(r) \int_0^{\infty}  \phi^2(r; \xi) \abs{\hat f(\xi)}^2 \, \om^2(\xi) \, \frac{\ang{\xi}}{\xi} \,  d \xi \, dr  \lesssim\\
  \lesssim \sup_{r >0, \xi >0}   \left[\frac{\abs{ \phi(r; \xi)}^2}{(r+1)^3}\frac{1}{\xi \ang{\xi}} \om(\xi) \right] \int_0^{\infty} V_2^2(r) (r+1)^3 \,  dr \int_0^{\infty} \ang{\xi}^2 \, \abs{\hat f( \xi)}^2 \,\om( \xi) \, d\xi \\
  \lesssim \int_0^{\infty} \ang{\xi}^2 \, \abs{\hat f( \xi)}^2 \,\om( \xi) \, d\xi   =  \| \sqrt{\LL_V} f\|_{L^2( 0, \infty)} %\simeq\| f\|_{H^1(0, \infty)}^2
 \end{multline*}
  The case $f = 0$ is similar. This proves~\eqref{pert reduce} and hence~\eqref{pert loc en}. This finishes the proof of Proposition~\ref{strich} pending the proofs of the technical statements regarding the distorted Fourier transform in Lemma~\ref{dis ft L0} and Lemma~\ref{dis ft} below.
\end{proof}
%With $K$ defined as above,  we have further reduced the proof to showing that
%\EQ{
%\|K f\|_{X} \le C\|f\|_{L^2_{t, x}}
%}

\subsection{Distorted Fourier Transform for $\LL_0$ and $\LL_V$} In this subsection we establish the technical statements regarding the distorted Fourier transforms for $\LL_0$ and $\LL_V$.

We begin with the free case, $\LL_0$, which will follow from well known estimates for the Helgason-Fourier transform in $\Hp^4$ restricted to radial functions.

\begin{lem}\label{dis ft L0} The half-line operator $\LL_0$ admits a distorted Fourier basis satisfying~\eqref{fb 0},~\eqref{free ft}--\eqref{free spec},

\end{lem}

\begin{proof}
As we remarked earlier, the distorted Fourier basis associated to $\LL_0$ is explicit and is given by
\begin{multline} \label{phi0 def}
\phi_0(r; \xi):= \sinh^{\frac{3}{2}}r\,\Phi_{\xi}(r)\\ = C \sinh^{\frac{3}{2}}r\int_0^{\pi}( \cosh r - \sinh r \cos  \te)^{-i \xi - \frac{3}{2}}  \sin^2 \te \, d \te,
\end{multline}
 where here $\Phi_{\xi}(r)$ are the elementary spherical functions and serve as the Helgason-Fourier basis for $\Delta_{\Hp^4}$. Indeed, the elementary spherical functions  $\Phi_{\xi}(r)$ in $\Hp^4$ satisfy
\ant{
H_0 \Phi_{\xi}(r) =  (\xi^2 + 1/4) \Phi_{\xi}(r).
}
For radial functions $G \in L^2(\Hp^4)$ we can define the Helgason-Fourier transform by
\ant{
\ti{G}( \xi)  =  \int_0^{\infty} \Phi_{\xi}(r)G(r) \sinh^3 r \, dr.
}
The associated inversion formula is
\EQ{ \label{Hf inv}
G(r)  = C \int_0^{\infty}\Phi_{\xi}(r) \, \ti{G}(\xi) \abs{c(\xi)}^{-2} \, d \xi,
}
where $C$ is a normalizing constant and  $c(\xi)$ is the Harish-Chandra $c$-function
\ant{
c(\xi) = \frac{ 4\Gamma(i \xi)}{\pi^{1/2}\Gamma( \frac{3}{2} + i \xi)}.
}
The Plancherel theorem also holds. In particular, the Helgason-Fourier transform extends to an isometry of radial functions $L^2_{\textrm{rad}}( \Hp^4) \to L^2( \R_+, \abs{c(\xi)}^{-2} \, d \xi)$ with
\EQ{ \label{Hf planch}
\int_0^{\infty} f_1(r) \ba{f_2(r)} \, \sinh^3 r \, dr  = C \int_0^{\infty}  \ti f_1( \xi) \ba{ \ti f_2( \xi)} \,  \abs{c( \xi)}^{-2} \, d \xi.
}
For the estimates, it follows from the definition that  $\abs{c( \xi)}^{-2}$ satisfies the bound
\EQ{\label{c est}
\abs{c(\xi)}^{-2} \lesssim  \abs{\xi}^2 (1+ \abs{ \xi}).
}
For the spherical functions $\Phi_{\xi}(r)$ we separate two cases and we refer the reader to \cite{Bray, IS, AP} for more details. We will follow the notation of \cite{IS}.  For $r \ge r_0 >0$ with $r_0<1$ fixed, we can write
\EQ{\label{big r}
\Phi_{\xi}(r) = e^{- \frac{3}{2} r}  ( e^{ir \xi} c( \xi) m_1(r, \xi) + e^{-ir \xi} c( - \xi) m_1(r, - \xi))
}
where the function $\abs{m_1(r, \xi)}$ is uniformly bounded in $r$ and $\xi$. For small $r$, say $r \le 1$,  we can write
\EQ{ \label{small r}
\Phi_\xi(r) = e^{i r \xi} m_2(r, \xi) + e^{-ir \xi} m_2(r, - \xi)
}
where $m_2(r, \xi)$ satisfies
\EQ{\label{m2}
\abs{m_2(r, \xi)} \lesssim (1+ r \abs{\xi})^{- \frac{3}{2}}
}
Finally, we recall the uniform estimate
\EQ{\label{Phi0}
\sup_{\xi \ge 0} \abs{\Phi_\xi(r)} \le  \Phi_0(r) \lesssim e^{-\frac{3}{2} r}(1+ r)
}
which can be proved directly from~\eqref{phi0 def}.

We can now transfer these results to $\phi_0(r; \xi)$ and define the distorted Fourier transform relative to $\LL_0$. First note that since
\ant{
\sinh^{\frac{3}{2}}r(H_0 \Phi_{\xi})(r) =  (\LL_0  \phi_0( \cdot; \xi))(r),
}
 we have
\ant{
  \LL_0 \phi_0(r;  \xi) = ( \xi^2 + 1/4) \phi_0(r; \xi)
  }
For $g \in L^2(0, \infty)$ we can then define
\ant{
 \ti g( \xi) = \int_0^{\infty}  \phi_0(r; \xi) g(r) \, dr
 }
It then follows from~\eqref{Hf inv} and the isometry $$L^2_{\textrm{rad}}( \Hp^4) \ni G \mapsto  \sinh^{\frac{3}{2}} r \,  G =:g \in L^2(0, \infty)$$ that the following inversion formula holds
\ant{
g(r) =  C \int_0^{\infty}  \phi_0(r; \xi)  \ti g( \xi)  \abs{c( \xi)}^{-2} \, d \xi
}
Hence we can define the spectral measure $\rho_0( d \xi)= \om_0( \xi) \, d \xi := C \abs{c( \xi)}^{-2} \, d \xi$, which proves~\eqref{free f inv}. Plancherel's theorem~\eqref{free planch} follows directly from~\eqref{Hf planch}.

The estimates~\eqref{free spec} now follow easily as well. Recall that  $$\phi_0(r; \xi) = \sinh^{3/2}r\Phi_\xi(r).$$ For $r \ge r_0 >0$, with $r_0<1$ fixed,~\eqref{big r} implies
\EQ{\label{phi0 big r}
\abs{\phi_0(r; \xi)}^2 \lesssim \abs{c( \xi)}^2, \mfor r \ge r_0 >0.
}
For $r \le 1$,~\eqref{small r} and~\eqref{m2} translate to
\EQ{\label{phi0 small r}
\abs{ \phi_0(r, \xi)}^2 \lesssim r^3(1+ r \abs{\xi})^{-3} \mfor r \le 1
}
We also note that~\eqref{Phi0} allows us to deduce that
\EQ{ \label{phi0 bound}
\sup_{\xi \ge 0} \abs{\phi_0(r, \xi)} \lesssim 1+r
}
Therefore, using~\eqref{phi0 big r},~\eqref{phi0 small r}, and~\eqref{c est} for $\xi \ge \xi_0>0$, we have
\ant{
\frac{\abs{ \phi_0(r; \xi)}^2}{(r+1)^2}\frac{\ang{\xi}}{\xi} \abs{c(\xi)}^{-2} \lesssim  \begin{cases}  \frac{1}{(1+r)^2}  \le C\mif r \ge r_0>0 \\  \frac{r^3}{(1+r)^2} \frac{\abs{\xi}^3}{(1+ r \abs{\xi})^3} \le C \mif r \le 1\end{cases}
}
For $\xi \le  \xi_0 \le 1$  the uniform estimate~\eqref{phi0 bound} as well as~\eqref{c est} imply
\ant{
\frac{\abs{ \phi_0(r; \xi)}^2}{(r+1)^2}\frac{\ang{\xi}}{\xi} \abs{c(\xi)}^{-2} \lesssim  \frac{\abs{ \phi_0(r; \xi)}^2}{(r+1)^2}\ang{\xi}^2 \abs{\xi} \lesssim \abs{\xi}  \lesssim 1
}
This completes the proof of Lemma~\ref{dis ft L0}.
\end{proof}

We now extend these results to $\LL_V$ via a perturbation argument. The key point that allows the analysis to go through is that the threshold $\frac{1}{4}$ is neither a resonance nor an eigenvalue for $\LL_V$.  The Weyl-Titchmarsh theory for half-line operators with a singular potential used below is standard and can be found in \cite[Section $3$]{GZ}.

\begin{lem}\label{dis ft} The half-line operator $\LL_V$ admits a distorted Fourier basis satisfying~\eqref{fb V}, as well as~\eqref{dist ft}--~\eqref{spec}.
\end{lem}

\begin{proof}
First, we remark that $\LL_V$ is in the limit-point case at both $r=0$ and $r= \infty$. Moreover, $\LL_V$ admits an entire Weyl-Titchmarsh solution $\phi(r; z)$ of
\EQ{
\LL_V \phi(r; z) = z^2  \phi(r; z), \quad z  \in \C, \, \, r \in (0, \infty),
}
which satisfies
\begin{itemize}
\item[(a)] For all $r \in (0, \infty)$, $\phi(r;  \cdot)$ is entire.
\item[(b)]  $\phi(r; z)$ is real-valued for $z \in \R$.
\item[(c)] $\phi( \cdot, z) \in L^2([0, b])$ for all $b \in (0, \infty)$ and $z\in \C$.
\end{itemize}
The above follows from~\cite[Lemma $3.12$]{GZ}. In addition, by \cite[Lemma $3.3$]{GZ} we can find $\te(r; z)$ so that
\ant{
\LL_V \te(r; z) = z^2 \te(r; z) \quad z  \in \C, \, \, r \in (0, \infty),
}
where $\te(r; z)$ is real-valued for $z \in \R$, entire in $z \in \C$ and so that
\EQ{\label{W phi te}
W(  \te( \cdot, z), \phi(\cdot; z)) = 1, \quad z \in \C
} 
We also introduce another solution, namely the {\em Weyl-Titchmarsh solution} at infinity,  $\psi(r; z)$, which, since $\LL_V$ is in the limit-point case at $r= \infty$, is uniquely characterized (up to constant multiples) by
\EQ{\label{WT inf}
&\LL_V\psi(r; z) = z^2 \psi(r; z)\\
&\psi(r; z) \in L^2[a, \infty) \quad \forall z \in \C \, \,  \textrm{with} \, \, \Im(z) >0 \mand \forall a>0
}
Since $\{ \phi( \cdot, z), \te( \cdot, z)\}$ form a fundamental system we can find a function $m(z)$ so that
\EQ{ \label{WT psi def}
\psi(r; z) =  \te(r; z) +  m(z) \phi(r; z).
}
This function $m(z)$ is analytic in $\Im(z)>0$ and is referred to as the {\em Weyl-Titchmarsh function}. We note that $m(z)$ determines the spectral measure. Indeed,
\EQ{
 \rho(d \xi) = 2 \xi \Im m( \xi + i 0) \, d \xi
 }
 If for  $f \in L^2(0, \infty)$ we define the Distorted Fourier transform
 \ant{
  \hat f( \xi) := \int_0^{\infty} \phi(r; \xi) \, f(r) \, dr
  }
 then~\eqref{f inv}--\eqref{H1 dis} hold.

We also record some useful relations. Note that for $\xi \in (0, \infty)$ we have
 \EQ{ \label{m wro}
 2i \Im m( \xi)  = W( \ba{\psi( \cdot; \xi)}, \psi( \cdot; \xi))
 } 
 and
 \EQ{ \label{phi psi/m}
  \phi(r; \xi) =  \frac{\Im \psi(r; \xi)}{ \Im m( \xi)}
  }

 Our goal now, is to approximate $\phi(r; \xi)$ and $ \om( \xi) := 2 \xi \Im m( \xi+i0)$  well enough  to prove the estimate~\eqref{spec}.  We will accomplish this in two steps. The first step addresses the case $0 < \xi \ll 1$ and the second provides estimates for $\xi \gg1$.

\begin{flushleft}\textbf{Step 1:} We begin by considering the case of small $\xi$, i.e., $0< \xi \ll 1$. We will construct $\phi(r; \xi)$ in this case by perturbing around a fundamental system at energy $1/4$. In particular, we first establish: \end{flushleft}
 \begin{claim}\label{LLV 14} There exists a fundamental system, $\{ \phi(r), \te(r)\}$ to $\LL_V f = \frac{1}{4} f$ with $W( \phi, \te) =1$. Moreover we have the asymptotic behavior
 \EQ{ \label{phi te at 0}
  &\phi(r)  = r^{\frac{3}{2}} +o(r^{\frac{3}{2}}) \mas r \to 0,\\
  &\te(r) =- \frac{1}{2}r^{-\frac{1}{2}} + o(r^{-\frac{1}{2}})\mas r \to 0,
  }
  as well as
  \EQ{ \label{phi te at inf}
  &\phi(r)   =  a_1 + b_1 r +o(1) \mas r \to  \infty, \, \,  \textrm{with}\, \,  b_1 \neq 0,\\
  &\te(r) =  a_2+b_2 r +o(1) \mas r \to \infty,
}
where $a_1 b_2- a_2 b_1 =1$. The fact that $\LL_V$ has no point spectrum allows us to ensure that $b_1 \neq 0$.
 \end{claim}
  To prove the claim, we begin by constructing solutions with the desired behavior at $r=0$. The behavior at $r= \infty$, in particular the fact that $b_1 \neq 0$,  will then follow from the fact that $\frac{1}{4}$ is neither an eigenvalue, nor a resonance for $\LL_V$.

  Since $V(r)$ is smooth, we can write
  \ant{
  V(r) = V(0) + O(r) \mas r \to 0
  }
  Similarly, we have
  \ant{
   \frac{3}{4 \sinh^2r} -\frac{3}{4r^2} =- \frac{1}{4}  +  O(r^2) \mas r \to 0
   }
  For small $r$ we can thus view solutions to $\LL_V$ as a perturbations of solutions to
  \EQ{\label{ti LL}
  \LL_E f = (1/4 - V(0)) f
  }
  where $\LL_E := - \p_{rr} + \frac{3}{4 r^2}$. A fundamental system for~\eqref{ti LL} is given by
  \EQ{\label{alt Euc fund sys}
  &f_0(r) = c_0 r^{\frac{1}{2}} J_1( a r)   =  r^{\frac{3}{2}} +o(r^{\frac{3}{2}})\mas r \to 0\\
   & f_1(r) = c_1 r^{\frac{1}{2}} Y_1( a r) = -\frac{1}{2}r^{-\frac{1}{2}} +o(r^{-\frac{1}{2}})\mas r \to 0\\
   & a:=\sqrt{\frac{1}{4} - V(0)}
}
where $J_1, Y_1$ are the order one Bessel functions and $c_0, c_1$ are chosen to obtain the desired behavior at $r= 0$ on the right above. Indeed, if we set
\EQ{\label{NN0 def}
\NN_0(r):= -\frac{3}{ 4\sinh^2 r}  + \frac{3}{4r^2}  - \frac{1}{4} -V(r) +V(0) =  O(r) \mas r \to 0
}
then using the variation of constants formula we can write,
\EQ{ \label{volt phi te}
\phi(r) = f_0(r) + \int_0^r G(s, r) \NN_0(s) \phi(s)\, ds\\
\te(r) = f_1(r) + \int_0^r G(s, r) \NN_0(s) \te(s)\, ds
}
where the Green's function $G(r, s)$ is given by
\ant{
	G(r, s) = - f_0(s) f_1(r) + f_0(r) f_1(s) 	
}
The Volterra iterations in~\eqref{volt phi te} converge for $r$ small enough and yield the desired behavior~\eqref{phi te at 0}. Since both $V(r)$ and $\frac{3}{4 \sinh^2 r}$ decay exponentially as $r \to \infty$, the behavior~\eqref{phi te at inf} is the only possible one. The fact that $b_1 \neq 0$ is the important consequence of the fact that $\LL_V$ is non-resonant at $\frac{1}{4}$. This completes the proof of Claim~\ref{LLV 14}.

We can now construct $\phi(r; \xi)$
for small $\xi$ and $r \lesssim \xi^{-1/3}$. In particular we claim that there exists $ \e>0$ and $\xi_0>0$ such that for all $\xi \le \xi_0$ we have the estimates
\EQ{\label{phi xi small}
&\phi(r; \xi) = \phi(r) + O(r \xi), \mfor 0 \le r \le \e \xi^{-\frac{1}{3}},  \\
& \phi'(r; \xi) = \phi'(r) + O( \xi) \mfor 0 \le r \le \e \xi^{-\frac{1}{3}}
}
Moreover we can require the fixed normalization
\EQ{\label{phi norm at 0}
\phi(r; \xi) = \phi(r) + o(r^{\frac{3}{2}}) \mas r \to 0.
} %We also claim the estimates
%\EQ{ \label{theta xi small}
%&\te(r; \xi) = \te(r) + O(r \xi), \mfor 0 \le r \le \e \xi^{-\frac{1}{3}},  \\
%& \te(r, \xi) = \te'(r) + O( \xi) \mfor 0 \le r \le \e \xi^{-\frac{1}{3}}
%}
To prove~\eqref{phi xi small} we first define a function $u_0(r; \xi)$ by%define $\phi(r; \xi)$ for small $\xi$ as a perturbation of $\phi(r)$. Set
\EQ{\label{phi int eq}
 u_0(r; \xi) =  \phi(r) - \xi^2\int_0^{r} \te (r) \phi(s) u_0(s; \xi) \, ds - \xi^2\int_r^{ \e \xi^{-1/3}}  \phi(r) \te(s)u_0(s; \xi) \, d s
% \phi(r; \xi) =  \phi(r) - \xi^2\int_0^{r} \te (r) \phi(s) \phi(s; \xi) \, ds - \xi^2\int_r^{ \e \xi^{-1/3}}  \phi(r) \te(s)\phi(s; \xi) \, d s
 }
We will then obtain $\phi(r; \xi)$ from $u_0(r; \xi)$ by a renormalization so as to obtain the precise behavior~\eqref{phi norm at 0} in addition to~\eqref{phi xi small}.

Although~\eqref{phi int eq} is not a Volterra integral equation, it can still be solved by a contraction argument. To see this, define $u_1(r; \xi)$ by
\EQ{\label{u def}
u_0(r; \xi) =  \phi(r) + r \xi u_1(r; \xi)
}
Rewriting~\eqref{phi int eq} in terms of $u_1$ we can define a linear map $T_{\e, \xi}$ by
\EQ{ \label{u int eq}
u_1(r; \xi) &= -r^{-1} \xi \theta(r) \int_0^r \phi(s)[ \phi(s) + s \xi u_1(s; \xi)] \, ds \\
&\quad - r^{-1} \xi \phi(r) \int_r^{\e \xi^{-\frac{1}{3}}} \theta(s)[ \phi(s)+ s \xi u_1(s; \xi)] \, ds\\
&=: T_{\e, \xi} u_1 (r; \xi)
}
One can then check that there exists $\xi_0>0$ and $\e>0$ fixed so that for all $0 \le \xi \le \xi_0$, the map $T_{\e, \xi}$ is a contraction in a ball of fixed size in the space $C([0, \e \xi^{-\frac{1}{3}}])$. Hence, there is a unique solution $u_1(r; \xi)$ to~\eqref{u int eq} satisfying
\ant{
\abs{u_1(r; \xi)} \le C, \quad \forall \, 0 \le r \le \e \xi^{-1/3}
}
and all $0 <\xi \le \xi_0$. By plugging these estimates for $u_1$ into~\eqref{u def} we establish that
\ant{
u_0(r; \xi) = \phi(r) + O(r \xi), \mfor 0 \le r \le \e \xi^{-\frac{1}{3}},
}
Next, from~\eqref{phi int eq} we see that
\EQ{ \label{u0 at 0}
u_0(r, \xi)  =  r^{\frac{3}{2}}(1 + O( \e^3 \xi+ \xi^2)) + o(r^{\frac{3}{2}}) \mas r \to 0
}
Finally, we obtain  $\phi(r; \xi)$  from  $u_0(r, \xi)$  by multiplication by a constant $C= C(\e, \xi)$ determined by~\eqref{u0 at 0} to ensure that~\eqref{phi norm at 0} holds. A similar argument can be made for the derivative $\phi'(r; \xi)$ to prove the second line in~\eqref{phi xi small}. %~\eqref{phi xi small}.

 To estimate $\phi(r, \xi)$ in the regime $r \ge  \e\xi^{-1/3}$ and to find the spectral measure $\om( \xi)$, we define the function $\ti\psi(r; \xi)$ by
 \EQ{
& \ti \psi(r ; \xi) =  e^{ir \xi} - \int_r^{\infty} \frac{\sin \xi (r-s)}{ \xi} \NN_2(s) \ti \psi(s;  \xi) \, ds\\
 & \NN_2(r):=- \frac{3}{4 \sinh^2r} - V(r)
 }
Because of the exponential decay of $\NN_2(r)$ as $r \to \infty$ the Volterra iteration converges for $r$ large enough and  $ \ti\psi(r; \xi)$ solves
\ant{
\LL_V \ti \psi( r; \xi) = (1/4 + \xi^2) \ti \psi(r; \xi)
}
In addition, we can  deduce that
\EQ{ \label{ti psi large r}
 &\ti \psi(r; \xi) = e^{ir \xi} + O(re^{-2r}) \mas r \to \infty\\
 & \ti \psi'(r; \xi) = i \xi e^{ir \xi} + O(re^{-2r}) \mas r \to \infty
}
This implies that
\EQ{ \label{wro ti psi}
 W( \ti \psi( \cdot; \xi), \ba{ \ti \psi( \cdot ; \xi)}) = W(e^{ir \xi}, e^{-i r \xi}) = -2 i \xi
 }
 Moreover, since the Weyl-Titchmarsh solution $\psi(r; z)$ is uniquely characterized up to constant multiples by $\psi(\cdot; z) \in L^2([c, \infty))$ for all $\Im( z)>0$ and $c>0$   we can find a smooth function $a( \xi)$ so that 
 \EQ{
a(\xi) \ti \psi(r; \xi) = \psi(r; \xi)
 }
% \Green{I think $a$ on the RHS was missing a power of $-1$ which I put in. Also below in equation \ref{temporary equation} the second $\psi$ in the third equation was missing a tilde}\\
 Observe that  in light of~\eqref{W phi te} and~\eqref{WT psi def},  $a( \xi)$ is given by
\EQ{\label{a=W}
 a( \xi)  = \frac{1}{ W(  \ti \psi( \cdot; \xi), \phi( \cdot; \xi))}
 }
 By~\eqref{m wro} and~\eqref{wro ti psi}, estimates for the function $a( \xi)$ can then be used to estimate the $m$-function.  Indeed we have the relation
 \EQ{\label{temporary equation}
 2i \Im m( \xi) = W(  \ba{ \psi(\cdot, \xi)}, \psi( \cdot; \xi) ) = \abs{a( \xi)}^2 W( \ba{\ti \psi( \cdot; \xi)}, \ti \psi( \cdot; \xi)) = 2i \xi \abs{a( \xi)}^2
 } 
 We therefore look to bound the right hand side of~\eqref{a=W}. We will show that $W( \ti \psi( \cdot; \xi), \phi( \cdot; \xi))$ is bounded away from $0$ for small $\xi$. To see this we evaluate the Wronskian at $r = \xi^{-1/6}$, which is large enough for small $\xi$ so that we can accurately approximate $\ti \psi(r; \xi)$, but also within the range in which we have good control of $\phi(r; \xi)$ by~\eqref{phi xi small}. The spectral information on $\LL_V$ enters crucially at this point as evaluating the Wronksian at $r = \xi^{-1/6}$ yields
 \EQ{
 \abs{W( \ti \psi( \cdot; \xi), \phi( \cdot; \xi)) \vert_{r= \xi^{-1/6}}}  \ge c \abs{b_1} >0
 }
 where we have used~\eqref{phi xi small} and~\eqref{ti psi large r} as well as~\eqref{phi te at inf} above. Hence $\abs{a( \xi)} \simeq 1$ and we can conclude that
 \EQ{
 \Im m( \xi) \simeq \abs{ \xi}  \mas  \xi  \to 0
 }
 For the spectral measure $\om( \xi) d \xi$ we then have
 \EQ{\label{om small xi}
 \om( \xi) \simeq  \abs{ \xi}^2 \mas \xi \to 0
 }
 Finally, we would like to estimate $\phi(r; \xi)$ for small $\xi$ and $r \ge \e \xi^{-1/3}$. Since the functions $\ti \psi(r; \xi)$ and $\ba{\ti \psi(r; \xi)}$ form a fundamental system and since $\phi(r; \xi) \in \R$ for $\xi \in \R$, we can find $b( \xi)$ so that
 \EQ{ \label{phi psi bar}
 \phi(r; \xi) = b( \xi)  \ti \psi(r; \xi)  + \ba{b( \xi)}  \ba{\ti \psi(r; \xi)}
 }
 We then have
 \EQ{
 b( \xi) = \frac{W( \phi( \cdot; \xi),   \ba{\ti \psi( \cdot; \xi)})}{W( \ti \psi( \cdot; \xi), \ba{ \ti \psi( \cdot; \xi)})}
 }
 We use~\eqref{wro ti psi} for the denominator whereas we evaluate the numerate again at $r= \xi^{-1/6}$ as above to prove that
 \ant{
 \abs{b(\xi)} = O(\xi^{-1}) \mas \xi \to 0
 }
 This gives us the estimate
 \EQ{ \label{xi phi large r}
\sup_{r \ge \e \xi^{-1/3}} \xi \abs{\phi(r; \xi)} =O(1) \mas \xi \to 0
}
We can now prove~\eqref{spec} for small $\xi \to 0$. Using~\eqref{om small xi} we can deduce that
\EQ{ \label{spec small xi}
 \frac{\abs{ \phi(r; \xi)}^2}{(r+1)^3}\frac{\ang{\xi}}{\xi} \om(\xi) \lesssim \frac{\abs{ \phi(r; \xi)}^2}{(r+1)^3}  \abs{\xi}  \quad \forall \xi  \le \xi_0
 }
 For $r \le  \e \xi^{-1/3}$ we can use the estimate~\eqref{phi xi small} together with the conclusions of Claim~\ref{LLV 14} to see that
 \ant{
 \abs{\phi(r; \xi)} \lesssim \min(1, r), \quad \forall \,  r \le \e \xi^{-1/3}
 }
 Hence the right hand side of~\eqref{spec small xi} is bounded by a constant in the regime $r \le \e \xi^{-1/3}$. If $r \ge \e \xi^{-1/3}$ then we use~\eqref{xi phi large r} to conclude that
 \ant{
  \frac{\abs{ \phi(r; \xi)}^2}{(r+1)^3}  \abs{\xi} \lesssim  \frac{1}{ r^3 \xi} \lesssim 1, \quad  \forall \, r \ge \e \xi^{-1/3}
  }
  This finishes the proof of~\eqref{spec} for $\xi  \ll 1$.
  \begin{flushleft} \textbf{Step 2:} We now consider the case  $\xi \ge  \Xi_0$  for $\Xi_0$ large. This is somewhat easier than the small $\xi$ case as we can perturb around explicit solutions to a well understood Euclidean problem. Define the half-line Schr\"odinger operator
  \end{flushleft}
  \ant{
  \LL_E f := - f_{rr} + \frac{3}{4r^2}f
  }
  Then  $\LL_E f =  z^2 f$ are Bessel equations and we define the Weyl-Titchmarsh solution at $r=0$, by
  \EQ{
  \phi_E(r; z) := 2 z^{-1} r^{\frac{1}{2}} J_1(zr).
  }
  Together with
  \EQ{\label{Euclidean theta}
  \te_E(r, z):=  \frac{\pi}{4} z r^{\frac{1}{2}}[ -Y_1(zr)+ \pi^{-1} \log(z^2) J_1(z r)],
  }
  we have a fundamental system with $W(  \te_E( \cdot; z), \phi_E( \cdot; z)) = 1$ for all $z \in \C$, see~\cite[Section~$4$]{GZ} for more discussion of this system. % \Green{I think here we can add ``Note that here we are using a different fundamental system from the one in equation (\ref{alt Euc fund sys}). The reason for this is to avoid the branch cut in $z$ for $f_1.$ See \cite{GZ} Section 4."}
  The Weyl-Titchmarsh solution at $r= \infty$ is given by
  \EQ{\label{psiE}
  \psi_E(r; z) :=z r^{\frac{1}{2}} i H_1^{(1)}(z r) =  \te_E(r; z) + m_E(z) \phi_E(r; z),
  }
  where above $J_1, Y_1$ are the order one Bessel functions and $H_1^{(1)}$ is the order one Hankel function. The Weyl-Titchmarsh $m$-function is also explicit and is given by
  \EQ{
  m_E(z) = \frac{\pi}{4} z^2[i- \pi^{-1} \log(z^2)].
  }
  We now rewrite $\LL_V f = (\frac{1}{4}+ \xi^2) f$ as
  \ant{
  - f_{rr} + \frac{3}{4r^2} f -  \xi^2 f =  \frac{3}{4}\left( \frac{1}{r^2}- \frac{1}{ \sinh^2 r}\right)f - V(r) f.
  }
  The variation of parameters formula then gives the Volterra integral equation for the outgoing solutions
  \EQ{ \label{volt ti psi}
   \ti \psi_+(r; \xi) =  \psi_E(r; \xi) + \int_r^{\infty} G_E(r, s, \xi) \NN_3(s) \ti \psi_+(s; \xi) \, ds,
   }
   where
   \ant{
   G_E(r, s, \xi) := \frac{ \psi_E(r; \xi) \ba{ \psi_E(s; \xi)} -  \psi_E(s; \xi) \ba{\psi_E(r; \xi)}}{ W( \psi_E( \cdot; \xi), \ba{\psi_E( \cdot; \xi)})},
   }
   and
   \ant{
   \NN_3(r) := \frac{3}{4}\left( \frac{1}{r^2}- \frac{1}{ \sinh^2 r}\right) - V(r).
   }
    The Wronskian is given by
\EQ{\label{wGE}
W(\ba{\psi_E( \cdot; \xi)},  \psi_E( \cdot; \xi)) = 2i \Im m_E(\xi) =  \frac{\pi}{2} i \xi^2.
}
Let   $\e>0$ be a small number to be determined below. We will solve~\eqref{volt ti psi} on the region $r \in [\frac{1}{2}\epsilon \xi^{-1}, \infty)$.  Combining~\eqref{wGE} with~\eqref{psiE},  we note that the Green's function $G_E$ takes the form 
\EQ{
G_{E}(r, s, \xi) = \frac{2}{ i \pi} r^{1/2} s^{1/2}  2\,  \Im\Big( H_1^{(1)}(r \xi) \ba{ H_1^{(1)}(s \xi)} \Big)
}
Using the standard asymptotic expansions and estimates for the Hankel function,~$H_1^{(1)}(x)$ with $x >0$ large (see for example~\cite{ASbook, Ol}),  we can deduce the  bound
\EQ{
\abs{G_E(r, s, \xi)} \lesssim_\e  \abs{\xi}^{-1}, \quad \forall \, \, \, \frac{\e}{2\xi}  \le r, s
}
Moreover, we have the estimate
\EQ{
\NN_3(r) \lesssim \frac{1}{(1+r)^2}, \quad \forall\, r >0
}
Thus, if we set $K(r, s, \xi) := G_E(r, s, \xi) \NN_3(s)$ we have the estimate
\ant{
\abs{K(r, s, \xi)} \lesssim_\e  \xi^{-1} \ang{s}^{-2}, \quad \forall \,\, \frac{\e}{2\xi} \le r,  s
}
and thus
\ant{
\int_{\frac{\e}{2\xi}}^{\infty} \sup_{r>\frac{\e}{2\xi}} \abs{K(r, s, \xi)} \, ds \lesssim_\e \xi^{-1}
}
This means that the Volterra integral~\eqref{volt ti psi} has a unique solution $\ti \psi_+(r; \xi)$ defined for all $r \ge \frac{1}{2}\e \xi^{-1}$ and satisfying
\EQ{\label{psi psiE}
\ti \psi_+(r; \xi) =  \psi_E(r; \xi)(1+O( \xi^{-1} \ang{r}^{-1})), \quad \forall\, r  \ge  \frac{\e}{2\xi}
}
We can then deduce that
\EQ{
W( \ba{ \ti \psi_+(\cdot; \xi)},  \ti\psi_+( \cdot; \xi))=  W( \ba{ \psi_E(\cdot; \xi)},  \psi_E( \cdot; \xi)) =  \frac{\pi i}{2} \xi^2
} 
Since the Weyl-Titchmarsh solution at $r=\infty$,
\EQ{\label{psi te phi}
\psi(r; \xi) := \te( r; \xi) + m(\xi) \phi(r; \xi)
}
is uniquely determined up to constant multiples by~\eqref{WT inf}, we can find a smooth function $d(\xi)$ such that
 \EQ{
 \psi(r; \xi) =d(\xi)  \ti \psi_+(r; \xi)
 }
 Using~\eqref{psi te phi}  we see that
 \EQ{ \label{d def}
 d(\xi)  =   \frac{1}{ W(  \ti \psi_+( \cdot; \xi), \phi(\cdot; \xi)) } %=  \frac{1}{ W( \phi(\cdot; \xi),   \psi_E( \cdot; \xi))} + O( \xi^{-1}) \mas  \xi \to \infty
 }
% Using~\eqref{psi psiE} and the fact that we have fixed the normalization
% \ant{
%& \phi(r; \xi) = r^{\frac{3}{2}} +o(r^{\frac{3}{2}}) \mas r \to 0, \quad \forall \, \xi>0\\
%& \phi'(r; \xi) = \frac{3}{2}r^{\frac{1}{2}} +o(r^{\frac{1}{2}}) \mas r \to 0, \quad \forall \, \xi>0
%}
% we can evaluate the right-hand-side of~\eqref{d def} as $r \to 0^+$ to conclude that
We claim that 
 \EQ{ \label{d=1}
 \abs{d( \xi)} \simeq 1  \mas  \xi \to  \infty
 }
 To prove~\eqref{d=1}, we first note that for $\e>0$ small enough, we can find constants $c_1, C_2$, which are independent of $\xi$,  so that  
 \EQ{ \label{phi2exi}
 c_1 r^{\frac{3}{2}} \le \phi(r; \xi)  \le  C_2 r^{\frac{3}{2}},  \quad \forall   r \le 2 \e \xi^{-1} \\
 c_1 r^{\frac{1}{2}} \le \phi'(r; \xi) \le   C_2 r^{\frac{1}{2} }, \quad \forall   r \le 2 \e \xi^{-1}
 }
 By rewriting the equation for $\phi(r; \xi)$ as 
 \ant{
& -\phi'' + \frac{3}{4 r^2} \phi =  \NN_4(r, \xi) \phi \\
& \NN_4(r)  = \xi^2 + \frac{3}{4}( \frac{1}{r^2} - \frac{1}{\sinh^2 r}) + V(r)
}
we see that~\eqref{phi2exi} follows directly from an iteration argument  on the interval $[0, 2\e \xi^{-1}]$ for 
 \ant{
 &\phi(r; \xi)  = r^{\frac{3}{2}} - \frac{1}{2}  \int_0^r (r^{3/2}s^{-1/2} - r^{-1/2} s^{3/2}) \NN_4(s, \xi) \phi(s, \xi) \, ds 
 }
 as long as $\e>0$ is chosen small enough. Now~\eqref{d=1} follows by combining~\eqref{psi psiE} and~\eqref{phi2exi} to evaluate the right-hand-side of~\eqref{d def} at the point $r = \e \xi^{-1}$, and then taking $\xi \to \infty$. Here we remark that after closing the Volterra iteration for $\ti \psi_+(r, \xi)$ in~\eqref{volt ti psi},  the control of $\ti \psi_+'( r;  \xi)$ near the point $r = \e \xi^{-1}$ needed for~\eqref{d=1} follows by essentially differentiating~\eqref{volt ti psi}, plugging in the bound~\eqref{psi psiE} for $\ti \psi_+(r;  \xi)$ and using the asymptotics of the derivative of the Hankel function, $H^{(1)}_1$, which can be found for example in Abramowitz-Stegun~\cite[p. $364$, $9.2.13$]{ASbook}.
 %where 
% \ant{
% G_1(r, s)  = -\frac{1}{2}(r^{3/2}s^{-1/2} - r^{-1/2} s^{3/2}), %\quad   \NN_4(r)  = \xi^2 + \frac{3}{4}( \frac{1}{r^2} - \frac{1}{\sinh^2 r}) + V(r)
% }

By~\eqref{m wro} and~\eqref{wro ti psi}, estimate~\eqref{d=1} for the function $d( \xi)$ can then be used to estimate the $m$-function.  Indeed we have the relation
 \EQ{
 2i \Im m( \xi) &= W( \ba{ \psi(\cdot, \xi)}, \psi( \cdot; \xi)) = \abs{d( \xi)}^2 W( \ba{ \ti\psi_+( \cdot; \xi)}, \ti \psi_+( \cdot; \xi)) \\&=   \frac{ \pi i}{2}\xi^2 + O(1) \mas \xi \to \infty
 } 
We therefore have proved that
 \EQ{
 \Im m( \xi) \simeq \abs{ \xi}^2  \mas  \xi  \to \infty.
 }
 For the spectral measure $ \rho(d \xi) = \om( \xi) d \xi$ we then have
 \EQ{\label{om big xi}
 \om( \xi) \simeq  \abs{ \xi}^3 \mas \xi \to \infty
 }
Finally we estimate $\phi(r; \xi)$ using the relation~\eqref{phi psi/m}.
\EQ{
\phi(r; \xi) &=  \frac{\Im \psi( r; \xi)}{ \Im m( \xi)} \simeq  \xi^{-2} \Im [ \psi_E(r; \xi)(1+O( \xi^{-1} \ang{r}^{-1}))]\\
& = \xi^{-1} ( r^{1/2}J_1(r \xi) (1+O(  \ang{r}^{-1})) \\
}
Using asymptotic formulas for $J_1(z)$ we obtain the estimate
\EQ{
\xi^{\frac{3}{2}} \abs{ \phi(r; \xi)} \lesssim 1 \mas \xi  \to \infty
}
for all $r>0$. Together with~\eqref{om big xi} this establishes~\eqref{spec} for all $\xi \ge \Xi_0$ for some $\Xi_0>0$ large. This completes the proof of Lemma~\ref{dis ft}. 
\end{proof}

%%%%%%%%%%%%%%%%%%%%%%%%%%%%%%%%%%%%%%%%%%%%%%%%%%%%%%%%%%%%%%%%%%%%%%%%%%%%%%%%%%
%%%%%%%%%%%%%%%%%%%%%%%%%%%%%%%%%%%%%%%%%%%%%%%%%%%%%%%%%%%%%%%%%%%%%%%%%%%%%%%%%%

\section{Proofs of Theorem~\ref{s2 stab} and Theorem~\ref{h2 stab}}\label{proofs}

In this final section we prove the two asymptotic stability results, Theorem~\ref{s2 stab} and Theorem~\ref{h2 stab}. We remark that due to Lemma~\ref{2d to 4d} it will suffice to work in the $4d$ setting and consider the Cauchy problems~\eqref{u eq} for the case of maps into $\Sp^2$ and~\eqref{u eq h} for the case of maps into~$\Hp^2$.

As we will be working with initial data $\vec u(0) = (u_0, u_1)$ that is small in the energy space $H^1 \times L^2(\Hp^4)$, we will also require only rough estimates on the size of the nonlinearities $\NN_{\Sp^2}$ and $\NN_{\Hp^2}$. In particular, we will make use of the following simple lemmas.

\begin{lem}\label{NN} Let $\vec \psi \in \E_{\la}$ and define $u$ by $\sinh r \, u(r) =  \psi(r)- Q_{\la}(r)$. Let $\NN_{\Sp^2}(r, u)$ be defined as in~\eqref{N S}, i.e.,
\ant{
\NN_{\Sp^2}(r, u)&:= \frac{\sin 2Q_{\la}}{ \sinh^3 r}\sin^2( 2 \sinh r\, u) + \cos 2Q_{\la} \frac{2 \sinh r\, u -  \sin (2 \sinh r \, u)}{2 \sinh^3 r}\\
&=: \F_{\Sp^2}(r, u) + \G_{\Sp^2}(r, u).
}
Then, there exist  constants $C_1 = C(\la)>0$ and $C_2>0$  so that
\EQ{\label{FG}
&\abs{\F_{\Sp^2}(r, u)} \le C_1 \ang{ \sinh r}^{-1} \abs{u}^2\\
& \abs{\G_{\Sp^2}(r, u} \le C_2 \abs{u}^3
}
\end{lem}
\begin{proof} To prove the first estimate in~\eqref{FG} we consider two regions, $r \le 1$ and $r \ge 1$. When $r \ge 1$ we use the estimate,
\ant{
\abs{\F_{\Sp^2}(r, u)} \le C  \sinh^{-1}r \abs{u}^2
}
Next, recall that  $Q_{\la}(r) = 2 \arctan(\la \tanh(r/2))$.  Then
\ant{
\sin 2 Q_{\la} %&= 2 \sin Q_{\la} \cos Q_\la  =4 \sin (Q_\la/2) \cos (Q/2 )\cos (Q_{\la})\\
&= 4 \tan(Q_\la/2) \cos^2 (Q_\lambda/2 )\cos (Q_{\la})\\
& = 4 \la \tanh (r/2)\cos^2 (Q_\lambda/2 )\cos (Q_{\la})
}
Therefore, for $r \le 1$ we can conclude that
\ant{
\abs{\F_{\Sp^2}(r, u)} \le C  \la \tanh(r/2) \sinh^{-1}r \abs{u}^2 \le C(\la) \abs{u}^2
}
which proves the first estimate in~\eqref{FG}. The second estimate in~\eqref{FG} is clear.
\end{proof}

\begin{lem}\label{NN1} Let $(\psi, 0) \in \EE_{\la}$ and define $u$ by $\sinh r\,  u(r) = \psi(r) - P_{\lmb}(r)$. Suppose that
\EQ{\label{psi-P A}
\| (\psi, 0) - (P_{\la}, 0) \|_{\HH_0} \le A
}
Let $\NN_{\Hp^2}(r, u)$ be defined as in~\eqref{N H}, i.e.,
\ant{
\NN_{\Hp^2}(r, u)&:= -\frac{\sinh 2P_{\la}}{ \sinh^3 r}\sinh^2( 2 \sinh r\, u) + \cosh2P_{\la}\frac{2 \sinh r\, u -  \sinh (2 \sinh r \, u)}{2 \sinh^3 r}\\
&=: \F_{\Hp^2}(r, u) + \G_{\Hp^2}(r, u).
}
Then, there exists a constant $K = K(\la, A)>0$  so that
\EQ{\label{FG1}
&\abs{\F_{\Hp^2}(r, u)} \le K \ang{ \sinh r}^{-1} \abs{u}^2,\\
& \abs{\G_{\Hp^2}(r, u)} \le K \abs{u}^3.
}
\end{lem}
\begin{proof} First we note that the estimate~\eqref{psi-P A} gives us an a priori estimate on the size of $\psi$. Indeed, since $\psi(0) - P_{\la}(0) = 0$ we have
\ant{
 (\psi(\rho) - P_{\la}(\rho))^2  &= \int_0^\rho \p_r(\psi(r) - P_{\la}(r))^2 \, dr  \\
 &=2 \int_0^\rho (\psi_r(r) - \p_rP_{\la}(r))( \psi(r)- P_{\la}(r)) \, dr \\
 & \le  \| (\psi, 0) - (P_{\la}, 0)\|_{\HH_0}^2
}
This means that
\EQ{\label{sup A}
 \sup_{r \in [0, \infty)}\abs{\sinh ru(r)}  =  \sup_{r \in [0, \infty)} \abs{ \psi(r) - P_{\lmb}(r)} \le A.  
}
Next, note that
\begin{align*}
	\sinh 4 \vartht
	=& 2 \sinh 2 \vartht \cosh 2 \vartht \\
	=& 4 \sinh \vartht \cosh \vartht (\cosh^{2} \vartht + \sinh^{2} \vartht) \\
	=& 4 \frac{\tanh \vartht (1+\tanh^{2} \vartht)}{(1-\tanh^{2} \vartht)^{2}}.
\end{align*}
Plugging in $\vartht = \arctanh ( \lmb \tanh(r/2))$, we see that
\EQ{\label{sin P}
\sinh (2 P_{\lmb}(r))
&=  4 \left( \frac{ \lmb \tanh (r/2) (1+\lmb^{2} \tanh^{2}(r/2)}{(1-\lmb^{2} \tanh^{2}(r/2))^{2}} ) \right)\\
&  \le 4 \la\tanh(r/2)\left(\frac{1+ \la^2}{1- \la^2}\right)\\
& \le C(\la)\la\tanh(r/2)
}
In light of~\eqref{sup A} and~\eqref{sin P} we can find $K= K(\la, A)$ so that
\ant{
\abs{\F_{\Hp^2}(r, u)}  = \abs{  \frac{\sinh 2P_{\la}}{ \sinh^3 r}\sinh^2( 2 \sinh r\, u) } \le  K(\la, A) \frac{\tanh(r/2)}{ \sinh^3 r} \sinh^2 r\abs{u}^2
}
and this proves the first estimate in~\eqref{FG1}. To prove the second estimate we define
\ant{
Z( \rho) := 4\frac{\rho -  \sinh (\rho)}{ \rho^3}
}
Then,
\ant{
\G_{\Hp^2}(r, u) = \cosh 2P_{\la} Z( \sinh r u) u^3
}
By~\eqref{sup A}, and the fact that $\cosh{2P_{\la}} \le \cosh(4 \arctanh{ \la})$,  we can find $K= K(\la, A)$ so that $\abs{\cosh 2P_{\la} Z( \sinh r\,  u(r))} \le K(\la, A)$, which completes the proof of the second estimate in~\eqref{FG1}.
\end{proof}

In light, of Lemma~\ref{NN} and Lemma~\ref{NN1} we can handle the small data theory for~\eqref{u eq} and~\eqref{u eq h} simultaneously. Indeed, consider the more general Cauchy problem
\EQ{\label{cp}
&u_{tt}-  u_{rr} - 3 \coth r \, u_r -2 u+ V u = \F(r, u) + G(r, u)\\
&\vec u(0) = (u_0, u_1) \in H^1 \times L^2( \Hp^4)
}
where $V = V(r)$ is a radial potential satisfying assumptions $(A)$ and $(B)$ as defined in the beginning of Section~\ref{strich section} so that Proposition~\ref{strich} applies. 
We assume that for $\vec u$ and $A > 0$ such that
\begin{equation} \label{eq:apriori4u}
	\nrm{\vec u}_{H^{1} \times L^{2}} \leq A,
\end{equation}
 the nonlinearities $\F, \G$ satisfy
\EQ{
&\abs{\F(r, u)} \lesssim_{A} \ang{\sinh r}^{-1} \abs{u}^2,	\\
& \abs{\G(r, u)} \lesssim_{A} \abs{u}^3.
}

We can now formulate the local well-posedness theory for~\eqref{cp}. For a time interval $0 \in I \subset \R$ , define the norms $S(I)$ and $N(I)$ by
\EQ{
&\|u\|_{S(I)} := \|u\|_{L^3_t(I; L^{6} (\Hp^4))}\\
&\|F\|_{N(I)} := \|F\|_{L^1_t(I; L^{2} (\Hp^4)) + L^{\frac{3}{2}}_t(I; L^{\frac{12}{7}}(\Hp^4))}
}
\begin{prop}\label{small data}
Let $\vec u(0) = (u_0, u_1) \in H^1 \times L^2 (\Hp^4)$ be radial. Then there is a unique  solution $\vec u(t) \in H^1 \times L^2(\Hp^4)$ to~\eqref{cp} defined on a maximal interval of existence $0 \in I_{\max}( \vec u) = (-T_-, T_+),$ and for any compact interval $J \subset I_{\max}$ we have

\ant{
\| u\|_{S(J)} + \nrm{\vec{u}}_{L^{\infty}_{t}(J; H^{1} \times L^{2})} < \infty.
}	
Moreover, a globally defined solution $\vec u(t)$ to ~\eqref{cp} for $t\in [0, \infty)$ scatters as $ t \to \infty$ to a free shifted wave, i.e., a solution $\vec u_L(t) \in H^1 \times L^2(\Hp^4)$ of
\EQ{\label{v free}
 v_{tt} - v_{rr} - 3 \coth r \, v_r -2v =0,
}
if and only if
\begin{equation*}
	\| u \|_{S([0, \infty))} + \nrm{\vec{u}}_{L^{\infty}_{t}([0, \infty); H^{1} \times L^{2})} < \infty.
\end{equation*}

Here scattering as $t \to \infty$ means that
\EQ{
\|\vec u(t) - \vec u_L(t) \|_{H^1 \times L^2(\Hp^4)} \to 0 \mas t \to \infty.
}
In particular, there exists a constant $\de>0$ so that
\EQ{ \label{global small}
 \| \vec u(0) \|_{H^1 \times L^2} < \de \Rightarrow
 	\| u\|_{S(\R)} + \nrm{\vec{u}}_{L^{\infty}_{t}(\bbR; H^{1} \times L^{2})} \lesssim \|\vec u(0) \|_{H^1\times L^2} \lesssim \de
 }	
and hence $\vec u(t)$  scatters to free waves as $t \to \pm \infty$.
\end{prop}
\begin{rem}
In our applications of this proposition to the equivariant wave map equation, we note that an {\it a priori} $L^{\infty}_{t} (H^{1} \times L^{2})$ bound holds thanks to conservation of energy. In particular, the criterion for scattering as $t \to \infty$ is simply $$\nrm{u}_{S([0, \infty))} < \infty,$$ where $u$ is defined as in~\eqref{u S2 def} or~\eqref{u H2 def} depending on the target $M$ and the parameter~$\lmb$.
\end{rem}

\begin{proof}
The proof of Proposition~\ref{small data} follows from the usual contraction mapping argument based on the Strichartz estimates in Proposition~\ref{strich} and we give a brief sketch as the details are standard. Indeed,
suppose that the bootstrap assumption
\begin{equation} \label{eq:btstrp}
	\nrm{\vec{u}}_{L^{\infty}_{t}(I; H^{1} \times L^{2})} \leq A
\end{equation}
holds for some $A > 0$. Then applying Proposition~\ref{strich} % with $(p, q, \gamma) = (3, 6, 1)$ and $(a, b, \sigma) = ( \infty, 2, 0)$ and
 to any time interval $I$ we have
\begin{align*}
& \| u \|_{S(I)} + \| \vec u(t)\|_{L^{\infty}(I; H^{1} \times L^{2})} \\
& \quad \lesssim \|\vec u (0) \|_{H^1 \times L^2} + \| \F( \cdot, u) + \G(\cdot, u) \|_{N(I)}\nonumber\\
& \quad \lesssim \|\vec u (0) \|_{H^1 \times L^2} + C_{A} (\| \ang{\sinh r}^{-1} u^2\|_{L^{\frac{3}{2}}_t(I; L^{\frac{12}{7}}_x)} + \| u^3\|_{L^1_t(I; L^2_x)})	\\
& \quad \lesssim \|\vec u (0) \|_{H^1 \times L^2} + C_{A} (\|\ang{\sinh r}^{-1}\|_{L^{\infty}_t L^4_x}\|u^2\|_{L^{\frac{3}{2}}_t L^3_x} + \|u\|_{S(I)}^3)	\\
& \quad \lesssim \|\vec u (0) \|_{H^1 \times L^2} + C_{A} (\|u\|_{S(I)}^2 + \|u\|_{S(I)}^3).
\end{align*}		
By the usual continuity argument, (expanding $I$), this implies the a priori estimate~\eqref{global small} for small data. The scattering is also standard and based on a global Strichartz estimate. Indeed, if we denote by $S_{-2+V}(t)$ the propagator of the free shifted wave equation on $\Hp^4$ perturbed by the radial potential $V(r)$, i.e., the propagator for~\eqref{lin wave V} with $F=0$,  we seek initial data $\vec v_L(0) \in H^1 \times L^2$ so that
\ant{
\vec u(t) = S_{-2+V}(t) \vec v_L(0) + o_{H^1 \times L^2}(1) \mas t \to \infty
}
In view of the Duhamel representation for $\vec u(t)$ and using the group property and unitarity of $S_{-2+V}$ this is tantamount to
\ant{
\vec v_L(0) =  \vec u(0) + \int_0^{\infty} S_{-2+V}(-s)(0, (\F+\G)( \cdot, u)(s)) \, ds
}
The integral on the right-hand side above is absolutely convergent in $H^1 \times L^2$ as long as  $\|u\|_{S([0, \infty))} + \nrm{\vec{u}}_{L^{\infty}([0, \infty); H^{1} \times L^{2})}< \infty$. That the finiteness of these norms of $u$ is a necessary condition is due the fact that free shifted waves satisfy it, whence by the small data theory (applied to large times) it carries over to any nonlinear wave that scatters. Now that we have found the linear wave $\vec v_{L}(t)  = S_{-2 +V}(t) \vec v_L(0)$ which approaches $\vec u(t)$ in the energy space we can easily pass to a solution to~\eqref{v free} with the same property. %Replacing $\vec v_L(0)$ with scattering data  $\vec u_{L}(0)$ for the free shifted wave~\eqref{v free} then follows from an application of the dispersive estimates~\eqref{CK}.
Indeed we can define the scattering data
\EQ{\label{scat data}
\vec u_{L}(0) =  \vec v_L(0) + \int_0^\infty S_{-2}(-s)(0, V v_L)(s) \, ds
}
where $S_{-2}$ is the free propagator for~\eqref{v free}. The fact that~\eqref{scat data} is in $H^1 \times L^2$ follows from~\eqref{CK} with the space $X=L^{\infty}_tH^1_x$.
\end{proof}

Finally, we remark that in light of the reductions to the $4d$ equations,~\eqref{u eq} and~\eqref{u eq h}, as well as the estimates in Lemma~\ref{NN} and Lemma~\ref{NN1} we have also completed the proofs of Theorem~\ref{s2 stab} and Theorem~\ref{h2 stab}.

\bibliographystyle{plain}
\bibliography{researchbib}

 \bigskip

\centerline{\scshape Andrew Lawrie, Sung-Jin Oh}
\smallskip
{\footnotesize
% please put the address of the first author
 \centerline{Department of Mathematics, The University of California, Berkeley}
\centerline{970 Evans Hall \#3840, Berkeley, CA 94720, U.S.A.}
\centerline{\email{ alawrie@math.berkeley.edu, sjoh@math.berkeley.edu}}
} % Do not forget to end the {\footnotesize by the sign }

 \medskip

\centerline{\scshape Sohrab Shahshahani}
\medskip
{\footnotesize
% please put the address of the first author
 \centerline{Department of Mathematics, The University of Michigan}
\centerline{2074 East Hall, 530 Church Street
Ann Arbor, MI  48109-1043, U.S.A.}
\centerline{\email{shahshah@umich.edu}}
} % Do not forget to end the {\footnotesize by the sign }

\end{document}